\documentclass[12pt, reqno]{article}

\RequirePackage{amsthm,amsmath,amsfonts,amssymb}
\RequirePackage[numbers]{natbib}
\RequirePackage{graphicx}

\usepackage{breakcites}

\usepackage{lipsum}
\usepackage{xargs}
\usepackage[pdftex,dvipsnames]{xcolor}
\usepackage{mathrsfs}

\usepackage{bibspacing}
\makeatletter
\newcommand{\subjclass}[2][2010]{%
  \let\@oldtitle\@title%
  \gdef\@title{\@oldtitle\footnotetext{#1 \emph{Mathematics subject classification.} #2}}%
}
\newcommand{\keywords}[1]{%
  \let\@@oldtitle\@title%
  \gdef\@title{\@@oldtitle\footnotetext{\emph{Key words and phrases.} #1.}}%
}
\makeatother

\usepackage{stmaryrd}
\usepackage{dsfont}
\usepackage{mathtools}
\usepackage{bbm}
\usepackage{bm}

\usepackage{booktabs}
\usepackage{pgfplots}
\usetikzlibrary{backgrounds,calc,positioning,arrows}

\usepackage{amsmath}
\usepackage{cases}
\usepackage[english]{babel}
\usepackage{amsthm}
\usepackage{exscale}
\usepackage{amsfonts}
\usepackage{amssymb}
\usepackage{fancyhdr}

\usepackage[latin1]{inputenc}
\usepackage{mathrsfs}
\usepackage{fancyhdr}
\usepackage[misc,geometry]{ifsym}
\usepackage{color}
\usepackage{dsfont}
\usepackage{esint}
\usepackage{hyperref}
\usepackage{amsmath,cases}

\usepackage{upgreek}
\usepackage{yfonts}
\usepackage[artemisia]{textgreek}

\numberwithin{equation}{section}

\newtheorem{theorem}{Theorem}[section]
\newtheorem{lem}[theorem]{Lemma}
\newtheorem{defi}[theorem]{Definition}

\newtheorem{pro}[theorem]{Proposition}
\newtheorem{lm}[theorem]{Lemma}
\newtheorem{coro}[theorem]{Corollary}
\newtheorem{remark}[theorem]{Remark}

\newtheorem{assumptions}[theorem]{Assumptions}

\def\eps{\varepsilon}

\def\rd{\mathbb{R}^d}

\def\naturals{\mathbb{N}}

\def\reals{\mathbb{R}}
\def\BoRd{\mathscr{B}(\rd)}
\def\ps{\Theta}
\def\reals{\mathbb{R}}
\def\probabilityspace{(\Omega, \mathscr{A}, \textsf{P})}
\def\ss{\mathbb{X}}
\def\samplespace{\mathbb{X}}
\def\ssa{\mathscr{X}}

\def\samplespace{\mathbb{X}}

\DeclareMathOperator*{\esssup}{ess\,sup}
\DeclareMathOperator*{\lesssup}{\lambda-ess\,sup}

\newcommand{\mr}{%
  \,\raisebox{-.127ex}{\reflectbox{\rotatebox[origin=br]{-90}{$\lnot$}}}\,%
}

\newcommand{\X}{\mathbb{X}}
\newcommand{\V}{\mathbb{V}}
\newcommand{\Z}{\mathbb{Z}}

\newcommand{\R}{\mathbb{R}}

\newcommand{\N}{\mathds{N}}

\newcommand{\Bcr}{\mathscr{B}}

\newcommand{\Wuno}{\mathcal{W}_1}
\newcommand{\Wdue}{\mathcal{W}_2}
\newcommand{\Wp}{\mathcal{W}_p}

\def\rd{\mathbb{R}^d}
\def\naturals{\mathbb{N}}
\def\reals{\mathbb{R}}

\newcommand{\psf}{\mathscr{T}}

\newcommand{\empiric}{\mathfrak{e}_n}

\newcommand{\ee}{\textsf{E}}
\newcommand{\pp}{\textsf{P}}

\newcommand{\ud}{\mathrm{d}}

\newcommand{\ind}{\mathds{1}}

\def\KKK{\color{black}}

\newenvironment{proofad}{\removelastskip\par\medskip
\noindent{\textbf {Proof of Theorem \ref{th:0}-{(i)}}.}
\rm}{\penalty-20\null\hfill$\square$\par\medbreak} 

\newenvironment{proofad0}{\removelastskip\par\medskip
\noindent{\textbf {Proof of Theorem \ref{th:0}-{(iii)}}.}
\rm}{\penalty-20\null\hfill$\square$\par\medbreak} 

\newenvironment{proofad2}{\removelastskip\par\medskip
\noindent{\textbf {Proof of Theorem \ref{maintrace}}.}
\rm}{\penalty-20\null\hfill$\square$\par\medbreak} 

\newenvironment{proofadS}{\removelastskip\par\medskip
\noindent{\textbf {Proof of Theorem \ref{th:0}-{(iv)}}.}
\rm}{\penalty-20\null\hfill$\square$\par\medbreak}

\newenvironment{proofadinfinity}{\removelastskip\par\medskip
\noindent{\textbf {Proof of Theorem \ref{infinity}}.}
\rm}{\penalty-20\null\hfill$\square$\par\medbreak} 

\newenvironment{proofadw1}{\removelastskip\par\medskip
\noindent{\textbf {Proof of Theorem \ref{th:0}-{(ii)}}.}
\rm}{\penalty-20\null\hfill$\square$\par\medbreak} 

\newenvironment{proofadconcrete}{\removelastskip\par\medskip
\noindent{\textbf {Proof of Theorem \ref{concretemixed}}.}
\rm}{\penalty-20\null\hfill$\square$\par\medbreak} 

\newenvironment{proofadfrancesi}{\removelastskip\par\medskip
\noindent{\textbf {Proof of Proposition \ref{francesi}}.}
\rm}{\penalty-20\null\hfill$\square$\par\medbreak} 

\usepackage{a4wide}
\author{Emanuele Dolera and Edoardo Mainini}

\title{\textbf{Lipschitz continuity of probability kernels in the optimal transport framework}}

\subjclass{60F10, 62F15}

 \keywords{Bayes formula, Bayesian consistency, Benamou-Brenier formula, Probability Kernel,
 Optimal transport, Wasserstein distance, weighted Sobolev spaces}

\date{}

\begin{document}
 \maketitle

\begin{center}
\end{center}
%

%
%
%
%
%
%

\begin{abstract}
In Bayesian statistics, a continuity property of the posterior distribution with respect to the observable variable is crucial as it expresses well-posedness, i.e., stability with respect to errors in the measurement of data. Essentially, this requires analyzing the continuity of a probability kernel or, equivalently, of a conditional probability distribution with respect to the conditioning variable.

Here, we tackle this problem from a theoretical point of view. Let $(\X, \ud_{\X})$ be a metric space, and let $\BoRd$ denote the Borel $\sigma$-algebra on $\rd$. 
Let $\pi(\cdot|\cdot) : \BoRd\times\ss \to [0,1]$ be a dominated probability kernel, i.e. of the form $\pi(\ud\theta|x)=g(x,\theta)\pi(\ud\theta)$ for some suitable function $g:\X\times\rd\to[0,+\infty)$. 
We provide general conditions ensuring the Lipschitz continuity of the mapping $\X \ni x \mapsto \pi(\cdot|x) \in \mathcal P(\rd)$ when the space of probability measures 
$\mathcal P(\rd)$ on $(\rd, \BoRd)$ is endowed with a metric arising within the optimal transport framework, such as a Wasserstein metric. In particular, 
we prove explicit upper bounds for the Lipschitz constant
in terms of Fisher-information functionals and weighted Poincar\'e constants, obtained by exploiting the dynamic formulation of the optimal transport.

Finally, we give some illustrations on noteworthy classes of probability kernels, and we apply the main results to improve on some open questions in Bayesian statistics, dealing with
the approximation of posterior distributions by mixtures and posterior consistency.
\end{abstract}

%
%
%
%
%
%
%
%


\tableofcontents

\newcommand{\Ti}{\mbox{\large $\boldsymbol\vartheta$}}
\newcommand{\Tis}{\mbox{\footnotesize $\boldsymbol\vartheta$}}

\newcommand{\Tir}{\mbox{\large $\boldsymbol\Uptheta$}}
\newcommand{\Tin}{\mbox{\footnotesize $\boldsymbol\Uptheta$}}

\newcommand{\Tft}{\mbox{\large\textfrak{S}}}
\newcommand{\dft}{d_{\mbox{\tiny\textfrak{S}}}}

\section{Introduction}

\subsection{Formulation of the problem and main contributions} \label{sect:formulation_main}

Several problems in probability and statistics involve  mappings of the form $x \mapsto \pi(\cdot | x)$, where $\pi(\cdot | \cdot) : \BoRd \times \X \rightarrow [0,1]$ is a \emph{probability kernel}. 
In general, $\X$ is a metric space endowed with distance $\ud_{\X}$ and Borel $\sigma$-algebra $\ssa$, while $\BoRd$ stands for the usual Borel $\sigma$-algebra on $\rd$.
Being any probability kernel $\pi(\cdot | \cdot)$ conceivable as a mapping from $\X$ into the space $\mathcal P(\rd)$ of all probability measures (p.m.'s) on $(\rd, \BoRd)$, our main goal is to provide general conditions for getting a global form of Lipschitz continuity, namely
\begin{equation} \label{main_problem}
\ud_{\mathcal P(\rd)}\big(\pi(\cdot | x_1), \pi(\cdot | x_2)\big) \leq L\, \ud_{\X}(x_1, x_2)\ \ \ \ \ \ \ \forall\ x_1, x_2 \in \X,
\end{equation}
where $\ud_{\mathcal P(\rd)}$ is a suitable distance on $\mathcal P(\rd)$ and $L \geq 0$. 
Of course, the problem is strongly influenced by the choice of the distance $\ud_{\mathcal P(\rd)}$ which, at least at the level of abstract theory, can be done in several ways. See, e.g., the review \cite{GibbsSu}. Here, we will focus only on the $p$-\emph{Wasserstein distance} $\Wp$, $p \geq 1$, and the \emph{total variational distance} $\ud_{TV}$ because of their mathematical tractability, their clever conception as minimal transport distances, and their relationships with other probability metrics. 
For the sake of clarity, we recall that
\begin{align*}
\ud_{TV}(\mu,\nu) &:= \sup_{A \in \BoRd} |\mu(A) - \nu(A)| \qquad\quad\qquad\qquad\qquad \forall \mu, \nu \in \mathcal P(\rd) \\
\Wp(\mu, \nu) &:= \inf_{\eta \in \mathcal{F}(\mu,\nu)} \left(\int_{\reals^{2d}} |\theta_1 - \theta_2|^p\ \eta(\ud \theta_1 \ud \theta_2) \right)^{1/p}\quad\quad \forall \mu, \nu \in \mathcal P_p(\rd),
\end{align*}
where $\mathcal P_p(\rd):= \left\{\zeta \in \mathcal P(\rd) : \int_{\rd} |\theta|^p \zeta(\ud\theta) < +\infty \right\}$ 
and $\mathcal{F}(\mu,\nu)$ denotes the class of all p.m.'s on $(\reals^{2d}, \mathscr{B}(\reals^{2d}))$ with first marginal $\mu$ and second marginal $\nu$. See, e.g., \cite{AGS,vilMass,V2} 
for further information about the $p$-Wasserstein distance. \\

 Our main results are concerned with dominated kernels of the form
\begin{equation}  \label{kernelgpi}
\pi(B|x) := \int_B g(x, \theta) \pi(\ud\theta) \quad\quad \forall\ B \in \BoRd,\quad \forall\ x \in \X\ ,
\end{equation}
with some measurable, non-negative function $g$ and some measure $\pi$ on $(\rd, \BoRd)$. In this setting, we provide novel contributions in different directions. 
First, we formulate  a general theory aimed at solving \eqref{main_problem}, with emphasis on  estimates for the Lipschitz constant $L$. See Theorem \ref{th:0} and its extensions in Section  \ref{sect:otherNew}. Second, we illustrate the new methods on some well-known classes of probability kernels, such as exponential families 
(see Subsection \ref{sect:exponential}) and certain truncation families (see Subsection \ref{wilfrido}). 
 Third, we show the usefulness of estimate \eqref{main_problem} for the solution of other allied questions, mainly of statistical nature (see Section \ref{sect:application}). 
We emphasize the following strength points of our theory: the generality of the kernels under consideration, which are not constrained to belong to specific classes; the  estimates for the constant $L$ given in terms of some well-known functionals involving $g$ and $\pi$; the focus on the 2-Wasserstein distance, for which we will take advantage of the dynamic formulation, recalled in Subsection \ref{sect:dynamic}; 
the inclusion of the non-standard case of kernels with a support that varies with $x$ (see Subsection \ref{sect:mainresults_2}).


\subsection{Basic motivations from probability and Bayesian statistics} \label{sect:literature}

A  basic motivation for the analysis of a property like \eqref{main_problem} comes from the theory of (regular) conditional distributions and its applications.
In fact, probability kernels arise naturally in connection with the {\it disintegration problem}, within the abstract measure-theoretic formulation due to Kolmogorov.   
See, e.g., Theorems 6.3 and 6.4 in \cite{ka}, and Chapters 1-5 of \cite{Rao} for an overview. For clarity, we recall the notion of disintegration, with the same notation of Subsection \ref{sect:formulation_main}: given a random vector $(X,Z)$ on a probability space $\probabilityspace$ with values in $\X\times\rd$, we say that a probability kernel
$\pi(\cdot | \cdot) : \BoRd \times \X \rightarrow [0,1]$ solves the disintegration problem if $\ee[\pi(B|X)\ind_A(X)] = \pp[X \in A, Z \in B]$ holds for any $B \in \BoRd$ and any $A\in \ssa$, where
$\ind_A$ denotes the indicator function. The well-known issue of non-uniqueness of solutions to the disintegration problem (in the sense that if $\pi_1(\cdot | \cdot)$ is a solution, then 
$\pi_2(\cdot | \cdot)$ is also a solution as soon as $\pp[\pi_1(\cdot |X) \neq \pi_2(\cdot |X)] = 0)$ introduces a remarkable gap between theory and practice, since it entails that
conditional probabilities of the form $\pp[\cdot | X = x]$ are in general meaningless for a single $x\in \X$ such that $\pp[X = x] = 0$. See the discussion about the so-called \textit{Borel paradox} in \cite{Rao}. 
However, the necessity of pointwise evaluations usually emerges in Bayesian inference (see \cite{CP} and the reference therein), statistical mechanics (see e.g. \cite{lanford}) and theory of stochastic processes (see e.g. \cite{LaGatta}), where $x$ stands for some really observed datum and the observer would like to evaluate a conditional probability exactly at $x$. 
This foundational mismatch could be overcome by introducing suitable additional conditions that grant
uniqueness in the disintegration problem: in fact, we recall that, if (the distribution of) $X$ has full support in $\X$, then there exists at most one probability kernel $\pi(\cdot | \cdot)$ satisfying both the disintegration and the property that $\X \ni x \mapsto \pi(\cdot|x) \in \mathcal P(\rd)$ is continuous with respect to the topology of weak convergence on $\mathcal P(\rd)$. 
The existence of such a continuous representative, under additional conditions on the joint distribution of $(X,Z)$, was first analyzed in \cite{zab}. 
See also Chapter 9 of \cite{tjur}. In this respect, a stronger form of continuity like \eqref{main_problem} 
expresses a quantitative stability of conditional distributions  with respect to small deviations of the observed point, in analogy with the classical notion of well-posedness introduced by Hadamard.
However, a general formalization seems still lacking, and deserves deeper investigations. \\

A specific situation of interest arises in Bayesian statistics in the case that the joint distribution of $(X,Z)$ turns out to be absolutely continuous with respect to a product measure, say 
$\lambda \otimes \pi$, on $(\X \times \rd, \ssa\otimes\BoRd)$, with density $f : \X \times \rd \to [0,+\infty)$. When a (jointly) continuous density $f$  is assigned as starting point of the analysis, the conditional
distribution of $Z$ given $X= x$ emerges more naturally from the well-known \emph{Bayes formula}, rather than a disintegration. Precisely, $\pp[Z \in \cdot | X = x]$ is given 
by a kernel of the form \eqref{kernelgpi} with
\begin{equation} \label{Bayes}
g(x, \theta) = \dfrac{f(x, \theta)}{\int_{\rd} f(x, \tau) \pi(\ud\tau)}
\end{equation}
for any $x \in \X$ such that $\int_{\rd} f(x, \tau) \pi(\ud\tau) > 0$. In this framework, very basic results aimed at proving a local form of \eqref{main_problem} are contained in our recent paper \cite{DM}, 
which is confined to the choice of the total variation distance. In the present paper, we will improve on the results of \cite{DM} by relaxing the regularity assumptions, by providing global Lipschitz continuity, and most importantly by considering  the Wasserstein distance. 
Concerning other quantitative estimates like \eqref{main_problem}, the literature is relatively scant. A fairly general approach can be found in the work \cite{S} by A.M. Stuart, 
who minted the expression \emph{Bayesian well-posedness} for a local version of \eqref{main_problem}. See Subsection 4.2 of \cite{S}. See also the discussion about well-posedness in \cite{LATZ}. 

Another strong motivation from Bayesian inference is the following. Let us consider again the evaluation of the conditional probability $\pp[Z\in \cdot|X= x]$. 
Besides disposing of a specific datum $x\in\X$, we assume the presence of some noise in the process of observation. This leads us to interpret $x$ as a realization of
$\varphi_{\varepsilon}(X)$ rather than of $X$ itself, where $\varphi_{\varepsilon} : \X \to \X$ is some random perturbation of the identity map, stochastically independent of $(X,Z)$.
If we dispose of some \textit{apriori} bound
(pointwise or in the mean) on the discrepancy between $\varphi_{\varepsilon}$ and the identity map, we could exploit a property like \eqref{main_problem} to get a bound on the discrepancy between 
the conditional distributions $\pp[Z\in \cdot|X= x]$ and $\pp[Z\in \cdot|\varphi_{\varepsilon}(X)= x]$. That is, \eqref{main_problem} highlights the impact of the perturbation of the data in inference.
This remark is of some relevance in the recent studies on \textit{differential privacy}. See, e.g., \cite{BernShel,KJKH}. \\

\subsection{Further motivations and applications} 

We present a short list of problems that further motivate our analysis and represent the main applications of our theory. We shall provide new explicit solutions to such problems in Section \ref{sect:application}, by stressing the key role of \eqref{main_problem}. We also mention some related works in the literature, that often make use a property like of \eqref{main_problem} only as a technical tool. 
\begin{itemize}
\item[a)] \textit{Bayesian well-posedness}. In the same spirit of \cite{S}, by Bayesian well-posedness we mean
the validity of a local version of \eqref{main_problem} along with \eqref{kernelgpi}-\eqref{Bayes}. This notion have been investigated in the context of Bayesian inverse problems 
in \cite{S,CDRS,DS,ILS,LATZ,SULL,TRILL}. Due to their specific focus,
these papers only deal with kernels arising from linear regression problems, which are in exponential form. 
In Subsections \ref{sect:exponential} and \ref{sect:nobservations} we also analyze Bayesian well-posedness with exponential kernels
and, by applying our main results from Section \ref{section:mainresults}, we provide new estimates for the Lipschitz constant. 
We also consider the customary situation of an inference process with multiple exchangeable observations. Other new results on Bayesian well-posedness will be given in 
Subsection \ref{wilfrido} where we analyze Pareto-like statistical models. 
\item[b)] \textit{Approximation of posterior distributions by mixtures}. This problem arises in Bayesian inference when the posterior is not expressible in closed form. To carry out the inferential procedures,
a possible strategy is to approximate the prior by a mixture of conjugate prior (conjugation being referred to the statistical model), leading to an approximated posterior which is again in the form of a mixture.
Here, property \eqref{main_problem} yields a bound for the error in approximating the posterior, besides a more precise characterization of the posterior weights. See Subsection \ref{sect:renyi}.   
See also \cite{West} for developments in parametric settings, \cite{resa} for the nonparametric approach, and \cite{scricciolo} for density estimation. In particular, 
Proposition 2 of \cite{resa} is an evident application of \eqref{main_problem}. 
\item[c)] \textit{Bayesian consistency}. The foundational topic of frequency validation of Bayesian procedures (see \cite{diafree1} and \cite[Chapter 6]{GV}) can be rewritten as an approximation 
problem between posterior distributions. See Subsection \ref{sect:consistency} along with our recent contributions \cite{CDFM,Dolera,DFM}, where \eqref{main_problem} is at the core of the main argument. 
\end{itemize}
Finally, we foresee a number of other interesting applications that, for reason of space, are not developed in this paper. Thus, we just mention the field of: 
\textit{Bayesian robustness} (see \cite{MiDu}); \textit{Bayesian deconvolution and empirical Bayes methods} (see \cite{EfronEB}); \textit{theory of computability} (see \cite{AFR}).
Hopefully, a general theory of Bayesian well-posedness could bring novel contributions also to these fields.

\section{Main Results} \label{section:mainresults}

\subsection{Lipschitz estimates in terms of $\ud_{TV}$, $\mathcal W_1$ and $\mathcal W_2$} \label{sect:mainresults_1}

The results of this subsection are concerned with kernels of the form \eqref{kernelgpi} which fulfill the following
\begin{assumptions} \label{ass:1}\rm
Let $\pi$ be a p.m. on $(\rd, \BoRd)$ such that $\text{supp}(\pi) = \overline\Theta$ for some (nonempty) connected open set $\Theta \subseteq \rd$, and $\pi(\partial\Theta)=0$. 
Let $\X$ be a convex open subset of $\R^m$, endowed with the reference $\sigma$-algebra $\ssa$ of all Lebesgue-measurable subsets of $\ss$.
Finally, let the function $g$ be an element of $L^1_{\mathcal L^m\otimes\pi}(\X \times \ps)$ with $\int_{\ps} g(x,\theta)\pi(\ud\theta)=1$ for all $x \in \ss$, where $\mathcal L^m$ denotes the $m$-dimensional Lebesgue measure.
\end{assumptions}
In the main theorem, we will also assume that $g\in L^1_\pi(\ps; W^{1,1}_{loc}(\X))$, meaning that the distributional gradient of the mapping $x\mapsto g(x,\theta)$  (denoted by $\nabla_x$) belongs to $L^1_{loc}(\X)$ for $\pi$-a.e. $\theta\in \ps$ and that $\nabla_x g\in L^1_{\mathcal L^m\otimes\pi}(\tilde{\X}\times\ps)$ for any open set $\tilde{\X}$ compactly contained in $\X$.
In such a case, if $g(x,\cdot)>0$ for $\pi$-a.e. $\theta\in\ps$,  we define the {\it Fisher functional} of $g$ relative to $\pi$ as
\begin{equation} \label{fff}
\mbox{$\;\mathcal{J}_{\pi}[g(x,\cdot)] := \displaystyle\left(\int_\ps \frac{|\nabla_x g(x,\theta)|^2}{g(x,\theta)}\, \pi(\ud\theta)\right)^{\frac12}.$}
\end{equation}

Another key assumption  of the theory will be the validity of the so-called \emph{weighted Poincar\'e-Wirtinger inequalities}. We say that a Radon measure $\mu$  on $(\ps, \psf)$ satisfies a weighted Poincar\'e-Wirtinger inequality of order 
$q \in [1,+\infty)$ if a constant $\mathcal C_q$ exists such that 
\begin{equation}\label{wirtinger}
\inf_{a\in\R}\left(\int_\ps |\psi(\theta)-a|^q\,\mu(\ud\theta)\right)^{\frac1q} \le \mathcal C_q\left(\int_\ps |\nabla \psi(\theta)|^q\,\mu(\ud\theta)\right)^{\frac1q}
\end{equation}
holds for every  $\psi\in C^1_c(\overline\ps)$. Here, $\psi\in C^1_c(\overline\ps)$ means that $\psi$ is the restriction to $\overline\ps$ of a $C^1$ compactly supported function on $\rd$.
We denote by $\mathcal C_q[\mu]$ the best constant in such inequality and we put $\mathcal C[\mu]:=\mathcal C_2[\mu]$. Further details are contained in Subsection \ref{sectionpoincare}. 

We also  consider (unweighted) Sobolev-Poincar\'e inequalities. Let either $1\le p<d$ or $1=p=d$. We say that $\ps$ satisfies a Sobolev-Poincar\'e inequality of order $p$ if a constant 
$\mathcal S_{p}$ exists such that
\begin{equation}\label{sobineq}
\inf_{a\in\R}\|\zeta-a\|_{L^{p^*}(\ps)}\le \mathcal S_{p}\,\|\nabla \zeta\|_{L^p(\Theta)}\qquad \mbox{for any $\zeta\in C^1_c(\overline\ps)$},
\end{equation}
where $p^*=\tfrac{dp}{d-p}$ if $1\le p<d$ and $p^*=+\infty$ if $p=d=1$. We denote by $\mathcal S_{p}(\ps)$ the corresponding best constant.

\begin{theorem} \label{th:0}
For a given kernel $\pi(\cdot|\cdot)$ in the form \eqref{kernelgpi}, let Assumptions \ref{ass:1} be in force. Let also $g\in L^1_\pi(\Theta; W^{1,1}_{loc}(\ss))$. 
When $\Wp$ is involved, assume further that $\int_\ps |\theta|^p\,\pi(\ud\theta|x) <+\infty$ for 
$\mathcal L^m$-a.e. $x \in \X$.
The following statements hold.  
\begin{itemize}
\item[\bf{(i)}] Suppose that 
$$
K:=\displaystyle\esssup_{x\in\X}\,\|\nabla_x g(x,\cdot)\|_{L^1_\pi(\ps)}<+\infty.
$$  
Then, there exists a $\ud_{TV}$-Lipschitz version of $\pi(\cdot|\cdot)$, satisfying \eqref{main_problem} with $L=K/2$. 
\item[\bf(ii)] For $1<p\le+\infty$ and $q=\tfrac p{p-1}$, suppose that 
$$
K:=\pi(\Theta)^{1/q}\,\displaystyle \mathcal C_{q}[\pi]\,\esssup_{x\in\X} \|\nabla_x g(x,\cdot)\|_{L^p_\pi(\ps)} <+\infty.
$$   
Then, there exists a $\Wuno$-Lipschitz version of $\pi(\cdot|\cdot)$, satisfying \eqref{main_problem} with $L=K$.  
\item[\bf(iii)]  If $g>0$ $(\mathcal L^m\otimes\pi)$-a.e. and 
$$
K:=\displaystyle\esssup_{x \in \X }\;\mathcal{C}[g(x,\cdot)\,\pi]\,\mathcal{J}_{\pi}[g(x,\cdot)] <+\infty,
$$ 
then there exists a $\mathcal W_2$-Lipschitz version of $\pi(\cdot|\cdot)$, satisfying \eqref{main_problem} with $L=K$.  
\item[\bf(iv)] 
If $g>0$   $(\mathcal L^m\otimes\pi)$-a.e., $\pi=\mathcal L^d\mr\Theta$ and
$$ 
K:=\displaystyle \mathcal S_{p}(\Theta)\,\esssup_{x\in\X} \left\|\frac1{g(x,\cdot)}\right\|_{L^{\frac p{2-p}}(\ps)}^{1/2}\left\|\nabla_x g(x,\cdot)\right\|_{L^{\frac r{r-1}}(\ps)}<+\infty,
$$
then there exists a $\mathcal W_2$-Lipschitz version of $\pi(\cdot|\cdot)$, satisfying \eqref{main_problem} with $L=K$.  
\end{itemize}
\end{theorem}

We notice that the assumption $g>0$ $(\mathcal L^m\otimes\pi)$-a.e., made in points {\bf (iii)}-{\bf (iv)}, entails that support of the p.m. $\pi(\cdot|x)$ coincides with $\overline\Theta$ for 
$\mathcal L^m$-a.e. $x$ in $\ss$. We will refer to this fact in the sequel by saying we are in the case of {\it fixed domains}. 
In Subsection \ref{sect:mainresults_2} we will provide other new results that generalize points {\bf (iii)}-{\bf (iv)} to the situation of {\it moving domains}, meaning that the 
support of $\pi(\cdot|x)$ is allowed to vary smoothly with $x$.
Concerning the first two points of Theorem \ref{th:0}, we remark that the assumptions of point {\bf (ii)} imply the ones of point {\bf (i)}, which is formally the limit case $p=1$ of point {\bf (ii)}. On the other hand,
if $\Theta$ is bounded, the elementary inequality $\mathcal W_1\le 2\mathrm{diam}(\Theta) \ud_{TV}$ allows to deduce an estimate for the $\mathcal W_1$ distance under the assumptions of point {\bf (i)}. 
Moreover, we notice that point {\bf (iv)} holds whenever $\ps$ is a domain for which the Sobolev-Poincar\'e inequality \eqref{sobineq} is satisfied. Therefore, 
point {\bf (iv)} applies for instance if $\ps$ is the whole of $\rd$ 
(and $\mathcal S_{p}(\rd)$ is explicit, see \cite{Au, Talenti}) or if $\Theta$ is a $W^{1,p}$ extension domain with $\mathcal L^d(\ps)<+\infty$ (see, e.g., \cite[Chapter 12]{L}). More generally, it applies if $\ps$ is a John domain (see \cite{BK, CWZ, H1}), including the half space and domains with compact Lipschitz boundary. If $d=1$, \eqref{sobineq} holds on any interval $\ps\subseteq\R$ with $\mathcal S_1(\ps)=1$.

We conclude with a brief discussion about the best constant in the weighted Poincar\'e-Wirtinger inequality  \eqref{wirtinger}.
The most classical Poincar\'e inequalities hold  by taking $\mu$ to be the $d$-dimensional  Lebesgue measure on a bounded domain $\Theta$ with Lipschitz boundary and $q=2$: the reciprocal square of  $\mathcal C[\mathcal L^d\mr\Theta]$ is the first nontrivial eigenvalue of the  Neumann Laplacian on $\Theta$. 
If $\mu$ is the $d$-dimensional Lebesgue measure on a bounded convex set $\Theta\subset\rd$, the classical result by Payne and Weinberger \cite{PW} shows that $\mathcal C[\mathcal L^d\mr\Theta]$ is proportional to the diameter of $\Theta$, see also \cite{AD, ENT, bebendorf} for  $q\neq 2$. Explicit estimates  for star-shaped domains are found in \cite{FR}.
 According to the Bakry-Emery condition in the Euclidean setting (see \cite{BE}), if $\mu$ is a p.m. on a convex set $\Theta\subseteq\rd$   and  $V\in C^2(\Theta)$ exists such that 
\begin{equation}\label{be}
\mu(\ud\theta)=e^{-V(\theta)}\,\ud\theta,\;\; \mbox{$\qquad\qquad\langle\mathrm{Hess}[V]\xi,\xi\rangle\ge\alpha>0\;\;$ in $\Theta\;\quad$  for any $\xi\in\rd$},
\end{equation}
then  \eqref{wirtinger} holds with   $\mathcal C[\mu]\le 1/\sqrt{\alpha}$. See for instance \cite{MS} or \cite[Chapitre 5]{ABC}, see also \cite{chafai, MJT}. 
On the other hand, it is shown in \cite{FNT} that for any log-concave measure $\mu$ on a bounded convex domain $\Theta$ of $\rd$ (i.e., for any convex $V$), the constant $\mathcal C[\mu]$ can be bounded explicitly by 
$\mathrm{diam}(\Theta)/\pi$. Therefore, the Poincar\'e best constant can be improved by the presence of a log-concave weight with $\alpha>0$  (the unweighted case corresponding here to $V=0$). Let us also mention the result by Bobkov \cite{B} which allows to estimate the Poincar\'e constant of a log-concave measure $\mu$ on $\rd$ in terms of the variance, i.e.,
$\mathcal C[\mu]\le 12\sqrt{3} \left(\int_{\rd}|\theta-\overline\mu|^2\,\mu(\ud\theta)\right)^{1/2}$ with 
$\overline\mu:=\int_{\rd}\theta\,\mu(\ud\theta)$.
The fundamental Bakry-Emery citerion admits other generalizations. For instance, \eqref{wirtinger} holds on $\rd$ if the condition $\frac12|\nabla V(\theta)|^2-\Delta V(\theta)\ge c>0$ is satisfied for any large enough 
$|\theta|$, see for instance \cite{BBCG, BCG}. Different results are also available for measures of the form $\mu=e^{-V}\nu$, where $\nu$ itself satisfies \eqref{wirtinger}, the most simple instance being the Holley-Stroock \cite{HS} perturbation principle $\mathcal C^2[\mu]\le\exp\{\sup V-\inf V\}\,\mathcal C^2[\nu]$. See e.g. \cite[Th\'eor\`eme 3.4.1]{ABC} or \cite{MS}. 
Further statements in this direction are contained in \cite[Proposition 4.1]{C}, where it is assumed that $\mu$ satisfies the stronger log-Sobolev inequality. 

\subsection{A remarkable example: Exponential models} \label{sect:exponential}

A useful rephrasing of the main results from Theorem \ref{th:0} can be obtained, in Bayesian statistical inference, 
by considering a statistical model in the following form:
\begin{equation}\label{fgphi}
f(x|\theta)=e^{\Phi(x,\theta)}h(x),\qquad g(x,\theta)=\frac{f(x|\theta)}{\rho(x)}=\frac{e^{\Phi(x,\theta)}}{\int_\Theta e^{\Phi(x,\tau)}\pi(\ud\tau)}
\end{equation}
for some measurable functions $h : \ss \to (0,+\infty)$ and $\Phi: \ss \times \Theta \to \reals$. Here,  $\pi$ denotes the prior probability measure and $\rho(x)= h(x)\int_\Theta e^{\Phi(x,\theta)}\,\pi(\ud\theta)>0$ for any $x\in\ss$. The function $g$ in \eqref{fgphi} is therefore obtained by applying the Bayes formula. Under the formulation \eqref{fgphi}, the Fisher functional $\mathcal J_{\pi}[g]$ defined in \eqref{fff} can be formally rewritten as 
\begin{equation}\label{newJ}\begin{aligned}
&\mathcal J_{\pi}[g(z,\cdot)] = \left(\int_\Theta |\Psi(z,\theta)|^2 g(z,\theta)\,\pi(\ud\theta)\right)^{\frac12},\\&\quad\mbox{where}\quad \Psi(x,\theta):=\nabla_x\Phi(x,\theta) - \int_\Theta\nabla_x \Phi(x,\tau)\,g(x,\tau)\,\pi(\ud\tau). \end{aligned}
\end{equation}
It is worth noticing that the mapping $\theta \mapsto \Psi(x,\theta)$ satisfies the null-mean property, i.e. $\int_\Theta\Psi(x,\theta)\,g(x,\theta)\,\pi(\ud\theta)=0$ for any $x\in\ss$, which 
allows a further application of the Poincar\'e inequality \eqref{wirtinger}. Therefore, Theorem \ref{th:0}-{\bf (iii)} can be revisited as follows.
\begin{coro}\label{coro:poincare} 
Let $\Phi\in C^1(\mathbb X\times\Theta)$ be such that $\theta \mapsto \nabla _x\Phi(x,\theta)$ is Lipschitz, for any $x\in \X$. Given a positive and measurable function
$h:\X\to\R$, let $f,g,\Psi$ be defined by \eqref{fgphi} and \eqref{newJ}. If $\int_\ps |\theta|^2 g(x,\theta)\,\pi(\ud\theta)<+\infty$ holds for every $x\in\X$ and
\begin{equation}\label{productassumption2}
K:=\esssup_{x\in\X}\; \left(\mathcal{C}[g(x,\cdot)\,\pi)]\right)^2\,\left(\int_\ps |\nabla_\theta \Psi(x,\theta)|^2\,g(x,\theta)\,\pi(\ud\theta)\right)^{\frac12} <+\infty, 
\end{equation}
then the probability kernel $\pi(\cdot|\cdot)$, defined by \eqref{kernelgpi} and \eqref{fgphi}, satisfies \eqref{main_problem} with $\ud_{\mathcal P(\rd)} = \Wdue$ and $L=K$.
\end{coro}

\begin{proof} 
The regularity  of $\Phi$ ensures that $\nabla_x\int_\Theta e^{\Phi(x,\theta)}\,\pi(\ud\theta)=\int_\Theta e^{\Phi(x,\theta)} \nabla_x\Phi(x,\theta)\,\pi(\ud\theta)$, and thanks to this property a computation immediately shows that \eqref{newJ} holds for every $x\in\X$.
As already mentioned, $\int_\Theta\Psi(x,\theta)\,g(x,\theta)\,\pi(\ud\theta)=0$ for any $x\in\ss$, thus an application of the Poincar\'e inequality \eqref{wirtinger} yields 
\[
\mathcal C[g(x,\cdot)\,\pi]\,\mathcal J_\pi[g(x,\cdot)]\le \left(\mathcal{C}[g(x,\cdot)\,\pi)]\right)^2\,\left(\int_\Theta |\nabla_\theta \Psi(x,\theta)|^2\,g(x,\theta)\,\pi(\ud\theta)\right)^{\frac12}
\]
for every $x\in\X$. The conclusion follows from Theorem \ref{th:0}-{\bf (iii)}.
\end{proof}

This corollary allows us to easily deal with statistical models $f(\cdot | \cdot)$ belonging to the well-known \emph{exponential family}. See, e.g., \cite{barndorff,brown} for a comprehensive treatment 
of the exponential family from the point of view of classical statistics, and \cite{diayilv} for a Bayesian approach. For the canonical exponential family, 
we consider two measurable functions $T : \ss \rightarrow \rd$ and $h : \ss \to (0,+\infty)$. Upon putting 
$
\Theta = \left\{\theta\in \rd\ :\ \int_{\ss} e^{T(x) \cdot \theta} h(x)\,\ud x <+\infty \right\},
$
the function $\Phi$ assumes the form 
\begin{equation} \label{Phi_exponential}
\Phi(x,\theta) = T(x) \cdot \theta - M(\theta),\qquad\mbox{with}\quad M(\theta) := \log \int_{\ss} e^{T(x) \cdot \theta} h(x)\,\ud x.
\end{equation}
In addition, we recall the standard \emph{regularity conditions} for the canonical exponential family: $\ps$ is a nonempty open subset of $\rd$ and the interior of the convex hull of the support of 
$h \circ T^{-1}$ is assumed to be nonempty. 
Under such conditions $\ps$ proves to be convex, while $M : \ps \to \R$ turns out to be strictly convex, analytic and steep (cf. Definition 3.2 of \cite{brown}). These considerations allows further
estimates on the Poincar\'e constant in \eqref{productassumption2}, according to the discussion of Subsection \ref{sect:mainresults_1}. Finally, if the function $T$ belongs to $C^1_b(\X; \rd)$, the integral term in \eqref{productassumption2} is formally re-written according to 
$$
\left(\int_\Theta |\nabla_\theta \Psi(x,\theta)|^2\,g(x,\theta)\,\pi(\ud\theta)\right)^{\frac12} =  |\nabla T(x)|\ . 
$$
In this setting, we can further refine Corollary \ref{coro:poincare}, thanks to the Bakry-Emery criterion \eqref{be}, by stating the following
\begin{pro} \label{pro_exponential}
Consider a statistical model from the exponential family with a Lipschitz continuous $T$, $h : \ss \to (0,+\infty)$, $\Theta$ and $M$ as above. 
Let $\pi(\ud\theta) = e^{-W(\theta)}\ud\theta$ with $W \in C^2(\ps)$. If $\mathrm{Hess}[M+W] \geq \alpha I$ on $\ps$ in the sense of quadratic forms for some $\alpha > 0$, and 
$
\int_{\Theta} |\theta|^2 \exp\{T(x) \cdot \theta - M(\theta) - W(\theta)\}\ud \theta < +\infty
$ 
for every $x\in \ss$, then the posterior distribution $\pi(\cdot|\cdot)$, defined by \eqref{kernelgpi} and \eqref{fgphi}, satisfies \eqref{main_problem} with distance $\Wdue$ and $K = \mathrm{Lip}(T)/\alpha$. 
\end{pro}

\begin{remark}\rm
Because of their frequent use in practical statistical context, the exponential family is often rewritten under different re-parametrizations, both of the parameter and the data. Of course, property \eqref{main_problem} depends crucially on the specific parametrization, and can fail after a re-parametrization. For example, the re-parametrization of the parameter in terms of the mean (see, for example, Chapter 3 of \cite{brown}) preserves the Lipschitz continuity if $\nabla M:\Theta\to\rd$ is itself Lipschitz. Apropos of the re-parametrization of the data, very often the sufficient statistics $T$ is 
itself viewed as the datum, which leads to a simpler problem. See Subsection \ref{sect:nobservations} below.
\end{remark}

\begin{remark}\rm
If $\Theta=\rd$ and $\alpha=0$ in Proposition \ref{pro_exponential},  an alternative estimate of the $\Wdue$-Lipschitz constant in \eqref{main_problem} is $L\le 12\sqrt{3}\,\mathrm{Lip}(T)\,\mathrm{Var}(e^{M+V})$, in view of an already recalled result by Bobkov \cite{B}. Further variants can be obtained by applying Proposition \ref{francesi} in the Appendix.
\end{remark}

\section{Applications} \label{sect:application}

\subsection{Statistical inference with $n$ exchangeable observations} \label{sect:nobservations}

In concrete statistical applications, it is customary to consider the observed datum $x$ as a vector $(x_1, \dots, x_n)$ containing the outcomes of $n$ experiments. 
Accordingly, the space $\ss$ mentioned in Subsection \ref{sect:formulation_main}
becomes a product space, say $\ss_1^n$.
In the Bayesian approach, the vector $(x_1, \dots, x_n)$ is viewed as the realization of some random vector, say $(X_1, \dots, X_n)$, and the core of the analysis hinges on the stochastic dependence between the components of this random vector. In particular, when the experiments are performed under ``ideally similar physical conditions'' the order in which the outcomes are collected becomes 
irrelevant. This intuitive, practical observation is captured by the notion of \emph{exchangeability}, introduced by B. de Finetti. See \cite{Ald(85)} for a comprehensive reference on exchangeability, and Section 2.12 of \cite{GDS} for a statistical perspective.

Here, we illustrate how to apply our theory of Lipschitz-continuous kernels within the field of statistical inference with $n$ exchangeable observations, lending our results a more statistical flavour and giving a deeper insight into the concept of ``Bayesian well-posedness''. First, we recall that a sequence $\{X_i\}_{i \geq 1}$ of $\ss_1$-valued random variables, defined on $\probabilityspace$, is exchangeable if the identity $\pp[X_1 \in A_1, \dots, X_n \in A_n] = \pp[X_1 \in A_{\sigma_n(1)}, \dots, X_n \in A_{\sigma_n(n)}]$ is fulfilled 
for any $n \in \naturals$, permutation $\sigma_n : \{1, \dots, n\} \rightarrow \{1, \dots, n\}$ and $A_1, \dots, A_n \in  \mathscr X_1$, where $\mathscr X_1$ is a $\sigma$-algebra on $\ss_1$. 
Under fairly general assumptions (e.g., when $\ss_1$ is a Polish metric space and $\mathscr X_1$ coincides with its Borel $\sigma$-algebra), de Finetti's representation theorem states that the law of the observations can be written as
$\pp[X_1 \in A_1, \dots, X_n \in A_n] = \int_{\mathbb T} \big[\prod_{i=1}^n \nu(A_i\ |\ \theta)\big] \pi(\ud\theta)$,
where $(\mathbb T, \mathscr T)$ is a suitable measurable space (the \emph{parameter space}), $\pi$ is a prior p.m. on $(\mathbb T, \mathscr T)$, and $\nu : \mathscr X_1 \times \mathbb T 
\rightarrow [0,1]$ is a kernel representing the statistical model for any single observation. If we suppose that the family $\{\nu(\cdot\ |\ \theta)\}_{\theta \in \mathbb T}$ of p.m.'s is dominated by some 
$\sigma$-finite measure $\lambda_1$ on $(\ss_1, \mathscr X_1)$, with relative density $f(\cdot | \theta)$, then, by resorting to the Bayes formula \eqref{Bayes}, the posterior distribution of the random parameter given the observations can be written as
\begin{equation} \label{posterior_exchangeable}
\pi_n(\ud\theta\ |\ x_1, \dots, x_n) := \frac{\left[\prod_{i=1}^n f(x_i | \theta)\right] \pi(\ud\theta)}{\int_{\mathbb T} \left[\prod_{i=1}^n f(x_i | \tau)\right] \pi(\ud\tau)}
\end{equation}
for any $(x_1, \dots, x_n) \in \ss_1^n$ such that $\int_{\mathbb T} \left[ \prod_{i=1}^n f(x_i | \tau) \right] \pi(\ud\tau) > 0$. Moreover, from a classical perspective, the product $\prod_{i=1}^n f(x_i | \theta)$, when viewed as a function of $\theta$, represents the \emph{likelihood function} $L_n(\theta; x_1, \dots, x_n)$. Hence, with a view to highlighting the role of our theory, we focus on the appealing situation in which there exists a \emph{classical sufficient statistics}. By the well-known Fisher-Neyman factorization criterion, we recall that a measurable mapping 
$\mathfrak{t}_n : (\ss_1^n, \mathscr X_1^n) \rightarrow (\mathbb{S}, \mathscr{S})$ is named a classical sufficient statistics whenever there exist a measurable space $(\mathbb{S}, \mathscr{S})$ and 
two measurable functions $\overline{g} : \mathbb{S} \times \mathbb T \rightarrow [0,+\infty)$ and $\overline{h} : \ss_1^n \rightarrow [0,+\infty)$ such that 
$L_n(\theta; x_1, \dots, x_n) = \overline{g}(\mathfrak{t}_n(x_1, \dots, x_n); \theta) \overline{h}(x_1, \dots, x_n)$
holds for every $(x_1, \dots, x_n) \in \ss_1^n$. We also notice that, in the exchangeable case, any classical sufficient statistics $\mathfrak{t}_n$ turns out to be a symmetric function of 
$x_1, \dots, x_n$. 
A remarkable example is obtain when $\ss_1$ is endowed with some metric structure and the mapping $x \mapsto f(x\ |\ \theta)$ is continuous and positive for every $\theta \in \mathbb T$.
In fact, $(\mathbb{S}, \mathscr{S})$ can be chosen as the space of all probability densities on $(\ss_1, \mathscr X_1, \lambda_1)$, endowed with the topology of weak (narrow) convergence and ensuing 
Borel $\sigma$-algebra $\mathscr{S}$, 
and $\mathfrak{t}_n(x_1, \dots, x_n)$ as the \emph{empirical measure} $\frac 1n \sum_{i=1}^n \delta_{x_i}$. In this case, \eqref{posterior_exchangeable} can be rewritten by replacing the product
$\prod_{i=1}^n f(x_i | \theta)$ with $\overline{g}(\mathfrak{t}_n(x_1, \dots, x_n); \theta)$
for any $(x_1, \dots, x_n) \in \ss_1^n$ such that $\int_{\mathbb T} \overline{g}(\mathfrak{t}_n(x_1, \dots, x_n); \tau) \pi(\ud\tau) > 0$. 
This identity is crucial to notice that, in the case of $n$ exchangeable observations,
it seems more natural to investigate the Lipschitz-continuity of the posterior distribution with respect to the variable $\mathfrak{t}_n$, rather than the original vector $(x_1, \dots, x_n)$. Thus, a natural reformulation of \eqref{main_problem} becomes
\begin{equation} \label{main_problem_exchangeable}
\ud_{\mathcal P(\mathbb T)}\left(\pi_n(\ud\theta\ |\ x_1, \dots, x_n), \pi_n(\ud\theta\ |\ y_1, \dots, y_n) \right) \leq K\ \ud_{\mathbb{S}}\left(\mathfrak{t}_n(x_1, \dots, x_n), \mathfrak{t}_n(y_1, \dots, y_n)\right)
\end{equation}
with some suitable distance $\ud_{\mathbb{S}}$ on $\mathbb{S}$. This reformulation is in harmony with the original assumption of exchangeability, since the RHS of 
\eqref{main_problem_exchangeable} is invariant after a permutation of the data $(x_1, \dots, x_n)$ or $(y_1, \dots, y_n)$, unlike the (product) distance between $(x_1, \dots, x_n)$ and $(y_1, \dots, y_n)$,
which is not preserved by permutation. As already noted in \cite[Section 2.3]{CDR} and \cite{CDRNota}, these considerations provide a new geometrical perspective on the basic formulation of Bayesian inference.

To illustrate the last consideration, we restrict to the case in which the above density $f(\cdot | \theta)$ has the exponential form as in \eqref{fgphi} and \eqref{Phi_exponential}.   
Thus, under the same standard regularity conditions for $\Theta$ of Subsection \ref{sect:exponential}, 
we can take $\mathbb T$ equal to $\overline{\Theta}$. In this framework, we have at our disposal the classical sufficient statistics
$\mathfrak{t}_n(x_1, \dots, x_n) = \frac 1n \sum_{i=1}^n T(x_i)$
which is an element of the interior $\Lambda$ of the convex hull of the support of $h \circ T^{-1}$. Indeed, we recall that $\nabla M : \Theta \rightarrow \Lambda$ is a smooth diffeomorphism
and $\hat{\theta}_n := (\nabla M)^{-1}(\mathfrak{t}_n(x_1, \dots, x_n))$ coincides with the maximum likelihood estimator (MLE). Thus, we will study the Lipschitz-continuity of the posterior distribution of the random parameter with respect to $\mathfrak{t}_n$ which, due to the recalled relation with the MLE, establishes an interesting link between Bayesian and classical statistics.
 
\begin{pro}
Consider a statistical model from the exponential family \eqref{fgphi}, with $T : \ss_1 \rightarrow \rd$, $h : \ss_1 \to (0,+\infty)$, $\Theta$ and $M$ as in {\rm Section \ref{sect:exponential}}. 
Let $\pi(\ud\theta) = e^{-W(\theta)}\ud\theta$ with $W \in C^2(\Theta)$. If $\mathrm{Hess}[M] \geq \alpha I$ and $\mathrm{Hess}[W] \geq \lambda_{\ast} I$ on $\Theta$
in the sense of quadratic forms, for some $\alpha > 0$ and $\lambda_{\ast} \in \mathbb R$, and
$$
\int_{\Theta} |\theta|^2 \exp\{n[\mathfrak{t}_n(x_1, \dots, x_n) \cdot \theta - M(\theta)] - W(\theta)\}\ud \theta < +\infty
$$ 
for every $n \in \mathbb N$ and $(x_1, \dots, x_n) \in \ss_1^n$, then the posterior $\pi_n(\cdot|\cdot)$ satisfies \eqref{main_problem_exchangeable}
for every $n \geq \max\{1, -\lambda_{\ast}/\alpha\}$, with $\mathbb T = \overline{\Theta}$, 
$\ud_{\mathcal P(\mathbb T)} = \Wdue$, $\ud_{\mathbb S}$ equal to the Euclidean distance on $\rd$, 
and $K = \frac{n}{n\alpha + \lambda_{\ast}}$. In addition, if $\nabla M$ is Lipschitz-continuous with constant $\ell$, then
$$
\Wdue\left(\pi_n(\ud\theta\ |\ x_1, \dots, x_n), \pi_n(\ud\theta\ |\ y_1, \dots, y_n) \right) \leq \frac{n \ell}{n\alpha + \lambda_{\ast}} |\hat{\theta}_n(x_1, \dots, x_n) - \hat{\theta}_n(y_1, \dots, y_n)|\ .
$$
holds for every $n \geq \max\{1, -\lambda_{\ast}/\alpha\}$ and $(x_1, \dots, x_n), (y_1, \dots, y_n) \in \ss_1^n$.
\end{pro}
\begin{proof} We
just notice that the function $\Phi$ of Section \ref{sect:exponential} becomes $\Phi(\mathfrak{t}_n,\theta) = n[\mathfrak{t}_n - M(\theta)]$, and apply Proposition \ref{pro_exponential}. 
\end{proof}

\subsection{Approximation of posterior distributions by mixtures} \label{sect:renyi}

This subsection is referred to the setting of Subsection \ref{sect:formulation_main}-\ref{sect:literature}, with the further assumption that $(\ss, \ud_{\ss})$ is totally bounded. 
A joint p.m. $\gamma$ is given on $(\X\times\rd, \ssa\otimes\BoRd)$, with first marginal $\chi$. The probability kernel $\pi(\cdot|\cdot) : \BoRd \times \X \rightarrow [0,1]$ is thought
of as a distinguished solution of the disintegration problem, that is $\int_A \pi(B|x) \chi(\ud x) = \gamma(A\times B)$ for any $A \in \ssa$ and $B \in \BoRd$. 
For simplicity, we assume that the support of $\chi$ coincides with the whole
of $\ss$. Now, we briefly describe an approximation procedure due to Renyi \cite{renyi}. See also \cite{pfa,resa} and references therein.  
Fix $\epsilon > 0$ arbitrarily. By total boundedness, there is a finite partition of $\ss$, denoted by $\{A_1, \dots, A_{k(\epsilon)}\}$, satisfying
\begin{enumerate}
\item[i)] $A_i \cap A_j = \emptyset$, for every $i,j \in \{1, \dots, k(\epsilon)\}$ with $i\neq j$
\item[ii)] $\cup_{j = 1}^{k(\epsilon)} A_j = \ss$
\item[iii)] $\chi(A_j) > 0$ for every $j \in \{1, \dots, k(\epsilon)\}$
\item[iv)] $\chi(\partial A_j) = 0$ for every $j \in \{1, \dots, k(\epsilon)\}$
\item[v)] $\text{diam}(A_j) \leq \epsilon$.
\end{enumerate}
The number $k(\epsilon)$ is usually referred to as the $\epsilon$-covering number of $(\ss, \ud_{\ss})$, and it is related to the dimension of $\ss$. 
We consider the following approximation of $\pi(\cdot|\cdot)$, given by
$$
\pi_{\epsilon}(B|x) := \sum_{j=1}^{k(\epsilon)} \frac{\gamma(A_j\times B)}{\chi(A_j)} \mathds{1}_{A_j}(x)
$$
for any $B \in \BoRd$ and $x \in \ss$. 
Finally, we endow the space $\mathcal P(\rd)$ of all p.m.'s on $(\rd, \BoRd)$ with the Borel $\sigma$-algebra $\mathscr P(\rd)$ originated by the weak convergence of p.m.'s. 
We have the following
\begin{pro} \label{prop:RS}
Let $\epsilon > 0$ and $\{A_1, \dots, A_{k(\epsilon)}\}$ be given as above.
Let $\ud_{\mathcal P(\rd)}$ be any distance which is convex and $\mathscr P(\rd) \otimes \mathscr P(\rd)\setminus \mathscr{B}([0,+\infty))$-measurable. 
Let the kernel $\pi(\cdot|\cdot)$ satisfy \eqref{main_problem} with such distance $\ud_{\mathcal P(\rd)}$. Then, 
\begin{equation} \label{KR_continuity}
\ud_{\mathcal P(\rd)}\left( \pi(\cdot|x), \pi_{\epsilon}(\cdot|x) \right) \leq L\epsilon\ , \qquad \forall\ x \in \ss.
\end{equation}  
\end{pro}

\begin{proof}
Fix $x \in \ss$. Then, $x \in A_{j(x)}$ for some $j(x) \in \{1, \dots, k(\epsilon)\}$ and
$\
\pi_{\epsilon}(\cdot|x)=\frac{1}{\chi(A_{j(x)})} \int_{A_{j(x)}}  \pi(\cdot|y) \chi(\ud y). 
$
Since $\pi(\cdot|x) = \frac{1}{\chi(A_{j(x)})} \int_{A_{j(x)}}  \pi(\cdot|x) \chi(\ud y)$, exploit the convexity of $\ud_{\mathcal P(\rd)}$ to obtain
$$
\ud_{\mathcal P(\rd)}\left( \pi(\cdot|x), \pi_{\epsilon}(\cdot\ |\ x) \right) \leq \frac{1}{\chi(A_{j(x)})} \int_{A_{j(x)}} \ud_{\mathcal P(\rd)}\left( \pi(\cdot|x), \pi(\cdot|y) \right) \chi(\ud y)\ .
$$
Combination of this last inequality with \eqref{main_problem} leads immediately to \eqref{KR_continuity}.
\end{proof}

The above proposition can be used to tackle the following question, which occurs very frequently in Bayesian inference. See \cite{resa} and \cite{Lijoi}
for formalizations within the Bayesian nonparametric setting and the parametric setting obtained by the classical exponential family, respectively. 
Let $\nu(\cdot | \cdot) : \ssa \times \rd \rightarrow [0,1]$
be a probability kernel representing the statistical model, not necessarily dominated. Given some prior $\pi$ on $(\rd, \BoRd)$, suppose that the posterior is not computable in a closed form,
so that very little can be said beyond its existence. This phenomenon usually happens in a semi-parametric or nonparametric setting. In any case, $\pi$ can be well approximated by mixtures
of the form $\sum_{j=1}^N \lambda_j \pi_j$, where $\pi_1, \dots, \pi_N$ are prior measures on $(\rd, \BoRd)$, usually belonging to some distinguished class, and $\lambda_1, 
\dots, \lambda_N \in [0,1]$ with $\sum_{j=1}^N \lambda_j = 1$. Now, assume that the posterior $\pi_j(\cdot | \cdot) : \BoRd \times \ss \rightarrow [0,1]$, relative to the prior $\pi_j$,
is actually computable in a closed form. Thus, it can be shown that the posterior $\pi_{\ast}(\cdot | \cdot) : \BoRd \times \ss \rightarrow [0,1]$, relative to
the prior $\sum_{j=1}^N \lambda_j \pi_j$, is equal to
$$
\pi_{\ast}(\cdot | x) = \lambda_j(x) \pi_j(\cdot | x)
\qquad\mbox{with }\quad
 \lambda_j(x) := \frac{\lambda_j \int_{\rd} f(x | \tau) \pi_j(\ud\tau)}{\sum_{i=1}^N \lambda_i \int_{\rd} f(x | \tau) \pi_i(\ud\tau)}. 
$$
Following \cite{resa,Lijoi}, we observe that the above Proposition \ref{prop:RS} can be used to compute the degree of approximation of the true posterior $\pi(\cdot|\cdot)$ by $\pi_{\ast}(\cdot | \cdot)$, uniformly 
with respect to the observed value $x$. For instance, our Proposition \ref{prop:RS} improves on Proposition 2 of \cite{resa} by providing an explicit rate of convergence.


\subsection{Bayesian Consistency} \label{sect:consistency}

In the problem of consistency, we start by considering a sequence of exchangeable observations, say $\{X_i\}_{i \geq 1}$, whose probability distribution is given by the identity
$\pp[X_1 \in A_1, \dots, X_n \in A_n] = \int_{\mathbb T} \big[\prod_{i=1}^n \nu(A_i\ |\ \theta)\big] \pi(\ud\theta)$, as explained in Subsection \ref{sect:nobservations}. In this subsection, 
we confine ourselves to case of real-valued $X_i$'s, so that $A_1, \dots, A_n \in \mathscr{B}(\reals)$, with reference measure 
$\lambda_1 = \mathcal{L}^1$. Moreover, we let $\ps$ be an open subset of $\rd$, and $\pi$ a p.m. with support equal to $\overline{\Theta}$ with $\pi(\partial \Theta)= 0$. Hence, the above space
$\mathbb T$ coincides with $\overline{\Theta}$. 
Lastly, we suppose that, for all $\theta \in \ps$, $\nu(\cdot| \theta)$ is absolutely continuous with respect to $\lambda_1$ with density
$f(\cdot|\cdot) > 0$, and that the mapping $x \mapsto f(x|\theta)$ is continuous.
In this framework, the posterior distribution is given by the Bayes formula \eqref{posterior_exchangeable}, while the likelihood can be written as $\exp\{ n\int_{\R} \log f(y|\theta) \empiric^x(\ud y)\}$
where $x = (x_1, \dots, x_n) \in \reals^n$ and $\empiric^x(\cdot) := \frac 1n \sum_{i=1}^n \delta_{x_i}(\cdot)$ denotes the \emph{empirical measure}. 
In the theory of Bayesian consistency, one fixes $\theta_0 \in \ps$ and generates a sequence $\{\xi_i\}_{i \geq 1}$ of i.i.d. random variables
from the p.m. $\nu(\cdot|\theta_0)$ given by the density $f(\cdot|\theta_0)$.
The objective is to prove that the posterior piles up near the true value $\theta_0$, i.e. that $\pi_n(U_0^c|\xi_1, \dots, \xi_n) \to 0$ as $n \rightarrow \infty$
for every neighborhood $U_0 \in \mathscr{B}(\ps)$ of $\theta_0$, where convergence is intended in probability. 
See \cite{diafree1} and \cite[Chapter 4]{GDS} for foundational motivations.
Now, with the help of the theory developed in this paper, we are able to provide a \emph{posterior contraction rate} at $\theta_0$, i.e. a sequence 
$\{\epsilon_n\}_{n \in\mathbb N}$ of positive numbers for which
\begin{equation} \label{PCR}
\pi_n\left(\{\theta \in \ps\ :\ |\theta - \theta_0| \geq M_n\epsilon_n\}\ \big|\ \xi_1, \dots, \xi_n\right)  \stackrel{\pp}{\longrightarrow} 0, \quad\quad \text{as}\ n \rightarrow \infty,
\end{equation}
holds for every diverging sequence $\{M_n\}_{n \geq 1}$ of positive numbers, where $\stackrel{\pp}{\longrightarrow}$ denotes convergence in probability. 
Cfr. Definition 8.1 in \cite{GV}. Now, we further assume that both $\nu(\cdot|\theta_0)$ and $\nu_1(\cdot)$
belong to $\mathcal{P}_1(\R)$, where $\nu_1(A) := \int_A \int_{\ps} f(x|\theta) \ud x \pi(\ud\theta)$. Thus, we can put
\begin{equation} \label{epsn}
\epsilon_n = \ee\left[\Wuno\left(\pi_n(\ud\theta|\xi_1, \dots, \xi_n); \delta_{\theta_0}\right)\right] 
\end{equation}
and notice that this choice actually provides a posterior contraction rate at $\theta_0$, highlighting the relevant role played by the Wasserstein distance in this theory. In fact, an application of the Markov 
inequality yields
$$
\pi_n\left(\{\theta \in \ps\ :\ |\theta - \theta_0| \geq M_n\epsilon_n\}\ \big|\ \xi_1, \dots, \xi_n\right) \leq \frac{1}{M_n\epsilon_n}\Wuno\left(\pi_n(\ud\theta|\xi_1, \dots, \xi_n); \delta_{\theta_0}\right)
$$
and the conclusion displayed in \eqref{PCR} follows by taking expectation of both sides of the above inequality, after recalling the suitable choice of $\epsilon_n$ made in \eqref{epsn}, 
Now, for any distribution function $F$ on $\R$, we introduce the probability kernel
\[\begin{aligned}
\pi_n^*(\ud\theta|F) :&=  \frac{\exp\{ n\int_{\reals} \log f(y|\theta) \ud F(y)\}}{\int_{\ps} \exp\{ n\int_{\reals} \log f(y | t) \ud F(y)\} \pi(\ud t)} \pi(\ud\theta) \\&=\frac{\exp\{ n \int_0^1 \log f(F^{-1}(u) | \theta) \ud u\}}{\int_{\ps} \exp\{ n\int_0^1 \log f(F^{-1}(u) | t) \ud u\} \pi(\ud t)} \pi(\ud\theta)
\end{aligned}\]
where, in the first line, integrals on $\R$ are intended in Riemann-Stieltjes sense, while, in the second line, $F^{-1}(u) := \inf\{y \in \reals\ |\ F(y) \geq u\}$. In this notation, we have
$\pi_n(\ud\theta | x) = \pi_n^*(\ud\theta | \hat{F}_n^x)$, where $\hat{F}_n^x(y) := \frac 1n \sum_{i=1}^n \ind_{[x_i, +\infty)}(y)$ denotes the empirical distribution function.
Thanks to the triangle inequality for the Wasserstein distance, we can provide the following useful bound for the expression of $\epsilon_n$ given in \eqref{epsn}, namely
\begin{equation*} \label{boundconsistency1}
\epsilon_n \leq \Wuno\left(\pi_n^*(\ud\theta | F_0); \delta_{\theta_0}\right) + \ee\left[\Wuno\left(\pi_n^*(\ud\theta | F_0); \pi_n^*(\ud\theta | \hat{F}_n^{\xi})\right)\right]
\end{equation*}
where $F_0(y) := \int_{-\infty}^y f(x | \theta_0) \ud x$ and $\hat{F}_n^{\xi}(y) := \frac 1n \sum_{i=1}^n \ind_{[\xi_i, +\infty)}(y)$. Apropos of the former term on the above RHS, we notice that
$\Wuno\left(\pi_n^*(\ud\theta | F_0); \delta_{\theta_0}\right) = \int_{\Theta} |\theta -\theta_0| \pi_n^*(\ud\theta | F_0)$.
Then, combining the definitions of \emph{Kullback-Leibler divergence} 
$K(\theta | \theta_0) := \int_{\reals} \log\left(\frac{f(y | \theta_0)}{f(y | \theta)}\right) f(y | \theta_0) \ud y$
with that of $\pi_n^*(\ud\theta | F_0)$, we can write
$$
\Wuno\left(\pi_n^*(\ud\theta | F_0); \delta_{\theta_0}\right) = \frac{\int_{\Theta} |\theta -\theta_0| e^{-n K(\theta | \theta_0)} \pi(\ud\theta)}{\int_{\Theta} e^{-n K(\theta | \theta_0)} \pi(\ud\theta)}\ . 
$$  
Here, we confine ourselves to dealing with \emph{regular models} (Cfr. \cite[Chapter 18]{Ferguson}), meaning that the Fisher information matrix $\mathrm{I}[\theta_0]$ at $\theta_0$, given by
$$
\mathrm{I}[\theta_0] := \left(-\int_{\reals} \left[\frac{\partial^2}{\partial \theta_i \partial \theta_i} f(x | \theta)\right]_{\theta = \theta_0} f(x | \theta_0)\ud x\right)_{ij}
$$ 
is strictly positive definite. Thus, with the quadratic form notation as in \eqref{be}, we have that
$K(\theta\ |\ \theta_0) = \frac 12\ \langle \mathrm{I}[\theta_0] (\theta - \theta_0),(\theta - \theta_0)\rangle + o(|\theta - \theta_0|^2)$ as $\theta \rightarrow \theta_0$, and that 
$\inf\{K(\theta | \theta_0)\ |\ \theta \in \Theta,\ |\theta - \theta_0| \geq \epsilon\} > 0$ for all sufficiently small $\epsilon > 0$. 
Now, an application of Theorem 41 in \cite{breitung} shows that
$$
\int_{\Theta} e^{-n K(\theta| \theta_0)} \pi(\ud\theta) \sim \left(\frac{2\pi}{n}\right)^{d/2} \frac{1}{\sqrt{\mathrm{I}[\theta_0]}}\ ,
$$
while Theorem 43 of the same reference gives
$$
\int_{\Theta} |\theta -\theta_0| e^{-n K(\theta | \theta_0)} \pi(\ud\theta) \sim \left(\frac{2}{n}\right)^{(d+1)/2} \frac 12\ \Gamma\left(\frac{d+1}{2}\right) 
\frac{\int_{\mathbb S^{d-1}} (\langle \mathrm{I}[\theta_0]^{-1} z,z\rangle )^{1/2} \ud \sigma(z)}{\sqrt{\mathrm{I}[\theta_0]}}
$$
where $\mathbb S^{d-1}$ stands for the surface of the ball of radius equal to 1 and centered at the origin of $\rd$. In conclusion, for regular models, we get 
$\Wuno\left(\pi_n^*(\ud\theta | F_0); \delta_{\theta_0}\right) \sim \frac{1}{\sqrt{n}}$ as $n \rightarrow +\infty$. 
At this stage, if we were able to show that the mapping $F \mapsto \pi_n^*(\ud\theta\ |\ F)$ is Lipschitz-continuous, in the sense that
\begin{equation}  \label{Lipschitz_consistency}
\Wuno\left(\pi_n^*(\ud\theta\ |\ F_1); \pi_n^*(\ud\theta\ |\ F_2)\right) \leq L(f,\pi)\ \Wdue(\mu_1; \mu_2) 
\end{equation}
with $\mu_i((-\infty, y]) = F_i(y)$ for $i=1,2$, for some constant $L(f,\pi) \geq 0$ independent of $n$, then we would conclude that
$$
\ee\left[\Wuno\left(\pi_n^*(\ud\theta\ |\ F_0); \pi_n^*(\ud\theta\ |\ \hat{F}_n^{\xi})\right)\right] \leq L(f,\pi) \ee\left[ \Wdue(\empiric^{\xi}; \nu(\cdot\ |\ \theta_0))\right]\ ,
$$ 
establishing in this way a very interesting connection. In fact, the term $\ee\left[ \Wdue(\empiric^{\xi}; \nu(\cdot\ |\ \theta_0))\right]$ is well-known in the probabilistic literature as \emph{speed of mean Glivenko-Cantelli convergence}, or \emph{monopartite matching problem}. See, for example, \cite{BoLe, DR, fourn}. In particular, for one-dimensional distributions, if $\nu(\cdot | \theta_0) \in \mathcal{P}_2(\R)$ satisfies also
\begin{equation} \label{BoLe}
\int_{\R} \frac{\nu((-\infty, x] | \theta_0) \nu((x, +\infty) | \theta_0)}{f(x|\theta_0)} \ud x < +\infty
\end{equation}
we have
$
\ee\left[ \Wdue(\empiric^{\xi}; \nu(\cdot| \theta_0))\right] \sim \frac{1}{\sqrt{n}}
$
as $n \rightarrow +\infty$, which again represent the optimal rate. Cfr. \cite[Theorem 5.1]{BoLe}.

To prove \eqref{Lipschitz_consistency}, we bring the theory developed in Section 2 into the game. We start from a well-known identity by Dall'Aglio, according to which $$\Wdue(\mu_1, \mu_2) = \|F_1^{-1} - F_2^{-1}\|_{\mathrm L^2(0,1)} := \left(\int_0^1 
|F_1^{-1}(u) - F_2^{-1}(u)|^2 \ud u\right)^{1/2}.$$ Thanks to this fact, we can apply point {\bf (ii)} or {\bf (iii)} of Theorem \ref{th:0}---or, more precisely, their infinite-dimensional reformulations,
stated as point {\bf (ii)} or {\bf (iii)} of Theorem \ref{infinity} below, with $\V = \mathrm L^2(0,1)$, 
\begin{equation} \label{Xconsistency}
\X = \{H : (0,1) \rightarrow \R\ |\ H(u) = \inf\{y \in \R\ |\ \mu((-\infty, y]) \geq u\}\ \text{for\ some}\ \mu \in \mathcal{P}_2(\R)\} 
\end{equation}
and
$$
g_n(H, \theta) := \frac{\exp\{ n \int_0^1 \log f(H(u)\ |\ \theta) \ud u\}}{\int_{\ps} \exp\{ n\int_0^1 \log f(H(u)\ |\ t) \ud u\} \pi(\ud t)} = \frac{e^{n \Phi(H,\theta)}}{\int_{\ps} e^{n \Phi(H,t)} \pi(\ud t)},
\quad\quad H \in \ss,
$$
where $\Phi(H,\theta) := \int_0^1 \log f(H(u)\ |\ \theta) \ud u$. Indeed, we notice that $\pi_n^*(\ud\theta\ |\ F) = g_n(F^{-1}, \theta) \pi(\ud\theta)$ for any distribution function $F$ with $F^{-1} \in L^2(0,1)$. We show an explicit solution based on Theorem \ref{infinity}-{\bf (iii)}. The evaluation of the Fisher functional starts from the evaluation of the Gateaux derivative of the mapping 
$H \mapsto g_n(H, \theta)$, namely 
\begin{align*}
\nabla_H g_n(H, \theta) &= n \frac{\nabla_H \Phi(H,\theta) e^{n \Phi(H,\theta)} \left(\int_{\ps} e^{n \Phi(H,t)} \pi(\ud t) \right) - e^{n \Phi(H,\theta)} \left(\int_{\ps} \nabla_H \Phi(H,t) e^{n \Phi(H,t)} \pi(\ud t) \right) }{\left(\int_{\ps} e^{n \Phi(H,t)} \pi(\ud t) \right)^2} \\
&= n g_n(H, \theta) \left[\nabla_H \Phi(H,\theta) - \int_{\ps} \nabla_H \Phi(H,t) g_n(H, t) \pi(\ud t) \right]\ .
\end{align*}
This computation yields
$$
\mathcal{J}_{\pi}[g(H, \cdot)] = n \left(\int_{\ps} \left\|\nabla_H \Phi(H,\theta) - \int_{\ps} \nabla_H \Phi(H,t) g_n(\cdot, t) \pi(\ud t) \right\|^2_{\mathrm L^2(0,1)} g_n(H, \theta) \pi(\ud \theta)\right)^{1/2}\ .
$$
Moreover, we notice that $\langle \nabla_H \Phi(H,\theta), \Psi\rangle_{\mathrm L^2(0,1)} = \int_0^1 \frac{\partial_x f(H(u)\  |\ \theta)}{f(H(u)\ |\ \theta)} \Psi(u) \ud u$
which, by resorting once again to the Poincar\'e inequality, entails
\begin{align*}
& \left(\int_{\ps} \left\|\nabla_H \Phi(H,\theta) - \int_{\ps} \nabla_H \Phi(H,t) g_n(\cdot, t) \pi(\ud t) \right\|^2_{\mathrm L^2(0,1)} g_n(H, \theta) \pi(\ud \theta)\right)^{1/2} \\
&\qquad \leq \mathcal{C}[g_n(H, \theta) \pi(\ud \theta)] \left(\int_{\ps} \left\|\nabla_H \frac{\partial_x f(H(\cdot)\  |\ \theta)}{f(H(\cdot)\ |\ \theta)}\right\|^2_{\mathrm L^2(0,1)} g_n(H, \theta) \pi(\ud \theta)\right)^{1/2}\ .
\end{align*}
We assume that the following scaling estimate holds 
\begin{equation} \label{consistencyPoincare}
\mathcal{C}^2[g_n(H, \cdot) \pi(\cdot)] \leq \frac{\tilde{C}(H; f, \pi)}{n}
\end{equation}
where $\tilde{C}(H; f, \pi)$ is a constant independent of $n$. 
Finally, we define
$$
\mathcal{E}(H; f, \pi) := \left( \sup_{n \in \naturals}\int_{\ps} \left\|\nabla_H \frac{\partial_x f(H(\cdot)\  |\ \theta)}{f(H(\cdot)\ |\ \theta)}\right\|^2_{\mathrm L^2(0,1)} g_n(H, \theta) \pi(\ud \theta)\right)^{1/2}
$$
and $L(f;\pi) := \sup_{H \in L^2(0,1)} \tilde{C}(H; f, \pi) \mathcal{E}(H; f, \pi)$. We can now condense this line of reasoning in the following
\begin{theorem} \label{thm:consistency}
Suppose that:
\begin{itemize} 
\item[i)] $f(x | \theta) > 0$ for all $(x,\theta) \in \R \times \ps$ and $x \mapsto f(x\ |\ \theta) \in C^2(\R)$ for all $\theta \in \ps$; 
\item[ii)] $\nu_1(\cdot) \in \mathcal{P}_2(\R)$, where $\nu_1(A) := \displaystyle\int_A \int_{\ps} f(x|\theta) \ud x \pi(\ud\theta)$;
\item[iii)] for fixed $\theta_0 \in \ps$,  $\{f(\cdot | \theta)\}_{\theta \in \ps}$ defines a $C^2$-regular model at $\theta_0$, as stated, e.g., in \cite[Chapter 18]{Ferguson};
\item[iv)] $\nu(\cdot | \theta_0) \in \mathcal{P}_2(\R)$ satisfies \eqref{BoLe};
\item[v)]  $\nabla_H \dfrac{\partial_x f(H(\cdot)  | \theta_0)}{f(H(\cdot) | \theta_0)}\;\in L^2(0,1)$ for any $H \in \X$, where $\X$ is defined by \eqref{Xconsistency};
\item[vi)] there exists $\tilde C(H;f,\pi)$ such that Poincar\'e constant of the posterior satisfies the bound \eqref{consistencyPoincare} for any $n\in\N$;
\item[vii)] $L(f;\pi) < +\infty$. 
\end{itemize} 
Then the posterior is consistent at $\theta_0$, with the optimal posterior rate $1/\sqrt{n}$. 
\end{theorem}
The validity of the estimate \eqref{consistencyPoincare}, which is here an assumption, is natural under suitable conditions like the ones in Proposition \ref{francesi} in the Appendix. The above assumptions 
vi)-vii) are therefore a rephrasing of the assumption in Theorem \ref{infinity}-{\bf (iii)},
which is stated later in Section \ref{sect:otherNew} (as a generalization of Theorem \ref{th:0}-{\bf (iii)}) 
and  can be invoked for proving Theorem \ref{thm:consistency}. 
For extensions and sharpening of this approach to Bayesian consistency and of Theorem \ref{thm:consistency}, including rigorous proofs, we refer to the recent contribution \cite{DFM}, where we  also show novel  applications.

\section{Some extensions and other new results} \label{sect:otherNew}


\subsection{Wasserstein distance: the PDE approach} \label{sect:dynamic}

Here, we briefly describe the techniques we shall exploit when considering the $2$-Wasserstein distance, in order to establish \eqref{main_problem} under  Assumptions \ref{ass:1}
for a probability kernel of the form \eqref{kernelgpi} and such that $\pi(\cdot|x)$ has finite second moment for any $x \in \X$. Indeed, it will be convenient to take advantage of the following dynamical formulation and to resort to the ensuing PDE approach.  

Letting $C^\infty_c(\overline \ps)$ denote the space of restrictions to $\overline\ps$ of $C^\infty_c(\rd)$ functions,
the dynamic formulation of the $2$-Wasserstein distance is based on the \emph{continuity equation}
\begin{equation}\label{basiccontinuity}
\frac{\ud}{\ud t}\int_{\overline\ps}\psi(\theta)\,\mu_t(\ud\theta)=\int_{\overline\ps}\left\langle\nabla \psi(\theta), \mathbf w_t(\theta) \right\rangle\,\mu_t(\ud\theta)\qquad\forall\psi\in C^\infty_c(\overline\ps),
\end{equation}
where $[0,1]\ni t\mapsto \mu_t\in\mathcal P_2(\overline\ps)$ is a narrowly continuous curve and $\overline\ps\ni\theta \mapsto \mathbf w_t(\theta)\in$ is a time-dependent velocity vector field. The \emph{Benamou-Brenier formula} \cite{BB} asserts that the Wasserstein distance between $\mu_0$ and $\mu_1$ can be computed as
\[
\Wdue(\mu_0,\mu_1)=\inf \int_0^1\left(\int_{\overline\Theta}|\mathbf w_t(\theta)|^2\,\mu_t(\ud\theta)\right)^{\frac12}\, \ud t,
\]
where the infimum is taken among all narrowly continuous curves from $\mu_0$ to $\mu_1$ in $\mathcal P_2(\overline\ps)$ and all Borel functions $[0,1]\times\overline\ps\ni (t,\theta)\mapsto\mathbf w_t(\theta)\in\rd$ such that $\mathbf w_t\in L^2_{\mu_t}(\overline\Theta;\rd)$ for a.e. $t\in(0,1)$ and such that \eqref{basiccontinuity} holds. 

By looking at the map $x\mapsto \pi(\cdot | x)\in\mathcal P_2(\overline\ps)$ associated to a probability kernel in the form \eqref{kernelgpi},
let us  fix two points $x_1, x_2\in\X$. We notice that a continuous curve $[0,1]\ni t\mapsto \alpha_{x_1,x_2}(t)\in\X$ such that $\alpha_{x_1,x_2}(0)=x_1$, $\alpha_{x_1,x_2}(1)=x_2$,  
naturally induces a curve on $\mathcal P_2(\overline\ps)$ defined by 
\begin{equation*}\label{CURVE}
[0,1]\ni t\mapsto \pi(\cdot|\alpha_{x_1,x_2}(t))\in\mathcal P_2(\overline\ps).
\end{equation*}
We use this curve for bounding the Wasserstein distance, as the computation of associated velocity vector fields $\mathbf{w}_t^{x_1,x_2}$ yields a direct estimate by means of the Benamou-Brenier formula. Indeed, if the vector field $\mathbf{w}_t^{x_1,x_2}\in L^2_{\pi(\cdot | \alpha_{x_1,x_2}(t))}(\overline\ps)$ satisfies the continuity equation in coupling with the curve 
$\pi(\cdot|\alpha_{x_1,x_2}(t))$,
for every fixed $x_1,x_2\in\X$, then the Benamou-Brenier formula entails
\begin{equation*}\label{wxyt}
\mathcal W_2(\pi(\cdot | x_1),\pi(\cdot | x_2))\le \int_0^1\left(\int_{\overline\ps} |\mathbf{w}_t^{x_1,x_2}(t,\theta)|^2\,\pi(\ud\theta | \alpha_{x_1,x_2}(t))\right)^{\frac12}\, \ud t.
\end{equation*}
Therefore, if we can further prove that $K\ge 0$ exist such that
\begin{equation}\label{reach}
\int_0^1\left(\int_{\overline\ps} |\mathbf{w}_t^{x_1,x_2}(t,\theta)|^2\,\pi(\ud\theta | \alpha_{x_1,x_2}(t))\right)^{\frac12}\,\ud t\le K|x_1-x_2|,
\end{equation}
then we obtain \eqref{main_problem} with the $\Wdue$ distance and $L = K$. 
In this regard, if $\alpha_{x_1,x_2}(t)$ is chosen to be a line segment, the velocity vector field scales as $|\alpha_{x_1,x_2}'(t)|=|x_1-x_2|$. 
Henceforth, we restrict indeed to the case of the line segment (which is related to the choice of $\X$ as a convex set), that is, we let 
\begin{equation}\label{esse}
\alpha_{x_1,x_2}(t)=\mathbf s_{x_1,x_2}(t):=(1-t)x_1+tx_2,\quad x_1,x_2\in\mathbb{X},\quad t\in[0,1].
\end{equation}

In order to obtain an estimate like \eqref{reach}, taking account of the time-scaling induced by the choice \eqref{esse}, we analyze the dual norm
\begin{equation}\label{universal}
\sup\left\{\,\partial_\nu\int_{\overline\ps}\psi(\theta)\,\pi(\ud\theta|x):\psi\in C^\infty_c(\overline\ps),\;\int_{\overline\ps}|\nabla\psi(\theta)|^2\,\pi(\ud\theta|x)\le1\right\},
\end{equation}
where $x\in\mathbb X$, $\nu$ is a unit vector in $\mathbb R^m$ and $\partial_\nu$ denotes the associated directional derivative.
We note that \eqref{universal} is the dual expression of the $L^2_{\pi(\cdot|x)}(\overline\ps)$ norm of the solution $\mathbf w_x^\nu$ to 
\begin{equation}\label{directionalcontinuity}
\partial_\nu\int_{\overline\ps} \psi(\theta)\,\pi(\ud\theta|x)=\int_{\overline\Theta}\left\langle\mathbf w_x^\nu(\theta),\nabla\psi(\theta)\right\rangle\,\pi(\ud\theta|x)\, \qquad\forall\psi\in C^\infty_c(\overline\Theta).
\end{equation}
For $x=\mathbf s_{x_1,x_2}(t)$ and $\nu=\tfrac{x_2-x_1}{|x_2-x_1|}$, we get indeed $\mathbf w_t^{x_1,x_2}=|x_2-x_1|\, \mathbf w_x^\nu$. Therefore, 
a crucial step towards the desired estimate \eqref{reach} will be an estimate for the norm \eqref{universal}.
Indeed, if $\|\mathbf w_x^\nu\|_{L^2_{\pi(\cdot | x)}(\Theta)}\le K$ for some constant $K$ that is independent of $x$ and $\nu$, then \eqref{reach} holds. 

It is natural to look for a solution to \eqref{directionalcontinuity} in the form of a gradient vector field $\mathbf w_x^\nu=\nabla u_x^\nu$, thus providing optimality of the $L^2_{\pi(\cdot | x)}$ 
norm (as we detail in Section \ref{theorysection}). Therefore, by recalling the general form \eqref{kernelgpi} of the probability kernel,
we formally interpret equation \eqref{directionalcontinuity} as a family of degenerate elliptic problems (where we write $g_x(\cdot)=g(x,\cdot)$, hinting at the fact that here $x\in\mathbb X$
plays the role of parameter)
\begin{equation}\label{abstractelliptic}\left\{\begin{array}{ll}
-\mathrm{div}\left(g_x\pi\,\nabla u_x^\nu\right)=\partial_\nu g_x\,\pi\quad &\mbox{in}\ \ps\\
g_x \nabla u_x \cdot \mathbf n=0\quad&\mbox{on}\ \partial\Theta,
\end{array}\right.
\end{equation}
where $\mathbf n$ denotes the normal to the boundary.

Existence, regularity and estimation of weak solutions to degenerate elliptic equations (see \cite{FJK, FKS}) are related to the validity of a weighted Poincar\'e inequality such as \eqref{wirtinger}, the weights being given here by the p.m.'s $\pi(\cdot|x)$ as $x$ varies in $\X$. 
In view of the above discussion, the result in Theorem \ref{th:0}-{\bf (iii)} 
 has a clear PDE interpretation:
$g\in L^1_\pi(\Theta; W^{1,1}_{loc}(\mathbb X))$ is a regularity assumption that allows to take the $\partial_\nu$-derivative under the integral sign in \eqref{directionalcontinuity}, 
while the condition involving both the Poincar\'e constant and the Fisher functional appears as an estimate of the solution to \eqref{abstractelliptic}.
A similar interpretation holds for Theorem \ref{th:0}-{\bf (iv)}.

\subsection{Estimates of $\Wdue$ on moving domains} \label{sect:mainresults_2}

Also in this subsection, we keep the mathematical setting of Assumptions \ref{ass:1} and we confine ourselves to treating 
kernels of the form \eqref{kernelgpi}. We provide two other results, in which we get rid of the positivity restriction on $g$ appearing in Theorem \ref{th:0}-{\bf (iii)} and of the  Sobolev assumption on 
$g$ in the $x$ variable. This task requires the introduction of some new notation, along with the assumption that $\pi$ admits a density $q$ with respect to the Lebesgue measure $\mathcal L^d$. 
Thus, without loss of generality, we fix  $\pi=\mathcal L^d\mr\Theta$ in \eqref{kernelgpi}, throughout this subsection.  
For $\mathcal L^m$-a.e. $x\in\X$, we assume that $\Theta_x:=\{g(x,\cdot)>0\}$ is, up to a $\mathcal L^d$-null set, an open connected subset of $\ps$ with locally Lipschitz boundary.
Moreover, for any direction $\nu\in\mathbb S^{m-1}$, we consider the following Neumann boundary value problem 
\begin{equation}\label{neumanngintro}\left\{\begin{array}{ll}
-\mathrm{div}(g_x\nabla u_x^{\nu})=\partial_\nu \tilde g_x\quad &\mbox{ in }\Theta_x\\
g_x\nabla u_x^{\nu}\cdot \mathbf{n}_x=g_x\mathbf V_x^{\nu}\cdot\mathbf n_x\quad &\mbox{ on }\partial\Theta_x
\end{array}\right.
\end{equation}
where $g_x$ is a shorthand for $g(x,\cdot)$ and $\mathbf n_x$ denotes the exterior unit normal to $\partial\Theta_x$. This problem represents of course a generalization of \eqref{abstractelliptic}. 
The map $(\X,\ps)\ni (x,\theta)\mapsto \tilde g(x,\theta)=\tilde g_x(\theta)$ is a Sobolev map extending $g(x,\cdot)$ to the whole of $\ps$, 
while $\partial_\nu$ denotes the derivative in the direction $\nu\in\mathbb S^{m-1}$. More precisely, we assume that  
there exists $\tilde g\in L^1_{loc}(\X\times\ps)$ such that $\tilde g\in W^{1,1}(\tilde{\X}\times\ps)$ for any open set $\tilde{\X}$ compactly contained in $\X$ and such that $g(x,\theta)=  \tilde g(x,\theta)\,
\mathds{1}_{\Theta_x}(\theta)$ for $(\mathcal L^m\otimes\mathcal L^d)$-a.e. $(x,\theta)\in\X\times\ps$.
We note this extension guarantees that the right hand side in the first equation of \eqref{neumanngintro} belongs to $L^1(\ps_x)$ for $\mathcal L^m$-a.e. $x\in\X$.
Moreover, $\mathbf V_x^{\nu}:\Theta_x\to\rd$ is the vector field representing the velocity of $\ps_x$ in $\ps$, when the parameter $x$ varies along the $\nu$ direction. Here, $\Theta_x$ is assumed to be the image 
of a reference connected open set with locally Lipschitz boundary, say $\Theta_*\subset\rd$, through $\Phi_x$, where $\{\Phi_x\}_{x\in \X}$ is a smooth family of diffeomorphisms.
In such a case, we say that the positivity set of $g(x,\cdot)$ varies according to a {\it $\X$-regular motion}. The detailed notion of $\X$-regular motion will be given in Definition \ref{xmotion}, in Section \ref{proofs}.
Then, we put $\mathbf V_x^\nu=\partial_\nu\Phi_x\circ\Phi_x^{-1}$ and we introduce the notation $\mathbf V_x$ for the matrix $\nabla_x \Phi_x\circ\Phi^{-1}_x$. 
If $g\in L^1_{loc}(\mathbb X\times\Theta)$ satisfies all the above conditions, we say that $g$ admits a {\it regular extension}. Again, a detailed notion of regular extension will be given in Definition \ref{densityxy}, in Section \ref{proofs}.

The way is now paved for the formulation of a first abstract result, where we refer to weak solutions to problem \eqref{neumanngintro}. For clarity, a weak solution $u_x^\nu$ is defined in the usual way, through integration by parts, as an element of the weighted Sobolev space $H^1(\Theta_x,g_x)$. See Definition \ref{soluzionedebole} below.
\begin{theorem}\label{concretemixed} 
Let $\pi(\cdot|\cdot)$ be a kernel in the form \eqref{kernelgpi}, with $\pi=\mathcal L^d\mr\Theta$ and with $g\in L^1_{loc}(\mathbb X\times\Theta)$ admitting a regular extension and satisfying $\int_\Theta |\theta|^2g(x,\theta)\,d\theta<+\infty$ for $\mathcal L^m$-a.e. $x\in\mathbb X$. 
For any $\nu\in\mathbb S^{m-1}$, suppose there exists a weak solution $u_x^\nu\in H^1(\ps_x,g_x)$ to 
the problem \eqref{neumanngintro}, for $\mathcal L^{m-1}$-a.e. $x\in\X$, and that 
\begin{equation}\label{nuovoess}
K:=\sup_{\nu\in\mathbb S^{m-1}} \esssup_{x\in\X}\left(\int_{\Theta_x}|\nabla u_x^{\nu}(\theta)|^2\,g(x,\theta)\,\ud\theta\right)^{1/2}<+\infty.
\end{equation} 
Then, there exists a $\Wdue$-Lipschitz version of $\pi(\cdot|\cdot)$ satisfying \eqref{main_problem} with $L=K$.
\end{theorem}

The next theorem will provide an estimate of the solution to problem \ref{neumanngintro}. For a kernel $\pi(\cdot|\cdot)$ in the form \eqref{kernelgpi} with $\pi=\mathcal L^d\mr\Theta$, 
we will use the shorthand $\mathcal C[g(x,\cdot)]$ to denote the Poincar\'e constant of the p.m. $g(x,\cdot)\,\mathcal L^d\mr\Theta$ on $(\ps,\psf)$. Denoting as usual by $\nabla_x$ the gradient in the $x$-variable and by 
$\nabla$ the gradient in the $\theta$-variable, we introduce the {\it Fisher functionals} associated to (the regular extension of) $g$ as 
\begin{equation}\label{fisherfunctional}
\mathcal{J}_1[\tilde g(x,\cdot)]:=\displaystyle\left(\int_{\Theta_x} \frac{\left|\nabla_x \tilde g(x,\theta)\right|^2}{\tilde g(x,\theta)}\, \ud\theta\right)^{\frac12},\qquad
\mathcal{J}_2[\tilde g(x,\cdot)]:=\displaystyle\left(\int_{\Theta_x} \frac{\left|\nabla \tilde g(x,\theta)\right|^2}{\tilde g(x,\theta)}\, \ud\theta\right)^{\frac12}\ .
\end{equation}

\begin{theorem}\label{maintrace} 
Let $\pi(\cdot|\cdot)$ be a kernel in the form \eqref{kernelgpi}, with $\pi=\mathcal L^d\mr\Theta$ and with $g\in L^1_{loc}(\mathbb X\times\Theta)$ admitting a regular extension $\tilde g$ and satisfying $\int_\Theta |\theta|^2g(x,\theta)\,\ud\theta<+\infty$ for $\mathcal L^m$-a.e. $x\in\mathbb X$. 
 If
\begin{equation*}\label{longassumption}\begin{aligned}
K:=\esssup_{x\in \X} &\Big{\{}\|\mathbf V_x(\cdot)\|_{W^{1,\infty}(\Theta_x)}\,\Big{(}1+\mathcal C[g(x,\cdot)]\left(1+\mathcal J_2[\tilde g(x,\cdot)]\right)\Big{)}
+\mathcal C[g(x,\cdot)]\mathcal J_1[\tilde g(x,\cdot)] \Big{\}}<+\infty
\end{aligned}\end{equation*}
is valid, then there exists a  $\Wdue$-Lipschitz version of $\pi(\cdot|\cdot)$ satisfying \eqref{main_problem} with $L=K$.
\end{theorem}

In the derivation of \eqref{neumanngintro} from the {continuity equation} (see Section \ref{sectionproofs}),
we handle the derivative of the integral on the left-hand side of \eqref{directionalcontinuity} by making use of the \emph{Reynolds transport formula} from continuum mechanics (see Lemma \ref{reynolds} in the Appendix).
This explains the role of the vector $\mathbf V_x^\nu=\partial_\nu\Phi_x\circ\Phi_x^{-1}$ that represents the spatial velocity, defined on the deformed configuration $\ps_x$, whereas $\ps_*$ is the reference configuration.
This approach is, in a sense, alternative (although less general) to the optimal transport formalism and the Monge-Ampere equation, but suitable to the statistical framework, where probability densities are often defined by truncation. In this context, the extension map $\tilde g_x(\cdot)$ is actually given a-priori, as in the examples dealing with Pareto-type statistical models, in Subsection \ref{wilfrido}. 
Of course, if $\Theta_x=\Theta$ for $\mathcal L^m$-a.e. $x\in\X$, Theorem \ref{maintrace} is reduced to a particular instance of Theorem \ref{th:0}-{\bf (iii)}, as problem \eqref{neumanngintro} is reduced to \eqref{abstractelliptic}.

\subsection{A remarkable example: Pareto statistical models}\label{wilfrido}

Pareto statistical models, which are considered in the next two propositions, provide a paradigmatic  application of Theorem \ref{concretemixed}, as they gives rise to a moving support of probability densities that are defined by truncation. 

\begin{pro}\label{Pareto 1D} 
Consider the one-dimensional Pareto statistical model
\[
\mathbb{X}=(1,+\infty), \quad \Theta=(1,\theta_0), \quad \theta_0\in(1,+\infty],\quad f(x|\theta)=\frac{\theta}{x^2}\,\mathds{1}\{\theta<x\}.
\]
Suppose we are given a prior distribution $\pi$, whose support is  $\overline\Theta$, admitting a density $q\in L^1(\Theta)$, and let $Q(\theta):=\theta q(\theta)$.   Assume further that  $Q\in W^{1,1}(\Theta)$ and $1/Q\in L^1(\Theta)$, and let  $\;C_Q(\theta):= Q(\theta)\left(\int_1^\theta Q(\tau)\,d\tau\right)^{-1/2}\left(\int_1^\theta\frac{d\tau}{Q(\tau)}\right)^{1/2}$.  Then the posterior distribution $\pi(\cdot|\cdot)$, defined by means of \eqref{kernelgpi} and \eqref{Bayes}, satisfies \eqref{main_problem}  with distance $\mathcal W_2$ and $L=K$, 
where $K:=\sup_{x\in\mathbb X} C_Q(x\wedge \theta_0)$.  
\end{pro}

\begin{proof}  We give the proof in case $\theta_0=+\infty$ (minor variants are required if $\theta_0<+\infty$). 
The proof is a direct application of Theorem \ref{concretemixed}. 
With respect to the notation therein, we drop the apex $\nu$ as the directional derivative is reduced to the derivative in the $x$-variable.
We have by Bayes formula \eqref{Bayes}
$$
\rho(x)=\frac1{x^2}\,\int_1^x Q(\tau)\,\ud\tau,\qquad g_x(\theta)=\frac{f(x|\theta)q(\theta)}{\rho(x)}=\frac{Q(\theta)}{\int_1^xQ(\tau)\,\ud\tau}\,\mathds{1}\{\theta<z\},
$$
so that $\Theta_x=(1,x)$ is the positivity set of $g_x$. It is clear that the function $\X\times\Theta\ni(x,\theta)\mapsto g_x(\theta)$ satisfies $\int_{\Theta} g_x(\theta)\,\ud\theta=1$ and $\int_\Theta \theta^2 g(x,\theta)\,\ud\theta<+\infty$ for every $x\in\X$. Moreover, as required by Theorem \ref{concretemixed}, it admits a regular extension according to 
Definition \ref{densityxy} which is found later in Section \ref{sect:movingproof}.  
Indeed, we may define  $\Phi_x(\cdot):\Theta_*\to\Theta_x$ by $\Theta_*=(1,2)$ and $\Phi_x(\theta) = (\theta-1)(x-1)+1$. As a consequence, we have
$\partial_x\Phi_x\circ(\Phi_x)^{-1}(\theta)=\tfrac{\theta-1}{x-1}$ and $\Phi_x(\theta)$ satisfies the conditions of Definition \ref{xmotion}. 
Moreover, we may consider the natural extension 
$$
\tilde g_x(\theta) := \frac{Q(\theta)}{\int_1^x Q(\tau)\,\ud\tau},\qquad x\in(1,+\infty),\; \theta\in(1,+\infty).
$$
By the assumptions on $q$, the map $(t,\theta)\mapsto \tilde g_t(\theta)$ belongs to $W^{1,1}((\alpha,\beta)\times\Theta)$ for any $1<\alpha<\beta<+\infty$. It is indeed a  regular extension in the sense of Definition \ref{densityxy}. Therefore, problem \eqref{neumanngintro} is reduced to
\begin{equation}\label{neumanng1d}
\left\{\begin{array}{ll}
-(g_x\,u_x')' = \partial_x \tilde g_x \quad &\mbox{ in } (1,x)\\
 g_x(x)\,u_x'(x) = g_x(x)\\
g_x(1)\,u_x(1) = 0,
\end{array}\right.
\end{equation}
where the $'$ stands for the derivative in the $\theta$ variable. By taking into account that
$
\partial_x\tilde g_x(\theta) = -{Q(\theta)Q(x)}\left(\int_1^x Q(\tau)\,\ud\tau\right)^{-2},
$
the solution $u_x$ to problem \ref{neumanng1d} satisfies for any $x\in(1,+\infty)$
\[
u_x'(\theta) = \frac{Q(x)}{Q(\theta)}\left(\int_1^x Q(\tau)\,\ud\tau\right)^{-1}\int_1^\theta Q(\tau)\,\ud\tau.
\] 
We obtain
\begin{equation*}\label{cizeta}
\left(\int_1^x |u_x'(\theta)|^2\,g_x(\theta)\,\ud\theta\right)^{1/2}\le Q(x)\left(\int_1^x Q(\tau)\,\ud\tau\right)^{-1/2}\left(\int_1^x\frac{\ud\tau}{Q(\tau)}\right)^{1/2} = C_Q(x).
\end{equation*}
By the assumptions on $q$ it easily follows that $C_Q(x)$ is bounded on $\X$, so that the assumptions of Theorem \ref{concretemixed} are satisfied, thus we conclude by taking 
$\sup_{x\in\X} C_Q(x)$ as bound for the Lipschitz constant.
\end{proof}


\begin{pro}
Let us consider the statistical model
\[
\mathbb X =(1,+\infty), \quad (\theta,\eps)\in\Theta=(1,2)^2, \quad f(x|\theta,\eps) = \frac{\eps \theta^\eps}{x^{1+\eps}}\, \mathds{1}\{\theta<x\}
\]
along with a prior probability density $q\in C^2(\Theta)$ such that $0<c_q\le q(\theta,\eps)$ for any $(\theta,\eps)\in\Theta$. 
Then there exists  an explicit positive constant $Z_q$, only depending on $c_q$ and $\|q\|_2:=\sup_\Theta q+\sup_\Theta|\nabla q|+\sup_\Theta |\nabla^2 q|$, such that the posterior distribution satisfies
\eqref{main_problem} with distance $\Wdue$ and $L = Z_q$.
\end{pro}

\begin{proof}
We first look at the values of $x\in(1,2]$.
 Let $C_q:=\sup_\Theta q$. A computation shows that for $\theta\in(1,2)$, $\eps\in(1,2)$,
\[
g_x(\theta,\eps) = \frac{f(\theta,\eps|x) q(\theta,\eps)}{\rho(x)} = \frac{\eps\theta^{\eps}q(\theta,\eps) \mathds{1}\{\theta<x\}(\theta,\eps)}{x^{\eps+1}\rho(x)},\]
where
\[
\rho(x) = \int_1^x\int_1^2\frac{\sigma\tau^{\sigma+1}}{x^{\sigma+1}}\,q(\tau,\sigma)\,\ud\sigma\,\ud\tau\ .
\]
An easy estimate shows that for any $x\in[1,2]$
\begin{equation}\begin{aligned}\label{basicrho}
&8C_q(x-1)\ge \rho(x)\ge c_q(x-1),\\& |\rho'(x)|\le 26C_q,\\& |\rho'(x)|+|\rho''(x)|+|\rho'''(x)|\le M \|q\|_2
\end{aligned}\end{equation}
where $M$ is a suitable numerical constant. Moreover
\[
\tilde g_x(\theta,\eps) = \frac{\eps\theta^{\eps}q(\theta,\eps)}{x^{\eps+1}\rho(x)}\qquad \mbox{and}  \qquad\partial_x \tilde g_x(\theta,\eps) = -A_x(\eps)\theta^\eps q(\theta,\eps),\]
where
\[
A_x(\eps) := \frac{\eps(\eps+1)}{x^{\eps+2}\rho(x)}+\frac{\eps\rho'(x)}{x^{\eps+1}\rho^2(x)},
\]
so that from \eqref{basicrho} we deduce that there exists a positive constant $K_q$, depending only on $c_q$ and $C_q$, such that
\begin{equation}\label{AZ}
|A_x(\eps)|\le \frac{K_q}{(x-1)^2}\quad\mbox{for any $\eps\in(1,2)$ and any $x\in(1,2]$}. 
\end{equation}
We see  that the map $(x,\theta,\eps)\mapsto\tilde g_x(\theta,\eps)$ belongs to $W^{1,1}((\alpha,\beta)\times(1,2)\times(1,2))$ for any $1<\alpha<\beta<2$.
Moreover, for any $x\in(1,2)$ the map $(\theta,\eps)\mapsto \partial_x\tilde g_x(\theta,\eps)$ belongs to $L^1(\Theta_x)$ where $\Theta_x := (1,x)\times(1,2) = \Phi_x(\Theta_2)$ and $\Phi_x(\theta,\eps)=((\theta-1)(x-1)+1,\eps)$ has first component as in the proof of Proposition \ref{Pareto 1D}. 
 In this way, we see that indeed $(x,\theta,\eps)\mapsto \tilde g_x(\theta,\eps)$ is a regular extension of the function $(x,\theta,\eps)\mapsto g_x(\theta,\eps)$ on $(1,2)\times\Theta$, according to Definition \ref{densityxy}. 
In order to conclude, we apply Theorem \ref{concretemixed}. The corresponding Neumann boundary value problem is posed on a rectangle, and precisely it is (again we drop the apex $\nu$ as $\X$ is one-dimensional)
\begin{equation*}\label{neumanngnew2d1}
\left\{\begin{array}{ll}
-\mathrm{div}(g_x\nabla u_x) = \partial_x\tilde g_x\quad &\mbox{ in }\Theta_x\\
\partial_\theta u_x = 1\quad &\mbox{ on }\Gamma_x := \{x\}\times(1,2)\\
\nabla u_x\cdot \mathbf n_x = 0\quad &\mbox{ on } \partial\Theta_x\setminus\Gamma_x.
\end{array}\right.
\end{equation*}
We define
$
G_x(\theta,\eps) := \int_1^\theta\partial_x \tilde g_x(\tau,\eps)\,\ud\tau = -{A_x(\eps)}\int_1^\theta\tau^\eps\,q(\tau,\eps)\,\ud\tau,
$
and we proceed with the estimate of the $H^1(\Theta_x,g_x)$ norm of $u_x$ by duality and using the notion of weak solution to the above problem (see Definition \ref{soluzionedebole}). We obtain after an integration by parts in the $\theta$-variable
\[
\begin{aligned}
\int_{\Theta_x} |\nabla u_x|^2\,g_x &= \sup_{\|\psi\|_x=1} \int_{\Theta_x}\nabla u_x\cdot\nabla \psi\,g_x\\&=\sup_{\|\psi\|_x=1}\left(\int_{\Theta_x}\psi(\theta,\eps)\partial_x\tilde g_x(\theta,\eps)\,\ud\theta\,\ud\eps+\!\int_1^2\psi(x,\eps) g_x(x,\eps)\,\ud\eps\right)\\
&= \sup_{\|\psi\|_x=1}\left(\int_{\Theta_x}-\partial_\theta\psi(\theta,\eps)G_x(\theta,\eps)\,\ud\theta\,\ud\eps+\int_1^2\psi(x,\eps)(G_x(x,\eps)+g_x(x,\eps))\,\ud\eps\right)
\end{aligned}\]
where the supremum is taken among test functions $\psi\in C^1_{g_x}(\overline\Theta_x)$ and $\|\psi\|_x$ is a shorthand for the norm $(\int_{\Theta_x}|\nabla\psi|^2 g_x)^{1/2}$ on $C^1_{g_x}(\overline\Theta_x)$. 
Therefore, with the notation $H_x(\eps) := G_x(x,\eps) + g_x(x,\eps)$ and with the divergence theorem we obtain
\[
\begin{aligned}
\int_{\Theta_x} |\nabla u_x|^2\,g_x& = \sup_{\|\psi\|_x=1} \left(\int_{\Theta_x}-G_x(\theta,\eps)\,\partial_\theta\psi(\theta,\eps)\,\ud\theta\,\ud\eps+\int_{\Theta_x}\mathrm{div}\left(\psi(\theta,\eps)\,\tfrac{\theta-1}{x-1}\,(H_x(\eps),0)\right)\,\ud\theta\,\ud\eps\right)\\
&=\sup_{\|\psi\|_x=1}\left(\int_{\Theta_x}\partial_\theta\psi\left(\tfrac{\theta-1}{x-1}\,H_x(\eps)-G_x(\theta,\eps)\right)\,\ud\theta\,\ud\eps+\int_{\Theta_x}\frac{\psi(\theta,\eps)H_x(\eps)}{x-1}\,\ud\theta\,\ud\eps\right)\\
&\le\sup_{\|\psi\|_x=1}\left(\int_{\Theta_x}|\nabla\psi(\theta,\eps)|\,|G_x(\theta,\eps)+H_x(\eps)|\,\ud\theta\,\ud\eps + \int_{\Theta_x}\frac{|\psi(\theta,\eps)H_x(\eps)|}{x-1}\,\ud\theta\,\ud\eps\right).
\end{aligned}
\]
By Cauchy-Schwarz inequality and  \eqref{wirtinger}, if  $\mathcal C[g_x]$ is the Poincar\'e constant of the probability measure $g_x\,\mathcal L^2\mr\Theta_x$, we get
\begin{equation}\label{estimo}
\begin{aligned}
\int_{\Theta_x}|\nabla u_x|^2\,g_x&\le\left(\left(\int_{\Theta_x}\frac{(G_x+H_x)^2}{g_x}\right)^\frac12+\frac{\mathcal C[g_x]}{x-1}\left(\int_{\Theta_x}\frac{H_x^2}{g_x}\right)^\frac12\right).
\end{aligned}
\end{equation}
Let us now compute suitable bounds for the terms in the above right hand side. 
From \eqref{basicrho} and \eqref{AZ} we get for any $\theta\in(1,2)$, any $\eps\in(1,2)$ and any $x\in(1,2]$ that
$
|G_x(\theta,\eps)| \le 4C_q|A_x(\eps)|\,(x-1)\le \frac{4C_q K_q}{x-1}
$ and thus
\begin{equation}\label{zz2}
\int_{\Theta_x}\frac{G_x(\theta,\eps)^2}{g_x(\theta,\eps)}\,\ud\theta\,\ud\eps\le 2^{10} c_q^{-1}C_q^3 K_q^2,\qquad \mbox{for every $x\in (1,2]$}.
\end{equation}
Since
$
H_x(\eps) 
= \dfrac{N_\eps(x)}{x^{\eps+2}\rho^2(x)}$,
where 
$$
N_\eps(x) := \left(-\eps(\eps-1)\rho(x)-\eps x\rho'(x)\right)\int_1^x\tau^\eps\,q(\tau,\eps)\,\ud\tau + \eps x^{1+\eps}\rho(x) q(x,\eps),
$$
we have $N_\eps(1)=0$, and since
$
N'_\eps(x) 
$
also vanishes at $x=1$, by a  Taylor expansion in $x$ we have $N_\eps(x)=\tfrac12(x-1)^2N''_\eps(\xi_\eps)$ for some $\xi_\eps\in[1,x]$ and a computation exploiting \eqref{basicrho} shows that for any $x\in (1,2]$ there holds 
$|N''_\eps(x)|\le 2U \|q\|_2$ for a universal constant $U$, so that we deduce $H_x(\eps)\le c_q^{-2}{U\|q\|_2}$ for any $\eps\in(1,2)$ and any $x\in(1,2]$. As a consequence, \eqref{basicrho} implies that for every $x\in(1,2]$
\begin{equation}\label{zz1} 
\frac1{x-1}\left(\int_{\Theta_x}\frac{H_x^2(\eps)}{g_x(\theta,\eps)}\,\ud\theta\,\ud\eps\right)^{\frac12} \le 8\,C_q^{-3/2}\,c_q^{-1/2}\, U\,\|q\|_2.
\end{equation}
Let us moreover treat $\mathcal C[g_x]$ by invoking the Holley-Stroock estimate (see Section \ref{section:mainresults}): letting $V_x(\theta,\eps):= -\log g_x(\theta,\eps)$, we have  $\mathcal C^2[g_x]\le \exp\{\sup_{\Theta_x}V_x-\inf_{\Theta_x}V_x\}\,\mathcal C^2[\mathcal L^2\mr\Theta_x]$, where $\mathcal C[\mathcal L^2\mr\Theta_x]$ enjoys the standard estimate \cite{PW} in terms of $\mathrm{diam}(\Theta_x)/\pi\le \sqrt2/\pi$, for any $1<x\le2$. Since a direct estimate shows that
$
\sup_{\Theta_x} V_x-\inf_{\Theta_x} V_x 
\le \log(C_q/c_q)+6\log2$ for every $x\in(1,2]$,
we eventually get $\mathcal C^2[g_x]\le 2^7\pi^{-2}\, C_q/c_q $. This estimate can be inserted in \eqref{estimo}, together with \eqref{zz1} and \eqref{zz2}, and we deduce that for a suitable explicit constant $Z_q$, only depending on $c_q$ and $\|q\|_2$, there holds
$
\int_{\Theta_x}|\nabla u_x|^2\,g_x\le Z_q^2$ for any $z\in(1,2]$.
In conclusion, Theorem \ref{concretemixed} yields the validity of \eqref{main_problem} with $\Wdue$ distance and $L = Z_q$ for any $x_1,x_2 \in(1,2]$. Since the density of $\pi(\cdot,\cdot|x)$ is $g_x(\cdot,\cdot)$, and since the latter is given by
$g_x(\theta,\eps) = {\eps\theta^{\eps}q(\theta,\eps)}/\mathcal Q$ for any $x>2$, where $\mathcal Q:={\int_1^2\int_1^2\sigma\tau^\sigma q(\tau,\sigma)\,\ud\tau\,\ud\sigma}$, the estimate \eqref{main_problem}
trivially extends to all $x_1,x_2\in\X$.
\end{proof}

\subsection{Infinite-dimensional sample space}\label{infinitesection} 

Here, generalizing the setting displayed in  Assumptions \ref{ass:1},
we extend Theorem \ref{th:0} to the case in which $\X$ is a convex set of a real separable Banach space $(\V, \|\cdot\|_{\V})$, endowed with a $\sigma$-finite reference measure $\lambda$. 
Upon denoting by $\mathcal H^1$ the 1-dimensional \emph{Hausdorff measure} on $(\V, \Bcr(\V))$, we prescribe the following properties for $\lambda$:
\begin{enumerate}
\item[$\lambda 1)$] $\lambda(A) > 0$ for every nonempty open set $A \in \Bcr(\V)$;
\item[$\lambda 2)$] if $\lambda(A) = 0$ for some $A \in \Bcr(\V)$, then $\lambda(\alpha A + v) = 0$ for all $\alpha > 0$ and $v \in \V$;
\item[$\lambda 3)$] for a set $A \in \Bcr(\V)$, the condition $\lambda\left(\left\{v \in \V\ |\ \mathcal H^1([0,v] \cap A) > 0\right\}\right) = 0$ entails $\lambda(A) = 0$, where $[0,v] := 
\{\alpha v\ |\ \alpha \in [0,1]\}$.
\end{enumerate}
We plainly observe that $\V$ can be taken equal to any Euclidean space $\rd$ for any $d \in \N$, and that the standard $d$-dimensional Lebesgue measure fulfills the conditions $\lambda 1)-\lambda 2)-\lambda 3)$. 
We  notice that $\lambda1)$ entails that the complement of any $\lambda$-null set is dense in $\V$. Non-degenerate Gaussian measures on $\V$ also satisfy $\lambda 1)-\lambda 2)-\lambda 3)$. See \cite{phelps}. 

In the next theorem, we again confine ourselves to treating kernels of the form \eqref{kernelgpi}, but now we stress that $g:\mathbb X\times\ps\to\R$ is defined pointwise. 
Therefore, $g$ and $\pi$ determine pointwise the probability kernel $\pi(\cdot|\cdot)$.    
Since we let $\X$ be a convex subset of $\V$, we have $\ud_\X(x_1,x_2):=\|x_1-x_2\|_{\V}$.  
For a $\lambda$-null subset $\Z$ of $\X$,
we define
\begin{equation}\label{Bset}
\mathbb{B}(\Z):=\{(x_1,x_2)\in\X^2: x_1, x_2\in\X\setminus \Z,\, \mathcal H^1([x_1,x_2] \cap \Z) = 0\}.
\end{equation}
We let $D_x$ denote the Gateaux differential operator with respect to the $x$-variable and $\V'$ be the dual space of $\V'$ with operator norm $\|\cdot\|_{\V'} $. 
Finally, the Fisher functional relative to $\pi$ is now defined as
$$\mbox{  
$\;\mathcal{J}_\pi[g(x,\cdot)]:=\displaystyle\left(\int_\Theta \frac{\left\|D_x g(x,\theta)\right\|_{\mathbb{V'}}^2}{g(x,\theta)}\, \pi(\ud\theta)\right)^{\frac12}.$}
$$
The way is now paved for the formulation of the following
\begin{theorem}\label{infinity} 
Suppose there exists a $\lambda$-null set $\Z\subset\X$ such that, for $\pi$-a.e. $\theta\in\Theta$, the mapping $x\mapsto g(x,\theta)$ is Gateaux-differentiable at any $x\in \X\setminus\Z$ and absolutely continuous on any segment $[x_1,x_2]$ with $(x_1,x_2)\in\mathbb B(\Z)$. When $\Wp$ is involved, it is further assumed that $\int_\ps |\theta|^p \pi(\ud\theta|x)<+\infty$ for any  $x\in\X$. Then, the following statements hold.
\begin{itemize}
\item[\bf(i)] Suppose that
$$
K:=\displaystyle\lesssup_{x\in\X} \int_\Theta\|D_x g(x,\theta)\|_{\V'}\,\pi(\ud\theta)<+\infty.
$$
Then, there exists a $d_{TV}$-Lipschitz version of $\pi(\cdot|\cdot)$, satisfying \eqref{main_problem} with $L=K/2$. 
\item[\bf(ii)]  
Let $1\le q<+\infty$ and let $p$ be the H\"older conjugate exponent of $q$. If $\pi$ admits a Poincar\'e constant $\mathcal C_q[\pi]$ and if
\[
K:=\pi(\Theta)^{\frac1q}\,\mathcal C_q[\pi]\,\lesssup_{x\in\mathbb X}\Big{\|}\|D_x g(x,\cdot)\|_{\V'}\Big{\|}_{L^p_\pi(\Theta)}<+\infty,
\]
then there exists a $\Wuno$-Lipschitz version of $\pi(\cdot|\cdot)$, satisfying \eqref{main_problem} with $L=K$. 
\item[\bf(iii)] 
Let  $g>0$ $(\lambda\otimes\pi)$-a.e. in $\X\times\Theta$
and let $\int_0^1\int_\Theta \|D_x g(\mathbf s_{x_1,x_2}(t),\theta)\|_{\V'}\,\pi(\ud\theta)<+\infty$ for any $(x_1,x_2)\in\mathbb B(\Z)$. If
\[
K:=\lesssup_{x\in\X}\mathcal C[g(x,\cdot)\,\pi]\,\mathcal J_\pi[g(x,\cdot)]<+\infty,
\]
then there exists a $\mathcal W_2$-Lipschitz version of $\pi(\cdot|\cdot)$, satisfying \eqref{main_problem} with $L=K$. 
\end{itemize}
\end{theorem}

If $\mathbb V$ is infinite-dimensional, a total variation distance estimate like the one of Theorem \ref{infinity}-{\bf (i)} can be found under stronger assumptions like the following:
the map $x\mapsto g(x,\theta)$ is Lipschitz with a constant $L_\theta$ satisfying $\int_\ps L_\theta\,\pi(\ud\theta)<+\infty$ and $\lambda$ is a Gaussian measure. In particular, these last assumptions imply Gateaux differentiability, according to \cite[Theorem 1.1]{LP} and references therein.

\section{Proofs}\label{sectionproofs}

\subsection{The dynamic formulation of the Wasserstein distance}\label{theorysection}
Here, we provide the theoretical framework for the estimate of the $2$-Wasserstein distance. We start by introducing some facts about the dynamic formulation by means of the continuity equation and the Benamou-Brenier \cite{BB} formula, which are related to the geometry of the space of probability measures. This theory is established in the seminal paper by Otto \cite{O}, see also \cite{OV}, in the books of Villani \cite{vilMass, V2}, as well as the book by Ambrosio, Gigli and Savar\'e \cite{AGS}, see also \cite{AG}. 
 Then, we give most emphasis to the continuity equation as a family of degenerate elliptic boundary value problems, parametrized by time.

We let $\Theta$ be an open connected subset of $\mathbb{R}^d$. We recall that  by $C^1_c(\overline\Theta)$ we denote the space of  functions in $C^1(\overline\Theta)$ whose support is a compact set contained in $\overline\Theta$. Of course, $C^1_c(\overline\Theta)\equiv C^1(\overline\Theta)$ if $\Theta$ is bounded. 
The space
$C^1_c(\overline\Theta)$ is separable with respect to the $C^1(\overline\Theta)$ norm $\|\psi\|_{C^1(\overline\Theta)}:=\sup_{\overline\Theta}|\psi|+\sup_{\overline\Theta}|\nabla\psi|$ and contains the space of $C^1$ functions with compact support in $\Theta$.

We shall consider Borel families of measures $\{\mu_t\}_{t\in[0,1]}\subset\mathcal{P}(\overline\Theta)$, i.e., $[0,1]\ni t\mapsto\mu_t(A)$ is Borel measurable for any Borel set $A\subseteq \overline\Theta$. 
Moreover, $[0,1]\ni t\mapsto \mu_t\in\mathcal{P}(\overline\Theta)$ is said to be a narrowly continuous curve if $t,t_0\in[0,1]$ and $t\to t_0$ imply the narrow convergence of $\mu_t$ to $\mu_{t_0}$.
In the following, we say that a narrowly continuous curve $[0,1]\ni t\mapsto \mu_t\in\mathcal{P}(\overline\Theta)$ satisfies the continuity equation on $\overline\Theta$, in coupling with a  family of vector fields 
$\{\mathbf{v}_t\}_{t\in[0,1]}$ such that $[0,1]\times\overline\Theta\ni(t,\theta)\mapsto\mathbf{v}_t(\theta)\in\rd$ is Borel measurable, if
\begin{equation}\label{distributionalcontinuityequation}
\int_0^1\int_{\overline\Theta}(\partial_t\varphi(t,\theta)+\nabla\varphi(t,\theta)\cdot\mathbf{v}_t(\theta))\,\mu_t(\ud\theta)\,\ud t=0\qquad\mbox{for any $\varphi\in C^1_c((0,1)\times\overline\Theta)$}.
\end{equation} 
Here and in the following, $\nabla$ denotes the gradient in the $\theta$ variable.

We start by giving sufficient conditions on the curve $[0,1]\ni t\mapsto\mu_t$  in order to apply the Benamou-Brenier formula and estimate the $2$-Wasserstein distance between $\mu_0$ and $\mu_1$.
\begin{theorem}\label{abs}
Let $\mu_0\in\mathcal P_2(\overline\Theta)$, $\mu_1\in\mathcal P_2(\overline\Theta)$, and let
 $\{\mu_t\}_{t\in[0,1]}\subset \mathcal{P}(\overline\Theta)$ be a Borel family of probability measures. 
 Let $\mathcal D$ be a countable dense subset of $  C^1_c(\overline\Theta)$  (with respect to the $C^1(\overline\Theta)$ norm).
Suppose that
\begin{itemize}
\item[i)] the map
$[0,1]\ni t\mapsto \displaystyle\int_{\overline\Theta}\psi(\theta)\mu_t(\ud\theta)$ is absolutely continuous for any $\psi\in \mathcal D$,
\item[ii)] 
$\Psi\in L^1(0,1)$, where 
$$
\Psi(t):=\sup\left\{\displaystyle\frac{\ud}{\ud s}\int_{\overline\Theta}\psi(\theta)\,\mu_s(\ud\theta){\Bigg{|}_{s=t}}:\quad \psi\in \mathrm{span}\mathcal D,\; \int_{\overline\Theta} |\nabla\psi(\theta)|^2\mu_t(\ud\theta)\le 1\right\}.
$$
\end{itemize}
Then, for a.e. $t\in (0,1)$, there exists a unique vector field $\mathbf{w}_t\in \overline{\{\nabla\psi: \psi\in C^1_c(\overline\Theta)\}}^{L^2_{\mu_t}(\overline\Theta;\rd)}$ which is solution to 
\begin{equation}\label{mixed}
\langle \mathbf{w}_t, \nabla\psi\rangle_{L^2_{\mu_t}(\overline\Theta;\rd)}=\frac{\ud}{\ud s}\int_{\overline\Theta}\psi(\theta)\,\mu_s(\ud\theta){\Bigg{|}_{s=t}}\qquad\forall \psi\in \mathrm{span} \mathcal D.
\end{equation}
Moreover, $\|\mathbf{w}_t\|_{L^2_{\mu_t}(\overline\Theta;\rd)}=\Psi(t)$ holds for a.e. $t\in(0,1)$, $\mu_t\in \mathcal P_2(\overline \Theta)$ for any $t\in(0,1)$ and 
\begin{equation}\label{BeBr}
\mathcal W_2(\mu_{t_1},\mu_{t_2})\le \int_{t_1}^{t_2} \|\mathbf{w}_t\|_{L^2_{\mu_t}(\overline\Theta;\rd)}\,dt
\qquad\mbox{for any $\;0\le t_1<t_2\le 1$}.
\end{equation}
\end{theorem}

\begin{proof}
 We preliminarily notice  that assumption {\it i)} implies that the curve $[0,1]\ni t\mapsto\mu_t$ is narrowly continuous, in view of the Portmanteau theorem (see, e.g. \cite[Section 2]{bill2}). Moreover, since $\mathcal D$ is countable, there exists a  $\mathcal{L}^1$-null set $N\subset (0,1)$ such that the mapping $t\mapsto\int_\Theta\psi(\theta)\,\mu_t(d\theta)$ is differentiable at $t\in (0,1)\setminus N$ for any $\psi\in \mathrm{span}\mathcal D$. Then the supremum in assumption {\it ii)} is well defined (and nonnegative) for every $t$ in $(0,1)\setminus N$, hence for a.e. $t\in(0,1)$.

A gradient vector field in $\{\nabla\psi: \psi\in \mathrm{span}\mathcal D\}$ admits a potential on $\overline\Theta$ which is unique up to a constant, therefore by mass conservation we see that for a.e. $t\in(0,1)$
\[
\mathcal{T}_t[\nabla \psi]:=\frac{\ud}{\ud s}\int_{\overline\Theta} \psi(\theta)\mu_s(\ud\theta){\Bigg{|}_{s=t}}
\]
defines indeed a linear functional on $\{\nabla\psi: \psi\in \mathrm{span}\mathcal D\}$. Moreover, since for any $\psi\in C^1_c(\overline\Theta)$ and any $t\in(0,1)$ there holds
$
\|\nabla \psi\|_{L^2_{\mu_t}(\overline\Theta;\rd)}\le\sup_{\overline\Theta}|\nabla\psi|,
$
we see that $\{\nabla\psi:\psi\in \mathrm{span}\mathcal D\}$ is dense in the $L^2_{\mu_t}(\overline\Theta;\rd)$ closure of the linear space $\{\nabla\psi: \psi\in C^1_c(\overline\Theta)\}$ for any $t\in(0,1)$.
Therefore, assumption {\it ii)} shows that, for a.e. $t\in(0,1)$, the operator $\mathcal{T}_t$ uniquely extends to a bounded linear functional on  the space $\overline{\{\nabla\psi: \psi\in C^1_c(\overline\Theta)\}}^{L^2_{\mu_t}(\overline\Theta;\rd)},$ where we find by Riesz representation theorem a unique vector field $\mathbf{w}_t(\cdot)$  such that
\begin{equation*}\label{1variable}
\mathcal{T}_t[\nabla\psi]=\int_{\overline\Theta}\langle \mathbf{w}_t(\theta),\nabla\psi(\theta)\rangle\,\mu_t(\ud\theta)\qquad \forall{\psi\in \mathrm{span}\mathcal D}
\end{equation*}
and such that $\|\mathbf{w}_t\|_{L^2_{\mu_t}(\overline\Theta;\rd)}=\Psi(t)$, thus $\mathbf{w}_t$ is the desired  solution to \eqref{mixed}.

At this stage, it is possible to prove (we refer to  \cite[Theorem 8.3.1]{AGS}, and we also include a proof in Lemma \ref{technical} in the Appendix) that there exists of a Borel map $[0,1]\times\overline\Theta\ni (t,\theta)\mapsto\mathbf{v}(t,\theta)\in\rd$ such that 
$\|\mathbf{v}(t,\cdot)\|_{L^2_{\mu_t}(\overline\Theta;\rd)}=\|\mathbf{w}_t\|_{L^2_{\mu_t}(\overline\Theta;\rd)}$ for a.e. $t\in(0,1)$ and such that the couple $(\mathbf{v}(t,\cdot),\mu_t)$ satisfies the continuity equation \eqref{distributionalcontinuityequation}. As a consequence, we may invoke the  Benamou-Brenier formula (see \cite[Theorem 8.3.1]{AGS} and \cite[Proposition 3.30]{AG}) to get $\mu_t\in \mathcal P_2(\overline\Theta)$ for every $t\in(0,1)$ and the validity of the estimate \eqref{BeBr}.
\end{proof}

\begin{remark}[Tangency condition]
\rm
Following \cite[Section 8.4]{AGS}, we define the tangent space to a measure $\mu$ in $\mathcal{P}_2(\overline\Theta)$ by  
$$
\mathcal{TAN}_\mu(\mathcal{P}_2(\overline\Theta)):=\overline{\{\nabla\psi: \psi\in C^1_c(\overline\Theta)\}}^{L^2_{\mu}(\overline\Theta;\rd)}.
$$ 
Therefore, in Theorem \ref{abs}, we conclude that $\mathbf{w}_t\in\mathcal{TAN}_{\mu_t}(\mathcal{P}_2(\overline\Theta))$ for a.e. $t\in(0,1)$. This is equivalent to saying that $\mathbf{w}_t$ has minimal $L^2_{\mu_t}(\overline\Theta;\rd)$ norm among $L^2_{\mu_t}(\overline\Theta;\rd)$ solutions  to \eqref{mixed}, see also \cite[Section 3.3.2]{AG}. 
\end{remark}

We notice that $\mu_t$ can be either supported on the whole of $\overline\Theta$ or on a subset which possibly depends on $t$.  
The rest of this section is devoted to a further analysis of the case of mobile support, starting with  some more  definitions and notation.

\begin{defi}[\bf Regular motion]\label{motion}
Let  $\Theta_* \subseteq\rd$ be a nonempty open connected set with locally Lipschitz boundary. We say that a smooth mapping $[a,b]\times\Theta_*\ni (t,\theta)\mapsto \Phi_t(\theta)\in\rd$ is a {\it regular motion} in 
$\Theta$ if the following conditions hold. For any $t\in[a,b]$, $\Phi_t$ is a diffeomorphism  between  $\Theta_*$ and a nonempty open connected set with locally Lipschitz boundary $\Theta_t:=\Phi_t(\Theta_*)\subseteq\Theta$. Further,  
there exist positive constants $k_1,k_2$ such that for any $t\in[a,b]$ and any $\theta\in\Theta_*$
\[
|\partial _t\Phi_t(\theta)|+|\nabla\Phi_t(\theta)|+|\nabla \partial_t\Phi_t(\theta)|\le k_2 \qquad\mbox{and}\qquad k_1\le \det\nabla\Phi_t(\theta).
\]
\end{defi}

We notice that under the assumptions of Definition \ref{motion}, $\Theta_*$ is bounded if and only if $\Theta_t$ is bounded for every $t\in[a,b]$.
A typical  example  of a family of diffeomorphisms that yields a regular motion  is $\Phi_t(\theta)=\theta+t\mathbf{v}(\theta)$, where $\mathbf{v}\in W^{1,\infty}(\Theta_*)\cap C^1(\Theta_*)$, $t\in[0,1]$ and $\sup_{\theta\in\Theta_*} |\nabla\mathbf{v}(\theta)|<1.$

We next apply the above definition to positivity sets of probability densities.
In view of the next definition, we say that $f\in ACL([0,1]\times\Theta)$ (in short, that $f$ has the ACL property) if for every coordinate direction $\nu$ of $\R\times \rd$ and for $\mathcal L^d$-almost any line $\ell_\nu$ in the direction of $\nu$, $f$ is absolutely continuous on any closed segment contained in $\ell_\nu\cap([0,1]\times\Theta)$.  More details about the ACL property will be given in the next subsection. Furthermore, we will denote by 
$\mathds 1_A$ the indicator function of a set $A\subseteq\Theta$ (i.e., $\mathds 1_A(\theta)$ is equal to $1$ if $\theta\in A$ and it is equal to $0$ otherwise).

\begin{defi} \label{densityassumption}
Let $(t,\theta)\mapsto g_t(\theta)\in\mathbb R$
be a nonnegative $ L^1_{loc}((0,1)\times\Theta)$ function such that $\int_\Theta g_t(\theta)\,\ud\theta=1$ for a.e. $t\in(0,1)$. 
We say that it admits a regular extension if the following conditions are satisfied: 
\begin{itemize}  \item[\it i)] for a.e. $t$ in $(0,1)$, the positivity set   $\{\theta\in\Theta: g_t(\theta)>0\}$ coincides (up to a $\mathcal L^d$-null set)  with a nonempty open connected set    $\Theta_t=\Phi_t(\Theta_*)\subseteq\Theta$ with locally Lipschitz boundary,  where $[0,1]\times\Theta_*\ni (t,\theta)\mapsto\Phi_t(\theta)$ is a regular motion according to {\rm Definition \ref{motion}};
\item[\it ii)] there exists a  $W^{1,1}((0,1)\times\Theta)\cap ACL([0,1]\times\Theta)$ function $[0,1]\times\Theta\ni(t,\theta)\mapsto \tilde g_{t}(\theta)\in\R$ such that 
$$
g_t(\theta)=\tilde g(t,\theta)\,\mathds 1_{\Theta_t}(\theta)\qquad\mbox {for $(\mathcal L^1\otimes\mathcal L^d)$-a.e. $(t,\theta)\in(0,1)\times\Theta$}.
$$
\end{itemize} 
\end{defi}

As a consequence of the latter definition, we notice that $\partial_t\tilde g_t\in L^1(\Theta)$ for a.e. $t\in (0,1)$ and $g_t\in W^{1,1}(\Theta_t)$ for a.e. $t\in(0,1)$, with a $L^1_{loc}$ trace on $\partial \Theta_t$ thanks to the standard characterization \cite{G} of traces of $W^{1,1}$ functions,    see for instance \cite[Chapter 15]{L}. 

The next result provides an estimate of the Wasserstein distance in terms of weak solutions to Neumann boundary value problems (on time-dependent domain).  This is an important step towards the proof of Theorem \ref{concretemixed} which will be provided in Subsection \ref{proofs}.
 Indeed,
to a regular extension of a function $g\in L^1_{loc}((0,1)\times\Theta)$ according to Definition \ref{densityassumption}
 we associate the family (parametrized by $t$) of Neumann boundary value problems
\begin{equation}\label{neumanng0}\left\{\begin{array}{ll}
-\mathrm{div}(g_t\nabla u_t)=\partial_t \tilde g_t\quad &\mbox{ in }\Theta_{t}\\
g_t\nabla u_t\cdot \mathbf{n}_t=g_t\,{\partial_t\Phi_t}\circ \Phi_t^{-1}\cdot\mathbf{n}_t\quad &\mbox{ on }\partial\Theta_{t},\\
\end{array}\right.
\end{equation}
Here, for given  $t$,  $\mathbf{n}_t$  denotes the ($\mathcal{H}^{d-1}$- a.e. existing on $\partial\Theta_t$) outer unit normal to $\partial\Theta_t$.
For such moving domains, a natural calculus tool is the Reynolds transport formula (see Lemma \ref{reynolds} in the Appendix): notice that $\partial_t\Phi_t\circ\Phi_t^{-1}$ represents  the velocity of the boundary. It is possible that such velocity vanishes on some part of the boundary or on the whole boundary (in particular, if $\Phi_t$ does not depend on $t$, then the domain is fixed, i.e., $\Theta_t\equiv\Theta$ for any $t$). 
In fact, by means of  Definition \ref{densityassumption} we require two properties: a regularly moving domain and the existence of a global Sobolev extension. Such properties will ensure the applicability of the Reynolds transport formula from Lemma \ref{reynolds} (which also implies the standard compatibility condition for the Neumann problem \eqref{neumanng0}, thanks to the mass conservation property $\tfrac{\ud}{\ud t}\int_{\Theta_t} g_t(\theta)\,\ud\theta=0$).
 For given $t$, a weighted Sobolev space on $\Theta_t$ (with weight $g_t$ that is positive $\mathcal L^d$-a.e. on $\Theta_t$) is the natural framework for a notion of weak solution to  problem \ref{neumanng0}.
 Moreover, if $\Theta_t$ is bounded we complement  \ref{neumanng0} with the null mean condition $\int_{\Theta_t}u_t(\theta) g_t(\theta)\,\ud\theta=0$ (instead if $\Theta_t$ is unbounded we complement \eqref{neumanng0} with a vanishing condition at infinity). 
Therefore, if $\Theta_t$ is bounded we let $C^1_{g_t}(\overline \Theta_t)$ be the space of functions $\psi\in C^1(\overline \Theta_t)$ such that $\int_{\Theta_t}\psi(\theta)g_t(\theta)\,\ud\theta=0$, while if $\Theta_t$ is unbounded we just let $C^1_{g_t}(\overline\Theta_t):=C^1_c(\overline\Theta_t)$.  We give the following
 
\begin{defi}[\bf Weak solution]\label{soluzionedebole} 
Let $g$ satisfy all the conditions in {\rm Definition \ref{densityassumption}}. Fix $t\in (0,1)$ such that $g_t\in W^{1,1}(\Theta_t)$ and $\partial_t\tilde g_t\in L^1(\Theta)$. 
The weighted Sobolev space $H^1(\Theta_t, g_t)$ is then defined as the completion of $C^1_{g_t}(\overline{\Theta}_t)$ 
w.r.t. the norm 
$
\|\psi\|_{H^1(\Theta_t, g_t)} := \big( \int_{\Theta_t}|\nabla \psi(\theta)|^2\,g_t(\theta)\,\ud\theta {\big)}^{1/2}.
$ 
We say that $u_t\in  H^1(\Theta_t,g_t)$ is a weak solution to  problem \eqref{neumanng0} if for every  $\psi\in C^{1}_{g_t}(\overline\Theta_t)$ there holds
\[
\int_{\Theta_t}\nabla\psi(\theta)\cdot\nabla{u}_t(\theta) g_t(\theta)\,\ud\theta=\int_{\Theta_t} \psi(\theta)\partial_t\tilde g_t(\theta)\,\ud\theta+\int_{\partial\Theta_t}\psi(\sigma)g_t(\sigma) \partial_t\Phi_t(\Phi_t^{-1}(\sigma))\cdot \mathbf{n}_t(\sigma)\,\mathcal{H}^{d-1}(\ud\sigma).
\]
\end{defi}

\begin{theorem}\label{abstractmixed}
Let $g$ satisfy all the conditions in {\rm Definition \ref{densityassumption}}. For any $t\in[0,1]$, let $\mu_t:=\tilde g_t\,\mathcal L^d\mr\Theta_t$ and suppose that $\mu_0\in \mathcal P_2(\overline\Theta)$ and 
$\mu_1\in\mathcal P_2(\overline\Theta)$. Suppose that, for a.e. $t\in(0,1)$, $u_t\in H^1(\Theta_t,g_t)$
 is a weak solution to problem \eqref{neumanng0}, and that the map $(0,1)\ni t\mapsto \left(\int_{\Theta}|\nabla u_t(\theta)|^2\,g_t(\theta)\,\ud\theta\right)^{1/2}$ belongs to $L^1(0,1)$.
Then, there hold $\mu_t\in \mathcal P_2(\overline\Theta)$ for any $t\in(0,1)$,  $\nabla u_t\in\mathcal{TAN}_{\mu_t}(\mathcal{P}_2(\overline\Theta))$ for a.e. $t\in(0,1)$ and
\begin{equation}\label{bbg}
\mathcal W_2(\mu_{t_1},\mu_{t_2}) \le\int_{t_1}^{t_2}\left(\int_\Theta|\nabla u_t(\theta)|^2\,g_t(\theta)\,\ud\theta\right)^{\frac12}\,\ud t \qquad\mbox{for any $\;0\le t_1<t_2\le 1$}.
\end{equation}
\end{theorem}
\begin{proof} 
Let $\mathcal D$ be a countable dense subset of $C^1_c(\overline\Theta)$ (in the $C^1(\overline\Theta)$ norm).
By Lemma \ref{reynolds} in the Appendix, $\tilde g$ is such that 
$
[0,1]\ni t\mapsto \int_{\Theta_t}\psi(\theta)\,\tilde g_t(\theta)\,\ud\theta=\int_\Theta\psi(\theta)\,\mu_t(\ud\theta)
$ 
is absolutely continuous, and since $\mathcal D$ is countable, the null set $N\in(0,1)$ of its nondifferentiability points can be assumed to be independent on $\psi\in\mathrm{span}\mathcal D$ (in particular, assumption    {\it i)} of Theorem \ref{abs} is satisfied by the Borel family $\{\mu_t\}_{t\in[0,1]}\subset\mathcal P(\overline\Theta)$). Moreover, with the notation
\begin{equation*}\label{barra}\overline\psi_t(\cdot):=\left\{
\begin{array}{ll}\psi(\cdot)-\int_{\Theta}\psi(\theta) g_t(\theta)\,\ud\theta&\quad \mbox{if $\Theta_t$ is bounded}\\
\psi(\cdot)&\quad\mbox{otherwise}
\end{array}\right. 
\end{equation*}
we have $\overline\psi_t\in C^1_{g_t}(\overline\Theta_t)$ and, for any $t\in(0,1)\setminus N$ and any $\psi\in\mathrm{span}\mathcal D$, we have
\[
\frac{\ud}{\ud r}\int_{\Theta} \psi(\theta)\,\mu_t(\ud\theta){\Bigg{|}_{r=t}}= 
\frac{\ud}{\ud r}\int_{\Theta_r} \psi(\theta)\tilde g(r,\theta)\,\ud\theta{\Bigg{|}_{r=t}} = \frac{\ud}{\ud r}\int_{\Theta_r} \overline\psi_t(\theta)\tilde g(r,\theta)\,\ud\theta{\Bigg{|}_{r=t}}.
\]
Therefore, still by making use of Lemma \ref{reynolds}, we apply Reynolds transport formula  we get 
\[\begin{aligned}
\frac{\ud}{\ud r}\int_{\Theta_r} \psi(\theta)\tilde g_r(\theta)\,\ud\theta{\Bigg{|}_{r=t}}
&=\int_{\Theta_t} \overline\psi_t(\theta)\partial_t\tilde g(t,\theta)\,\ud\theta+ \int_{\partial\Theta_t}\overline\psi_t(\sigma)g_t(\sigma) \partial_t\Phi_t(\Phi_t^{-1}(\sigma))\cdot \mathbf{n}_t(\sigma)\,\mathcal{H}^{d-1}(\ud\sigma)\\&=\int_{\Theta_t}\nabla\overline\psi_t(\theta)\cdot\mathbf{w}_t(\theta) g_t(\theta)\,\ud\theta=\int_{\Theta_t}\nabla\psi(\theta)\cdot\mathbf{w}_t(\theta) g_t(\theta)\,\ud\theta
\end{aligned}\]
  for any $t\in(0,1)\setminus N$ and any $\psi\in\mathrm{span}\mathcal D$,
where we used the fact that  $u_t$ is a weak solution to \eqref{neumanng0} and the notation $\mathbf{w}_t:=\nabla u_t$. By assumption we have 
$u_t\in H^1(\Theta_t,g_t)$ for a.e. $t\in (0,1)$, thus 
$
\mathbf{w}_t\in \overline{\{\nabla\psi:\psi\in C^1_c({\overline\Theta_t})\}}^{L^2_{g_t}(\Theta_t;\rd)}.
$   
Since any $C^{1}_c(\overline{\Theta}_t)$ function can be extended to a function in  $C^1_c(\overline\Theta)$, by truncation and extension the spaces 
$
\overline{\{\nabla\psi:\psi\in C^1_c(\overline{\Theta}_t)\}}^{L^2_{g_t}(\Theta_t;\rd)}$ and $\overline{\{\nabla\psi:\psi\in C^1_c(\overline\Theta)\}}^{L^2_{g_t}(\Theta;\rd)}
$
are isometric. 
Thus, for a.e. $t\in(0,1)$, $\mathbf{w}_t$ is the unique solution in 
$\overline{\{\nabla\psi:\psi\in C^1_c(\overline\Theta)\}}^{L^2_{g_t}(\Theta;\rd)}
$
  to problem \eqref{mixed} and by Riesz isomorphism it satisfies $\|\mathbf{w}_t\|_{L^2_{g_t}(\Theta;\rd)}=\Psi(t)$ for a.e. $t\in (0,1)$, where $\Psi$ is defined in Theorem \ref{abs}. Hence, condition {\it ii)} of Theorem \ref{abs} is also satisfied. Therefore,  
\eqref{bbg} follows along with  $\mathbf{w}_t\in\mathcal{TAN}_{\mu_t}(\mathcal{P}_2(\overline\Theta))$ for a.e. $t\in(0,1)$.
\end{proof}

\subsection{Basic estimates of $\ud_{TV}$, $\mathcal W_1$ and $\mathcal W_2$: proof of Theorem \ref{th:0} }\label{proofs}

In this subsection, we provide the proofs of the basic results on a finite dimensional sample space under the validity of Assumptions  \ref{ass:1}. 
Let us introduce some further notation. For $\nu\in\mathbb S^{m-1}$, let $ P_\nu$ be the projection operator onto $\{\nu\}^\perp:=\{z\in\mathbb{R}^m : z\cdot\nu=0\}$ 
and let 
$\X_\nu:= P_\nu(\X)$. We notice that $\X_\nu$ inherits the convexity of $\X$. We introduce the line segment $I_{\xi,\nu}:=\{\xi+t\nu:t\in\R\}\cap\X$, and cleary $I_{\xi,\nu}\neq\emptyset$ if $\xi\in\X_\nu$.
The following is the standard ACL characterization of Sobolev functions (see for instance \cite{bogachev,EG,H,L}):  $G\in W^{1,1}_{loc}(\X)$ if and only if it admits a representative such that, for any coordinate direction $\nu$ in $\reals^m$,  the restriction  to $I_{\xi,\nu}$ is locally absolutely continuous for $\mathcal L^{m-1}$-a.e. $\xi\in\X_\nu$. In such case, the ACL property holds with respect to any direction, and it can be rephrased as follows. Given $\nu\in\mathbb S^{m-1}$, for $\mathcal L^{m-1}$-a.e. $\xi\in\X_\nu$, the $L^1(0,1)$ map
$t\mapsto G(\mathbf s_{x_1,x_2}(t))$ is (up to having modified  $G$ on a $\mathcal  L^m$-null set) absolutely continuous on $(0,1)$, where $x,y$ are any two distinct points of $I_{\xi,\nu}$, and 
\begin{equation}\label{gradientformula}\begin{aligned}
\frac{\ud}{\ud t}\left( G(\mathbf s_{x_1,x_2}(t))\right) &= |x_1-x_2|\, \partial_{\nu} G(\mathbf s_{x_1,x_2}(t)) \\&= \nabla G(\mathbf s_{x_1,x_2}(t)) \cdot (x_2-x_1)\quad\mbox{for $\mathcal{L}^1$-a.e. $t \in (0,1)$.}
\end{aligned}\end{equation}
The weak $\nu$-directional derivative of $G$ coincides with the pointwise $\mathcal L^m$-a.e. classical $\nu$-directional derivative.  
Before the proof of the main theorems, we state the following simple lemma.
\begin{lem}\label{Ggrande}
Let ${\mathbb Y}\subset\reals^m$ be open.  Let  $\psi\in L^\infty_\pi(\Theta)$. Let $g\in L^1_{\mathcal L^m\otimes \pi}({\mathbb Y}\times\Theta)$. If $\int_\Theta \|g(\cdot, \theta)\|_{W^{1,1}({\mathbb Y})}\, 
\pi(\ud\theta) < +\infty$, then  
\begin{equation}\label{Gipsy}
G_{\psi}(\cdot) := \int_\Theta \psi(\theta) g(\cdot, \theta) \,\pi(\ud\theta)
\end{equation} 
belongs to $W^{1,1}({\mathbb Y})$ and 
\begin{equation} \label{gradx}
\nabla_x G_{\psi}(x) = \int_\Theta \psi(\theta) \, \nabla_x  g(x, \theta) \, \pi(\ud\theta)\qquad \mbox{ for $\mathcal{L}^m$-a.e. $x \in {\mathbb Y}$. }
\end{equation}
\end{lem}
\begin{proof}
By the assumptions, $g\in L^1_{\mathcal L^m\otimes \pi}({\mathbb Y}\times\Theta)$, $\nabla_x g \in L^1_{\mathcal L^m\otimes\pi}({\mathbb Y}\times\Theta)$ 
and for $\pi$-a.e. $\theta\in\Theta$  the mapping $x\mapsto g(x,\theta)$ belongs to $W^{1,1}(\mathbb Y)$. We apply Fubini's theorem to get 
\[\begin{aligned}
\int_{{\mathbb Y}}G_\psi(x)\nabla_x \zeta(x)\,\ud x &= \int_\Theta \psi(\theta)\left(\int_{{\mathbb Y}} g(x,\theta)\nabla_x\zeta(x)\,\ud x\right)\,\pi(\ud\theta)\\&= -\int_\Theta\psi(\theta)\left(\int_{{\mathbb Y}}\nabla_x g(x,\theta)\zeta(x)\,\ud x\right)\,\pi(\ud\theta) \\
&=-\int_{{\mathbb Y}} \zeta(x)\left(\int_\Theta \psi(\theta)\nabla_x g(x,\theta)\,\pi(\ud\theta)\right)\,\ud x
\end{aligned}\]
for any $\zeta\in C^{\infty}_c({\mathbb Y})$ and
$$
\left|\int_{{\mathbb Y}}\left(\int_\Theta\psi(\theta)\nabla_x g(x,\theta)\,\pi(\ud\theta)\right)\,\ud x\right|\le \|\psi\|_{L^\infty_\pi(\Theta)}\int_\Theta\|g(\cdot,\theta)\|_{W^{1,1}({\mathbb Y})}\,\pi(\ud\theta)<+\infty.
$$
Therefore, the right-hand-side in \eqref{gradx}  belongs to $L^1({\mathbb Y})$ and it is the weak gradient of $G_\psi$.
 \end{proof}

We proceed to the proof of the  main results. We start with the most direct proof concerning the estimate in total variation distance.   We also refer to   \cite{DM} for further results involving the total variation distance.
We recall the dual formulation of the total variation distance. For $\mu,\nu\in\mathcal{P}(\overline\Theta)$ there holds
\begin{equation*}\label{tvdual}
\ud_{TV}(\mu,\nu)=
\frac12 \sup_{{\psi\in C_c(\overline\Theta)}\atop{|\psi|\le 1}}\left(\int_{\overline\Theta}\psi(\theta)\,\mu(\ud\theta)-\int_{\overline\Theta}\psi(\theta)\,\nu(\ud\theta)\right),
\end{equation*}
where $C_c(\overline\Theta)$ is the set of continuous functions on $\overline\Theta$ having compact support contained in $\overline\Theta$.
We notice that due to the separability of $C_c(\overline\Theta)$ it is possible to compute the above supremum on a countable dense subset (w.r.t. the sup norm).

 \begin{proofad}
We first claim that for any bounded continuous function $\psi$ on $\overline\Theta$, the function $G_\psi$ from \eqref{Gipsy}, which is in $W^{1,1}_{loc}(\mathbb X)$ by Lemma \ref{Ggrande},
belongs to $W^{1,\infty}({\mathbb{X}})$. Indeed, since $\pi(\Theta|x)=\int_\Theta g(x,\theta)\,\pi(\ud\theta)=1$ for $\mathcal L^{m}$-a.e. $x\in\X$,  
it is clear that $|G_\psi(x)|\le \sup_{\overline\Theta}|\psi|$ for $\mathcal L^{m}$-a.e. $x\in\X$.
By Lemma \ref{Ggrande} and by assumption, we get 
\begin{equation*}\label{pointzero}
|\nabla G_\psi(x)|\le \sup_{\overline\Theta}|\psi|\int_\Theta |\nabla_x g(x,\theta)|\,\pi(\ud\theta)=K\,\sup_{\overline\Theta}|\psi|
\end{equation*}
for $\mathcal L^{m}$-a.e. $x\in{\X}$. The claim is proved. In particular, for any $\psi\in C_c(\overline\Theta)$, $G_\psi$ has a Lipschitz representative on $\X$.

Let $\mathcal D$ denote a countable dense subset (in the sup norm) of $\{\psi\in C_c(\overline\Theta): |\psi|\le 1\,\mbox{on }\overline\Theta\}$. 
Let $\hat g$ be a representative (according to the $(\mathcal L^m\otimes\pi)$-a.e. identification) of $g$ such that $\int_\Theta \hat g(x,\theta)\,\pi(\ud\theta)=1$ for every $x\in\X$. Therefore, $\hat\pi(\ud\theta|x):=
\hat g(x,\theta)\,\pi(\ud\theta)$ is a representative of the kernel defined by \eqref{kernelgpi}, and $\hat G_\psi(x):=\int_\Theta\psi(\theta)\hat g(x,\theta)\,\pi(\ud\theta)$ is a representative of $G_\psi$ for any $\psi\in\mathcal D$. Moreover, for $\psi\in\mathcal D$, $\hat G_\psi $ agrees $\mathcal L^m$-a.e. with a Lipschitz function on $\X$, i.e., there exists a $\mathcal L^m$-null set $\Z_\psi\subset\X $ such that 
$|\hat G_\psi(x_2)-\hat G_\psi(x_1)|\le K|x_2-x_1|$ for any $x_1,x_2\in\X\setminus\Z_\psi$.
Since $\mathcal D$ is countable, there exists a $\mathcal L^m$-null set $\Z\subset\X$ such that  for every $\psi\in\mathcal D$ the restriction of $\hat G_\psi$ to $\X\setminus\Z$ is Lipschitz 
(with Lipschitz constant bounded by $K$).
Therefore
\[
\begin{aligned}
\ud_{TV}(\hat\pi(\cdot|x_2),\hat\pi(\cdot|x_1)) 
&=\frac12 \sup_{\psi\in \mathcal D }
(\hat G_\psi(x_2)-\hat G_\psi(x_1))\le \frac12 \sup_{\psi\in \mathcal D}  \|\nabla G_\psi\|_{L^\infty(\X)}\,|x_2-x_1|\\&	\le \frac K2  |x_2-x_1|
\end{aligned}
\] 
for any $x_1,x_2\in\X\setminus\Z$. Note that $\hat\pi(\cdot|x)\in\mathcal{P}(\Theta)$ for any $x\in\X\setminus\Z$. Since $\X\setminus\Z$ is dense in $\X$ and since $(\mathcal{P}(\Theta), d_{TV})$ is a complete metric space, the mapping $\X\setminus\Z\ni x\mapsto\hat \pi(\cdot|x)\in\mathcal{P}(\Theta)$ admits a unique Lipschitz continuous extension (with respect to the total variation distance) to the whole of $\X$ with the same 
Lipschitz constant $K/2$.
\end{proofad}

For the proof of Theorem \ref{th:0}-{\bf (ii)},
we take advantage of the Kantorovich-Rubinstein dual fomulation of the $1$-Wasserstein distance. See \cite[Section 1.2]{vilMass}.
For any $\mu,\nu\in\mathcal{P}_1(\overline\Theta)$, there holds
\begin{equation*}\label{dualw1}
\mathcal W_1(\mu,\nu)=
\sup_{{\psi\in C^1_c(\overline\Theta)}\atop{\mathrm{Lip}(\psi)\le 1}}\left(\int_{\overline\Theta}\psi(\theta)\,\mu(\ud\theta)-\int_{\overline\Theta}\psi(\theta)\,\nu(\ud\theta)\right).
\end{equation*}
Again, the separability of $C_c^1(\overline\Theta)$ allows to take the above supremum on a countable dense set (in the sup norm).

\begin{proofadw1}
We claim that the function $G_\psi$ from \eqref{Gipsy} belongs to $W^{1,\infty}({\mathbb{X}})$ for any $\psi\in C^1_c(\overline\Theta)$. 
Indeed, as seen in the proof of Theorem \ref{th:0}-{\bf (i)}, we have $|G_\psi(x)|\le \sup_{\overline\Theta}|\psi|$ for $\mathcal{L}^{m}$-a.e. $x\in\X$.
Moreover, by Lemma \ref{Ggrande}, by H\"older inequality and by the Poincar\'e inequality \eqref{wirtinger}, since $\nabla G_\psi=\nabla G_{\psi-a}$ for any $a\in\R$, we get 
\begin{equation*}\label{extraw1}
\begin{aligned}
|\nabla G_\psi(x)|&=\inf_{a\in\R}\left|\int_\Theta(\psi(\theta)-a)\,\nabla_x g(x,\theta)\,\pi(\ud\theta)\right|\\&\le\inf_{a\in\R}\left(\int_\Theta |\psi(\theta)-a|^q \,\pi(\ud\theta)\right)^{\!\!\frac1q}\left(\int_\Theta{|\nabla_x g(x,\theta)|^p}\,\pi(\ud\theta)\right)^{\!\!\frac1p}\\
&\le \mathcal C_q[\pi] \left(\int_\Theta |\nabla\psi(\theta)|^q \,\pi(\ud\theta)\right)^{\frac1q} \left(\int_\Theta{|\nabla_x g(x,\theta)|^p}\,\pi(\ud\theta)\right)^{\!\!\frac1p}\le K\,\mathrm{Lip}(\psi)
\end{aligned}
\end{equation*}
for $\mathcal L^{m}$-a.e. $x\in\X$, thus proving the claim.

Let $\mathcal D$ be a countable dense subset (in  the sup norm) of $\{\psi\in C^1_c(\overline\Theta): \mathrm{Lip}(\psi)\le 1\}$.  
By the same argument as in the proof of Theorem \ref{th:0}-{\bf (i)}, we obtain a $\mathcal L^m$-null set $\Z$ in $\X$ and a representative (still denoted by $\pi(\cdot|\cdot)$) of the kernel defined by 	
\eqref{kernelgpi} such that, for any $x_1,x_2\in\X\setminus\Z$
\[
\begin{aligned}
\Wuno(\pi(\cdot|x_1),\pi(\cdot|x_2))
\le \sup_{\psi\in \mathcal D} \|\nabla G_\psi\|_{L^\infty(\X)}\,|x_2-x_1| \le K|x_2-x_1|.
\end{aligned}
\]
Since $\X\setminus\Z$ is dense in $\X$ and since $(\mathcal{P}_1(\overline\Theta),\Wuno)$ is complete, there exists a unique map 
$\X\ni x\mapsto \pi^*(\cdot|x)\in\mathcal{P}_1(\overline\Theta)$ that satisfies the above Lipschitz estimate on the whole of $\X$, with the same Lipscthitz constant $K$, and such that $\pi^*(\cdot|x)\equiv\pi(\cdot|x)$ for any $x\in\X\setminus\Z$.  Since the assumptions of Theorem \ref{th:0} are also satisfied, $x\mapsto \pi^*(\cdot|x)$ is also continuous with respect to the total variation distance, therefore $\pi^*(\overline\Theta|x)=\pi^*(\Theta|x)=1$ for any $x\in\X$. 
\end{proofadw1}
\begin{proofad0}
Once more, we start by claiming that, for any  $\psi\in C^1_c(\overline\Theta)$, the function
$G_\psi$ from \eqref{Gipsy} belongs to $W^{1,\infty}(\X)$. Indeed, we have as usual $\|G_\psi\|_{L^\infty(\X)}\le\sup_{\overline\Theta}|\psi|$
and again by Cauchy-Schwarz inequality and by \eqref{wirtinger}, by the positivity of $g$ and by assumption, we get
\[\begin{aligned}
|\nabla_x G_\psi(x)| &=\inf_{a\in\R} \left|\int_\Theta(\psi(\theta)-a)\,\nabla_x g(x,\theta)\,\pi(\ud\theta)\right|
\\&\le\inf_{a\in\R}\left(\int_\Theta |\psi(\theta)-a|^2 g(x,\theta)\,\pi(\ud\theta)\right)^{\frac12}\left(\int_\Theta\frac{|\nabla_x g(x,\theta)|^2}{g(x,\theta)}\,\pi(\ud\theta)\right)^{\frac12} \nonumber \\
&\le \mathcal C[g(x,\cdot)\pi] \left(\int_\Theta |\nabla\psi(\theta)|^2 g(x,\theta)\,\pi(\ud\theta)\right)^{\frac12}\mathcal J_{\pi}[g(x,\cdot)]\\&\le K \left(\int_\Theta |\nabla\psi(\theta)|^2 g(x,\theta)\,\pi(\ud\theta)\right)^{\frac12}\le K\mathrm{Lip}(\psi)\label{gradzetagii}
\end{aligned}\]
for $\mathcal L^{m}$-a.e. $x\in\X$, thus proving  the claim. As seen in the proof of Theorem \ref{th:0}-{\bf (ii)}, this shows that there exists a $\Wuno$-Lipschitz map $\X\ni x\mapsto\pi^*(\cdot|x)\in\mathcal{P}_1(\overline\Theta)$, with Lipschitz constant $K$, which is a version of the  kernel defined by \eqref{kernelgpi}.   
We are left to check that $\pi^*(\cdot|x)$ is Lipschitz with respect to $\Wdue$ as well.  Note that by assumption the second moment of $\pi^*(\cdot|x)$ is finite for $\mathcal L^m$-a.e. $x\in \X$.

We first notice that $G_\psi^*(x):=\int_{\overline\Theta}\psi(\theta)\,\pi^*(\ud\theta|x)$ is the Lipschitz-continuous representative of $G_\psi$, for any  $\psi\in C^1_c(\overline\Theta)$. Indeed, the $\Wuno$-Lipschitz 
estimate entails 
\[
\begin{aligned}
|G^*_\psi(x_1)-G_\psi^*(x_2)|&=\left|\int_{\overline\Theta}\psi(\theta)\,\pi^*(\ud\theta|x_1)-\int_{\overline\Theta}\psi(\theta)\,\pi^*(\ud\theta|x_2)\right|\\&\le\int_{\overline\Theta\times\overline\Theta}|\psi(\theta_1)-\psi(\theta_2)|\,\eta_{x_1,x_2}(\ud\theta_1\ud\theta_2)\\
&\le \mathrm{Lip}(\psi)\,\Wuno(\pi^*(\cdot|x_1),\pi^*(\cdot|x_2)))\le K\, \mathrm{Lip}(\psi)\,|x_1-x_2|
\end{aligned}
\]
for any $x_1,x_2\in\X$, where $\eta_{x_1,x_2}\in\mathcal{P}(\overline\Theta\times\overline\Theta)$ is an optimal coupling between $\pi^*(\cdot|x_1) $ and 
$\pi^*(\cdot|x_2)$ for the $1$-Wasserstein distance.
In particular for any $x_1,x_2\in\X$, the map $[0,1]\ni t\mapsto G^*_\psi(\mathbf s_{x_1,x_2}(t))$ is absolutely continuous for any $\psi\in C^1_c(\overline\Theta)$, so that assumption {\it i)} of Theorem \ref{abs} is satisfied by the narrowly continuous curve $[0,1]\ni t\mapsto\pi^*(\cdot|\mathbf s_{x_1,x_2}(t))\in\mathcal{P}(\overline\Theta)$. 

Let $\mathcal D$ be a countable dense subset of $C^1_c(\overline\Theta)$ (in the $C^1(\overline{\Theta})$ norm). Let $\nu\in \mathbb S^{m-1}$.
We take advantage of the fact that any $\mathcal L^m$-null subset $\Z$ of $\X$ has the following property (by Fubini theorem): for $\mathcal L^{m-1}$-a.e. $\xi\in\X_\nu$ there holds
$\mathcal L^1(\Z\cap I_{\xi,\nu})=0$, and the $\mathcal L^{m-1}$-null set of $\xi$'s where this property fails can be taken to be independent of $\psi\in\mathcal D$, since $\mathcal D$ is countable. Therefore, thanks to \eqref{gradzetagii},  for $\mathcal L^{m-1}$-a.e. $\xi\in\mathbb X_\nu$ and any $x_1,x_2\in I_{\xi,\nu}$, we have the following:
\[
\begin{aligned}
|\nabla_x G^*_\psi(\mathbf s_{x_1,x_2}(t))|&=\inf_{a\in\R} \left|\int_\Theta(\psi(\theta)-a)\,\nabla_x g(\mathbf s_{x_1,x_2}(t),\theta)\,\pi(\ud\theta)\right|\\&\le K \left(\int_\Theta |\nabla\psi(\theta)|^2 g(\mathbf s_{x_1,x_2}(t),\theta)\,\pi(\ud\theta)\right)^{\frac12}
\end{aligned}
\]
for a.e. $t\in(0,1)\setminus N$ and any $\psi\in\mathcal D$, where $N$ is a null set which is again independent of $\psi\in \mathcal D$. Moreover, for any $t\in(0,1)\setminus N$ the latter inequality also holds for any 
$\psi\in\mathrm{span}\mathcal D$, due to the linearity of $\psi\mapsto G^*_\psi(x)$.

As a consequence, for $\mathcal L^{m-1}$-a.e. $\xi\in\X_\nu$ and any $x_1,x_2\in I_{\xi,\nu}$, we have 
 \begin{equation*}\begin{aligned}
\frac{\ud}{\ud r}\int_{\overline\Theta}\psi(\theta)\,\pi^*(\ud\theta|\mathbf s_{x_1,x_2}(r)){\bigg|}_{r=t}&=\frac{\ud}{\ud r}(G^*_\psi(\mathbf s_{x_1,x_2}(r))){\bigg|}_{r=t} \le |x_1-x_2| \cdot |\nabla_x G^*_\psi(\mathbf s_{x_1,x_2}(t))|\\
&\le K|x_1-x_2| \left(\int_\Theta |\nabla\psi(\theta)|^2 g(x,\theta)\,\pi(\ud\theta)\right)^{\frac12}
\end{aligned}
\end{equation*}
for any $t\in(0,1)\setminus N$ and any $\psi\in \mathrm{span}\mathcal D$. Whence,
\[
\begin{aligned}
\Psi_{x_1,x_2}(t) &:= \sup_{\psi \in\mathrm{span} \mathcal D}\left\{\frac{\ud}{\ud r}\int_{\overline\Theta}\psi(\theta)\,\pi^*(\ud\theta|{\mathbf s_{x_1,x_2}(r)}){\bigg|}_{r=t}: \int_{\overline{\Theta}}|\nabla\psi(\theta)|^2\,
\pi^*(\ud\theta|g_{\mathbf s_{x_1,x_2}(t)} )\le1\right\}\\&\le K|x_1-x_2|
\end{aligned}
\]
for a.e. $t\in(0,1)$. Here, $(0,1)\ni t\mapsto\Psi_{x_1,x_2}(t)$ is measurable, being the supremum of the linear span of countably many measurable functions. Moreover, we deduce
from the latter estimate that $\int_0^1\Psi_{x,y}(t)\,dt\le K |x-y|$ for $\mathcal L^{m-1}$-a.e. $\xi\in\X_\nu$ and any $x_1,x_2\in I_{\xi,\nu}$.
In particular, for $\mathcal L^{m-1}$-a.e. $\xi\in\mathbb X_\nu$ and any $x_1,x_2\in I_{\xi,\nu}$, assumption {\it ii)} of Theorem \ref{abs} is satisfied by the curve $[0,1]\ni t\mapsto\pi^*(\cdot|\mathbf s_{x_1,x_2}(t))\in\mathcal{P}(\overline\Theta)$, and
we also notice that (since $\pi^*(\cdot|x)\in\mathcal{P}_2(\overline\Theta)$ for $\mathcal L^m$-a.e. $x\in\X$) we have $\pi^*(\cdot|\mathbf s_{x_1,x_2}(t))\in\mathcal{P}_2(\overline\Theta)$ for a.e. $t\in(0,1)$ up to another 
$\mathcal L^{m-1}$-null set of $\xi$'s in $\X_\nu$.
Therefore, by applying Theorem \ref{abs},  for  $\mathcal{L}^{m-1}$-a.e. $\xi\in\mathbb X_\nu$ and any $x_1,x_2\in I_{\xi,\nu}$ we get that both 
$\pi^*(\cdot|x_1)$ and $\pi^*(\cdot|x_2)$ belong to $\mathcal P_2(\overline\Theta)$ and that the bound in \eqref{main_problem} is fulfilled for such $x_1, x_2$, with the $\Wdue$ distance and with $L=K$.
By the arbitrariness of $\nu$ and by the narrow continuity of $\X\ni x\mapsto \pi^*(\cdot|x)\in \mathcal{P}(\overline\Theta)$,
the $\mathcal W_2$-Lipschitz estimate extends to any $x_1,x_2\in\X$.
Indeed, given generic $x_1,x_2\in\X$ with $x_1 \neq x_2$, letting $\nu:=\tfrac{x_2-x_1}{|x_2-x_1|}$, it is enough to take sequences $x_{1,n}\to x_1$ and $x_{2,n} \to x_2$ such that, for every $n\in\naturals$, 
$\tfrac{x_{2,n}-x_{1,n}}{|x_{2,n}-x_{1,n}|}=\nu$ and such that  \eqref{main_problem} applies for any couple of points on the line $I_{ P_{\nu}(x_{2,n}),\nu}$. Then, \eqref{main_problem} applies for the couple $x_{1,n}, x_{2,n}$, for any $n$, and it passes to the limit by the narrow lower semicontinuity of $\Wdue$, according to \cite[Proposition 7.13]{AGS}. 
\end{proofad0}
\begin{proofadS}
The proof is very similar to the previous ones.
We first show that $G_\psi$ from \eqref{Gipsy} belongs to $W^{1,\infty}(\mathbb X)$ for any  $\psi\in C^1_c(\overline\Theta)$. It belongs indeed to $W^{1,1}_{loc}(\mathbb X)$ by Lemma \ref{Ggrande}, and to $L^\infty(\mathbb X)$ with the same argument  of the proof of Theorem \ref{th:0}-{\bf (i)}.
Moreover, combining Lemma \ref{Ggrande}, the Sobolev inequality \eqref{sobineq} with critical exponent $p^*=\frac {dp}{d-p}$ ($p^*=+\infty$ if $p=d=1$), and H\"older's inequality, we get
\begin{equation}\label{inline}
\begin{aligned}
|\nabla_x G_\psi(x)|&\le\inf_{a\in\R}\int_\Theta|\psi(\theta)-a|\,|\nabla_x g(x,\theta)|\,\ud\theta\le\inf_{a\in\R}\|\psi-a\|_{L^{p^*}(\Theta)}\left\|\nabla_x g(x,\cdot)\right\|_{L^{\frac {p^*}{p^*-1}}(\Theta)}\\
&\le\mathcal  S_{p}(\Theta)\left(\int_\Theta|\nabla\psi(\theta)|^2\,g(x,\theta)\,\ud\theta\right)^{\frac12}\left\|\frac{1}{g(x,\cdot)}\right\|_{L^{\frac{p}{2-p}}(\Theta)}^{\frac12}
\left\|\nabla_x {g(x,\cdot)}\right\|_{L^{\frac {p^*}{p^*-1}}(\Theta)}
\end{aligned}
\end{equation}
for $\mathcal L^{m}$-a.e. $x\in\X$.
By assumption and by  \eqref{inline} we conclude that $\|\nabla_x G_\psi\|_{L^\infty({\X})}\le K\mathrm{Lip}(\psi)$. 
As seen in the proof of Theorem \ref{th:0}-{\bf (ii)}, it follows that the probability kernel $\pi(\cdot|\cdot)$ defined by \eqref{kernelgpi} admits a $\mathcal W_1$-Lipschitz continuous version $\X\ni x\mapsto \pi^*(\cdot|x)\in\mathcal{P}(\overline\Theta) $.  With the same argument of the proof of Theorem \ref{th:0}-{\bf (iii)}, 
the proof concludes by showing that $x\mapsto\pi^*(\cdot|x)\in\mathcal{P}_2(\overline\Theta)$ is also $\Wdue$-Lipschitz-continuous, with Lipschitz constant not exceeding $K$.
\end{proofadS}

\subsection{Moving domains: proof of {\rm Theorem \ref{concretemixed}} and of {\rm Theorem \ref{maintrace}}}
\label{sect:movingproof}

We deal with solutions to nonhomogeneous Neumann boundary value problems, following the line of Theorem \ref{abstractmixed}. 
Given a propability kernel $\pi(\cdot|\cdot)$, the curve $[0,1]\ni t\mapsto\pi(\cdot|{\mathbf s_{x_1,x_2}(t)})\in\mathcal{P}(\overline\Theta)$ depends on the two parameters $x_1,x_2$. Accordingly, we specify the notion of regular motion, which is essentially the same as Definition \ref{motion}. 
\begin{defi}[\bf $\X$-regular motion]\label{xmotion}
Let $\Theta_*\subseteq\rd$ and $\Theta_x\subseteq\Theta$ be nonempty open connected sets with locally Lipschitz boundary, for any $x\in\X$. We say that  a smooth mapping $\X\times\Theta_\ast\ni(x,\theta)\mapsto\Phi_x(\theta)$ is a $\X$-regular motion if $[0,1]\times\Theta_\ast\ni(t,\theta)\mapsto\Phi_{\mathbf s_{x_1,x_2}(t)}(\theta)$ is regular motion according to {\rm Definition \ref{motion}} for any $x_1,x_2\in\X$ and 
$\Theta_x=\Phi_x(\Theta_\ast)$ for any $x\in\X$. In such assumptions, we further define for any $x\in\X$ and any $\nu\in\mathbb S^{m-1}$ the function $\mathbf V_x^\nu:\Theta_x\to\rd$ (resp. $\mathbf V_x :\Theta_x\to\R^{d\times d}$) by $\mathbf V_x^\nu:=\partial_\nu\Phi_x\circ\Phi_x^{-1}$ (resp. $\mathbf V_x:=\nabla_x\Phi_x\circ\Phi_x^{-1}$).
\end{defi}

\begin{defi}[\bf Regular extension]\label{densityxy} 
Let $g \in L^1_{loc}({\X}\times\Theta)$ satisfy $\int_\Theta g(x,\theta)\,\ud\theta=1$ for a.e. $x\in\X$. We say that $g$ admits a regular extension if the following conditions are satisfied: 

\smallskip
   
\noindent{\it i)} 
there is a $\X$-regular motion $\Phi_x: \Theta_*\to\Theta_x$ according to {\rm Definition \ref{xmotion}} such that $\Theta_x\equiv\{g(x,\cdot)>0\}$ for $\mathcal L^{m}$-a.e. $x\in\X$; 

\smallskip

\noindent{\it ii)} there exists $\tilde g\in L^1_{loc}(\X\times\Theta)$ such that  $\tilde g\in W^{1,1}(\tilde{\X}\times\Theta)$ for any open set $\tilde{\X}$ compactly contained in $\X$  
and such that 
\begin{equation}\label{trunc}
\tilde g(x,\theta)\mathds 1_{\Theta_x}(\theta) = g(x,\theta)\qquad\mbox{ for $(\mathcal L^m\otimes\mathcal L^d)$-a.e. $(x,\theta)$ in $\X\times\Theta$}.
\end{equation} 
\end{defi}
Of course, for fixed $x$, the above identification $\equiv$ is understood up to $\mathcal L^d$-null sets of $\Theta$.
As $\tilde g$ from Definition \ref{densityxy} is in $W^{1,1}(\tilde{\X}\times \Theta)$, we shall use Sobolev regularity on linear submanifolds (see also \cite[Theorem 2.5.3]{bogachev}). 
We summarize some basic facts in the following proposition.
\begin{pro}\label{bogachev} 
Let $ \tilde g \in W^{1,1}(\tilde{\X}\times \Theta)$ for any open set $\tilde{\X}$ compactly contained in $\X$.
Let  $\nu \in\mathbb{S}^{m-1}$. 
For $\mathcal L^{m-1}$-a.e. $\xi\in\X_\nu$ and any two distinct points $x_1,x_2\in I_{\xi,\nu}$, the map $(t,\theta)\mapsto \tilde g(\mathbf s_{x_1,x_2}(t),\theta)$ belongs to $W^{1,1}((0,1)\times\Theta)$ and for 
$(\mathcal L^1\otimes\mathcal L^d)$-a.e. $(t,\theta)\in(0,1)\times\Theta$ there hold
\begin{equation}\label{onechain}
\begin{aligned}
&\frac{\ud}{\ud t}\tilde g(\mathbf s_{x_1,x_2}(t),\theta)= |x_2-x_1|\, \partial_\nu\tilde g(\mathbf s_{x_1,x_2}(t),\theta),\\&
\left|\frac{\ud}{\ud t}\, \tilde g(\mathbf s_{x_1,x_2}(t),\theta)\right| \le |x_2-x_1| \left|\nabla_x\, \tilde g(\mathbf s_{x_1,x_2}(t),\theta)\right|,
\end{aligned}
\end{equation}
and, in particular, $\partial_\nu\tilde  g(x,\cdot)\in L^1(\Theta)$ for $\mathcal{L}^1$-a.e. $x\in I_{\xi,\nu}$.
\end{pro}
\begin{proof} 
The ACL representative of $\tilde g$ (here not relabeled) has the ACL property on almost any lower dimensional hyperplane intersecting $\tilde{\X}\times \Theta$.
Let $\nu\in\mathbb S^{m-1}$;
for $\mathcal L^{m-1}$-a.e. $\xi\in\X_\nu$, the map $(t,\theta)\mapsto g(\xi+t\nu,\theta)$ belongs therefore to  $W^{1,1}((t_1,t_2)\times\Theta)$ for any $t_1<t_2$ such that $\xi+t_i\nu\in\mathbb X$, $i=1,2$.  The map 
$(t,\theta)\mapsto \tilde g(\mathbf s_{x_1,x_2}(t),\theta)$ belongs to $W^{1,1}((0,1)\times\Theta)$ as the composition of the latter with the segment parametrization $[0,1]\ni t\mapsto\mathbf s_{x_1,x_2}(t)$, where $x=\xi+t_1\nu$ and $y=\xi+t_2\nu$. Then, \eqref{onechain} follows from the fact that $\tilde g$ has classical $\nu$-directional derivative almost everywhere, coinciding with the scalar product of $\nu$ with the gradient.  
See for instance \cite[Theorem 4, pp. 200]{GMS}.
\end{proof}

Associated to a function $g$ as in Definition \ref{densityxy},
we consider the boundary value problem \eqref{neumanngintro}, where $\mathbf V_x^\nu:=\partial_\nu \Phi_x\circ (\Phi_x)^{-1}$.
The precise notation for \eqref{neumanngintro}  is the same of problem \eqref{neumanng0}, apart from the $\X$-valued index $x$ instead of $t$. Therefore, for those couples $x,\nu$ such that $g(x,\cdot)\in W^{1,1}(\Theta_x)$ and $\partial_\nu\tilde g(x,\cdot)\in L^1(\Theta_x)$, we may define a weak solution $u_x^{\nu}\in H^{1}(\Theta_x,g(x,\cdot))$ to \eqref{neumanngintro} by means of Definition \ref{soluzionedebole}.

\begin{proofadconcrete}  
Throughout this proof, for notational ease, we shorten the expression $\mathbf s_{x_1,x_2}(t)$ to $\mathbf s(t)$, whenever it is clear which couple $(x_1, x_2)$ we are referring to.  
We start by preliminarily observing that, given $\nu\in\mathbb S^{m-1}$, for $\mathcal L^{m-1}$-a.e. $\xi\in\X_\nu$ and any couple of distinct points $x_1,x_2\in I_{\xi,\nu}$, 
the map $(t,\theta)\mapsto \tilde g(\mathbf s(t),\theta)$ belongs to $W^{1,1}((0,1)\times\Theta)$, by Proposition \ref{bogachev}. This fact entails that 
$g(\mathbf s(t),\cdot) \in W^{1,1}(\Theta_{\mathbf s(t)})$ and $\tfrac{\ud}{\ud t} \tilde g(\mathbf s(t),\cdot) \in L^1(\Theta_{\mathbf s(t)})$ for a.e. $t\in (0,1)$.
Therefore, we may take advantage
of the notion of weak solution to problem \eqref{neumanng0} as given in Definition \ref{soluzionedebole}, with $g_t(\cdot)$ therein replaced by $g(\mathbf s(t),\cdot)$, and 
$\Phi_t$ therein replaced by $\Phi_{\mathbf s(t)}$.

The proof is an application of Theorem \ref{abstractmixed}, for almost every line in $\X$ in any given direction.
Indeed, let us consider an ACL representative of the regular extension $\tilde g$, that we still denote by $\tilde g$. Of course, combining the assumptions on $g$ with \eqref{trunc}, we have
 $\int_{\Theta_x}\tilde g(x,\theta)\,\ud\theta=1$ and $\int_{\Theta_x} |\theta|^2\tilde g(x,\theta)\,\ud\theta<+\infty$ for $\mathcal L^m$-a.e. $x\in\X$. 
Let $\nu\in\mathbb S^{m-1}$. For $\mathcal L^{m-1}$-a.e. $\xi\in\X_\nu$ and any $x_1,x_2\in I_{\xi,\nu}$,
 we apply Theorem \ref{abstractmixed} to obtain $\tilde\pi(\cdot|x_1)\in\mathcal P_2(\overline\Theta)$, $\tilde\pi(\cdot|x_2)\in\mathcal P_2(\overline\Theta)$ and 
\begin{equation}\label{pitilde}
\Wdue\left(\tilde \pi(\cdot|x_1),\tilde\pi(\cdot|x_2)\right)\le |x_1-x_2|\,\int_0^1\left(\int_\Theta |\nabla u^\nu_{\mathbf s(t)}(\theta)|^2\,g(\mathbf s(t),\theta)\,\ud\theta\right)^{\frac12}\,\ud t,
\end{equation}
where $\tilde\pi(\cdot|x):=\tilde g(x,\cdot)\,\mathcal L^d\mr\Theta_x$ is a representative of the kernel $\pi(\cdot|\cdot)$, in view of \eqref{trunc}.
We notice that the appearance of the factor $|x_1-x_2|$ is due to \eqref{onechain} and to the identity
\begin{equation}\label{newchain}
\frac{\ud}{\ud t}\left(\Phi_{\mathbf s(t)}(\theta)\right) = |x_1-x_2|\, \partial_\nu\Phi_{\mathbf s(t)}(\theta).
\end{equation}
As a consequence, for $\mathcal L^{m-1}$-a.e. $\xi\in\X_\nu$ and any couple of distinct points $x_1,x_2\in I_{\xi,\nu}$,
we get $\Wdue\left(\tilde \pi(\cdot|x_1),\tilde\pi(\cdot|y_2)\right)\le K |x_1-x_2|$.
The last inequality follows from \eqref{pitilde} and \eqref{nuovoess}: we bound once more the $\mathcal L^1$-essential supremum on $(0,1)$ with the $\mathcal L^m$-essential supremum on $\X$, for all but a 
$\mathcal L^{m-1}$-null set of lines in a given direction.

Now, let $\psi\in C^1_c(\overline\Theta)$ and
$
\tilde G_\psi(x):=\int_{\Theta_x}\psi(\theta)\tilde g(x,\theta)\,\ud\theta,
$
so that, by Definition \ref{densityxy}, we get $|\tilde G_\psi(x)|\le \sup_\Theta|\psi|$ for $\mathcal L^m$-a.e. $x\in\X$. 
By performing the same estimate of the proof of Theorem \ref{abstractmixed}, also taking \eqref{onechain} and \eqref{newchain} into account, we have the following:  given any $\nu\in\mathbb S^{m-1}$, for $\mathcal L^{m-1}$-a.e. $\xi\in\X_\nu$ and any couple of distinct points $x_1,x_2\in I_{\xi,\nu}$ \hbox{there holds}
\begin{equation}\label{quadruple}
\begin{aligned}
&\frac{\ud}{\ud r}\tilde G_\psi(\mathbf s(r)){\Bigg{|}_{r=t}} = |x_1-x_2| \int_{\Theta_{\mathbf s(t)}}\nabla\psi(\theta)\cdot \nabla u^\nu_{\mathbf s(t)}(\theta)\,g(\mathbf s(t),\theta)\,\ud\theta
\end{aligned}
\end{equation}
for a.e. $t\in(0,1)$, where we have used the definition of $u_z^\nu$ as solution to the boundary value problem \eqref{neumanngintro}.
Taking \eqref{gradientformula} into account, \eqref{quadruple} can be rephrased as follows: 
$
\nu\cdot \nabla _x \tilde G_\psi(x)= \int_{\Theta_x}\nabla\psi(\theta)\cdot \nabla u^\nu_x(\theta)\,g(x,\theta)\,\ud\theta 
$
for $\mathcal L^m$-a.e. $x\in\X$. Since $\int_{\Theta_x}\tilde g(x,\theta)\,\ud\theta=1$ for $\mathcal L^m$-a.e. $x\in\X$, we further estimate by the Cauchy-Schwarz inequality and \eqref{nuovoess}, to get
\[
\begin{aligned}
|\nu\cdot \nabla_x\tilde G_\psi(x)| 
& \le \left(\int_{\Theta_x} |\nabla\psi(\theta)|^2 g(x,\theta)\,\ud\theta\right)^{\frac12}\left(\int_{\Theta_x}|\nabla u^\nu_x(\theta)|^2 g(x,\theta)\,\ud\theta\right)^{\frac12}\le K\mathrm{Lip}(\psi)
\end{aligned}
\]
for $\mathcal L^m$-a.e. $x\in\X$. Here, $K$ is independent of $\nu$ by assumption, hence $|\nabla_x \tilde G_\psi(x)| \le K\mathrm{Lip}(\psi)$ for $\mathcal L^m$-a.e. $x\in\X$.
Note that $\tilde G_\psi$ is a representative of the $L^\infty(\X)$ function $G_\psi(\cdot):=\int_\Theta\psi(\theta)\,g(\cdot,\theta)\,\ud\theta$.
Having shown that $|G_\psi(x)|\le \sup_\Theta|\psi|$ and $|\nabla_x G_\psi(x)|\le K\mathrm{Lip}(\psi) $ for $\mathcal L^m$-a.e. $x\in\X$,
by the same argument as in the proof of Theorem \ref{th:0}-{\bf(ii)}, we obtain the existence of a $\mathcal W_1$-Lipschitz representative $\pi^*(\cdot|\cdot)$ for the probability kernel $\pi(\cdot|\cdot)$.

Since $G_\psi^*(x):=\int_{\Theta}\psi(\theta)\,\pi^*(\ud\theta|x)$ is Lipschitz on $\X$
and $\tilde G_\psi$ is $ACL$, the two functions coincide pointwise everywhere on almost every segment in a given direction $\nu\in\mathbb S^{m-1}$. Taking a countable dense subset $\mathcal D$ of $\psi\in C^1_c(\overline\Theta)$ (in the $C^1(\overline\Theta)$ norm) shows that $\pi^*(\cdot|\cdot)$ coincides with $\tilde \pi(\cdot|\cdot)$ on almost every line segment in the same direction $\nu$.
Therefore, given $\nu\in\mathbb S^{m-1}$, $\pi^*(\cdot|\cdot)$ itself satisfies $\mathcal W_2\left(\pi^*(\cdot|x_1),\pi^*(\cdot|x_2)\right)\le K |x_1-x_2|$ for $\mathcal L^{m-1}$-a.e. $\xi\in\mathbb X_\nu$ and any couple of distinct points $x_1,x_2\in I_{\xi,\nu}$. The result follows by the same argument at the end of the proof of Theorem \ref{th:0}-{\bf (iii)}.
\end{proofadconcrete}

In most situations a solution  to \eqref{neumanngintro} is not at disposal. Therefore, with some stronger assumptions we try to give an estimate of the norm of the solution in its dual formulation as seen in Theorem \ref{abs}. This is done in Theorem \ref{maintrace}.
We recall that the definition of the Fisher functionals $\mathcal J_1$ and $\mathcal J_2$ is given in \eqref{fisherfunctional}.

\begin{proofad2} 
Throughout this proof, as in the previous one, we simplify the notation by writing  $\mathbf s(t)$ in place of $\mathbf s_{x_1,x_2}(t)$, since no ambiguity arises. 
We apply  Reynolds transport formula. Let us consider an ACL representative of $\tilde g$, still denoted by $\tilde g$. Let $\psi\in C^1_c(\overline\Theta)$.
Given any $\nu\in\mathbb S^{m-1}$, for $\mathcal L^{m-1}$-a.e. $\xi\in\X_\nu$ we take any $x_1,x_2\in I_{\xi,\nu}$ and we obtain the absolute continuity of the map 
$[0,1]\ni t\mapsto \int_{\Theta_{\mathbf s(t)}}\psi(\theta)\tilde g(\mathbf s(t),\theta)\,\ud\theta$, along with
\begin{equation}\label{uffa}
\begin{aligned}
&\frac{\ud}{\ud r}\int_{\Theta_{\mathbf s(t)}}\psi(\theta)\,{\tilde g(\mathbf s(r),\theta)}\,\ud\theta{\bigg|}_{r=t}
=\frac{\ud}{\ud r}\int_\Theta\overline\psi_{\mathbf s(t)}(\theta)\,{\tilde g(\mathbf s(r),\theta)}\,\ud\theta{\bigg|}_{r=t}\\ 
&\; = |x_1-x_2| \int_{\Theta_{\mathbf s(t)}}\overline\psi_{\mathbf s(t)}(\theta)\,\partial_\nu{\tilde g(\mathbf s(t),\theta)}\,\ud\theta +|x_1-x_2| \int_{\Theta_{\mathbf s(t)}}\mathrm{div}\left(\overline\psi_{\mathbf s(t)}(\theta)\tilde g(\mathbf s(t),\theta)\mathbf V_{\mathbf s(t)}(\theta)\right)\,\ud\theta
\end{aligned}
\end{equation}
for a.e. $t\in(0,1)$. Here, we have used Lemma \ref{reynolds} and \eqref{onechain}, and we have introduced the function
$$
\overline\psi_{\mathbf s(t)}(\cdot):=\psi(\cdot)-\int_{\Theta_{\mathbf s(t)}} \psi(\theta)\tilde g(\mathbf s(t),\theta)\,\ud\theta, \qquad t\in[0,1].
$$

Let us proceed by estimating the two terms in the right hand side of \eqref{uffa}. 
The first term in the right hand side of \eqref{uffa} can be treated as in the proof of Theorem \ref{th:0}-{\bf (iii)}, so that by Cauchy-Schwarz inequality
\begin{equation*}
\int_{\Theta_{\mathbf s(t)}}\overline\psi_{\mathbf s(t)}(\theta)\,\partial_\nu \tilde g(\mathbf s(t),\theta)\,\ud\theta\le \left(\int_\Theta |\overline\psi_{\mathbf s(t)}(\theta)|^2\,g(\mathbf s(t),\theta)\,\ud\theta\right)^{\frac12}  
\mathcal{J}_1(\tilde g(\mathbf s(t),\cdot))
\end{equation*}
and then the Poincar\'e inequality \eqref{wirtinger} implies 
\begin{equation}\label{termine1}
\int_{\Theta_{\mathbf s(t)}}\overline\psi_{\mathbf s(t)}(\theta)\,\partial_\nu \tilde g(\mathbf s(t),\theta)\,\ud\theta\le \mathcal C[g(\mathbf s(t),\cdot)] \left(\int_\Theta |\nabla\psi(\theta)|^2\,g(\mathbf s(t),\theta)\,\ud\theta\right)^{\frac12}  \mathcal{J}_1[\tilde g(\mathbf s(t),\cdot)].
\end{equation}
Note that $\tilde g(\mathbf s(t),\cdot)\mathds 1_{\Theta_{\mathbf s(t)}}(\cdot)$ and $g(\mathbf s(t),\cdot)$ coincide for a.e. $t\in(0,1)$ as $L^1(\Theta)$ functions. 
The divergence term in \eqref{uffa} can be estimated by Cauchy-Schwarz and Poincar\'e's inequalities: indeed, since $\int_\Theta g(x,\theta)\,\ud\theta=1$ and $|\mathbf V_z^\nu|\le |\mathbf V_z|$, making use of the shorthand $\mathcal A_t:=\|\mathbf V_{\mathbf s(t)}\|_{W^{1,\infty}(\Theta_{\mathbf s(t)})}$, there holds
\begin{equation*}\label{termine2}\begin{aligned}
&\int_{\Theta_{\mathbf s(t)}}\mathrm{div}\left(\overline\psi_{\mathbf s(t)}(\theta)\tilde g(\mathbf s(t),\theta)\mathbf V_{\mathbf s(t)}(\theta)\right)\,\ud\theta \\&\qquad\le \mathcal A_t\left(  \int_{\Theta_{\mathbf s(t)}}|\nabla (\overline\psi_{\mathbf s(t)}(\theta)\,\tilde g(\mathbf s(t),\theta))|\,\ud\theta+\int_{\Theta_{\mathbf s(t)}}|\overline\psi_{\mathbf s(t)}(\theta)|\,\tilde g(\mathbf s(t),\theta)\,\ud\theta\right)\\
&\qquad\le  \mathcal A_t
\left[\left(\int_{\Theta_{\mathbf s(t)}}|\nabla\overline\psi_{\mathbf s(t)}|^2\tilde g(\mathbf s(t),\theta)\,\ud\theta\right)^{\frac12}+\int_{\Theta_{\mathbf s(t)}}|\overline\psi_{\mathbf s(t)}|\,(\tilde g(\mathbf s(t),\theta)+|\nabla \tilde g(\mathbf s(t),\theta)|)\,\ud\theta\right]\\
&\qquad\le \mathcal A_t \left(\int_{\Theta}|\nabla \psi(\theta)|^2\,g(\mathbf s(t),\theta)\,\ud\theta\right)^{\frac12}
\Big{(}1+\mathcal C[g(\mathbf s(t),\cdot)]\,+ \mathcal C[g(\mathbf s(t),\cdot)]\mathcal J_2(\tilde g(\mathbf s(t),\cdot))\Big{)}.
\end{aligned}\end{equation*}

By plugging \eqref{termine1} and the latter estimate into \eqref{uffa}, we get the following: given any $\nu\in\mathbb S^{m-1}$, for $\mathcal L^{m-1}$-a.e. $\xi\in\X_\nu$ and any couple $x_1,x_2\in I_{\xi,\nu}$, there holds
\begin{equation}\label{overfull}\begin{aligned}
&\frac{\ud}{\ud r}\int_{\Theta_{\mathbf s(t)}}\psi(\theta)\,{\tilde g(\mathbf s(r),\theta)}\,\ud\theta{\bigg|}_{r=t}
\\&\quad\le \mathcal C[g(\mathbf s(t),\cdot)] \left(\int_\Theta |\nabla\psi(\theta)|^2\,g(\mathbf s(t),\theta)\,\ud\theta\right)^{\frac12} \, \mathcal{J}_1[\tilde g(\mathbf s(t),\cdot)]\,|x_1-x_2|\\
&\quad\;\;\;\;+ \mathcal A_t \left(\int_{\Theta_{\mathbf s(t)}}|\nabla \psi(\theta)|^2\,g(\mathbf s(t),\theta)\,\ud\theta\right)^{\frac12}
\Big{(}1+\mathcal C[g(\mathbf s(t),\cdot)](1+\mathcal J_2(\tilde g(\mathbf s(t),\cdot)))\Big{)}\,|x_1-x_2|
\end{aligned}\end{equation}
for a.e. $t\in(0,1)$.
Now, let $G_\psi(\cdot):=\int_\Theta g(\cdot,\theta)\,\ud\theta$, so that $|G_\psi(x)|\le \sup_{\overline\Theta}|\psi|$ for $\mathcal L^m$-a.e. \!$x\in\X$. 
However, by \eqref{overfull}, by assumption and by the same argument as in the proof of Theorem \ref{concretemixed}, we get $|\nabla G_\psi(x)|\le K\mathrm{Lip}(\psi)$ for $\mathcal L^m$-a.e. \!$x\in\X$. Again this shows that there exists a $\mathcal W_1$-Lipscthiz representative $\pi^*(\cdot|\cdot)$ of the kernel $\pi(\cdot|\cdot)$.

We now let $\tilde \pi(\cdot|x):=\tilde g(x,\cdot)\,\mathcal L^d\mr\Theta_x$, which gives a representative of the kernel $\pi(\cdot|\cdot)$.
Let $\mathcal D$ be a countable dense subset of $C^1_c(\overline\Theta)$ (in the $C^1(\overline\Theta)$ norm) and let $\nu \in\mathbb S^{m-1}$. For  $\mathcal{L}^{m-1}$-a.e. \!$\xi\in\X_\nu$ and any 
$x_1,x_2\in I_{\xi,\nu}$, from \eqref{overfull} we get 
\[
\begin{aligned}
\Psi_{x_1,x_2}(t):&=\sup\left\{\frac{\ud}{\ud r}\int_\Theta\psi(\theta)\,\tilde \pi(\ud\theta|\mathbf s(t)){\bigg|}_{r=t}: \psi\in \mathrm{span}\mathcal D,\;\int_\Theta|\nabla\psi(\theta)|^2\,
\tilde\pi(\ud\theta|\mathbf s(t))\le1\right\}\\ 
&  \le |x_1-x_2|\, \|\mathbf V_{\mathbf{s}(t)}\|_{W^{1,\infty}(\Theta_{\mathbf s(t)})} \Big{(}1+\mathcal C[g(\mathbf s(t),\cdot)] (1+\mathcal J_2[g(\mathbf s(t),\cdot)])\Big{)} \\&\qquad+|x_1-x_2|\,\mathcal C[g(\mathbf s(t),\cdot)] \mathcal J_1[g(\mathbf s(t),\cdot)]
\end{aligned}
\]
for a.e. $t\in(0,1)$. Therefore, the estimate
\[
\begin{aligned}
&\int_0^1\Psi_{x_1,x_2}(t)\,\ud t \le|x_1-x_2|\int_0^1  \|\mathbf V_{\mathbf{s}(t)}\|_{W^{1,\infty}(\Theta_{\mathbf s(t)})} 
\Big{(}1+\mathcal C[g(\mathbf s(t),\cdot)](1+\mathcal J_2[g(\mathbf s(t),\cdot)]) \Big{)}\,\ud t\\
&\qquad +|x_1-x_2| \int_0^1\mathcal C[g(\mathbf s(t),\cdot)] \mathcal J_1[g(\mathbf s(t),\cdot)]\,\ud t\le K\,|x_1-x_2|
\end{aligned}
\]
holds for $\mathcal{L}^{m-1}$-a.e. \!$\xi\in\X_\nu$ and any $x_1,x_2\in I_{\xi,\nu}$. By invoking Theorem \ref{abs},
we deduce that $\tilde\pi(\cdot|x_1)$ and $\tilde\pi(\cdot|x_2)$ are in $\mathcal P_2(\overline\Theta)$ and that $\mathcal W_2(\tilde\pi(\cdot|x_1),\tilde\pi(\cdot|x_2))\le K|x_1-x_2|$, for $\mathcal{L}^{m-1}$-a.e. 
\!$\xi\in\X_\nu$ and any $x_1,x_2\in I_{\xi,\nu}$. By the same argument as in the proof of Theorem \ref{concretemixed}, the same conclusion holds for $\pi^*(\cdot|\cdot)$, which identifies with $\tilde\pi(\cdot|\cdot)$ on almost every line in any given direction. But $x\mapsto \pi^*(\cdot|x)$ is $\mathcal W_1$-Lipschitz on the whole of $\X$, so that we conclude by the argument already explained at the end of the proof of Theorem \ref{th:0}-{\bf(iii)}.
\end{proofad2}

\subsection{Infinite-dimensional sample space: proof of Theorem \ref{infinity}}
We next provide the proof of the results that deal with infinite-dimensional sample space from Subsection \ref{infinitesection}. In this case we shall prove that a Lipschitz estimate holds for  `good couples' 
$(x_1,x_2)\in\mathbb B(\Z)$ and then we invoke the Lipschitz extension result from Lemma \ref{lm:lipschitzextension} in the Appendix. 

\begin{proofadinfinity} We start by proving point {\bf (i)}.
Let $\psi\in\ C^1_c(\overline\Theta)$ and $G_\psi(x):=\int_\Theta g(x,\theta)\,\pi(\ud\theta)$. Given a couple $(x_1,x_2)\in \mathbb B(\Z)$,
for $\pi$-a.e. $\theta\in\Theta$, the map $[0,1]\ni t\mapsto g(\mathbf s_{x_1,x_2}(t),\theta)$ has, by assumption, the following properties: it is absolutely continuous, and for a.e. $t\in(0,1)$ the point 
$(\mathbf s_{x_1,x_2}(t),\theta)$ is a Gateaux-differentiability point of $g$ with respect to the $x$-variable. Therefore,  by an application of 
Fubini's theorem, we get
\[\begin{aligned}
&G_\psi(x_2)-G_\psi(x_1)= \int_\Theta \psi(\theta)(g(\mathbf s_{x_1,x_2}(1),\theta)-g(\mathbf s_{x_1,x_2}(0),\theta))\,\pi(\ud\theta)\\&\quad=\int_\Theta\psi(\theta)\int_0^1\frac{\ud}{\ud t}(g(\mathbf s_{x_1,x_2}(t),\theta))\,\ud t\,\pi(\ud\theta)\\&\quad=\int_0^1\int_\Theta\psi(\theta)\langle D_x g(\mathbf s_{x_1,x_2}(t),\theta),x_2-x_1\rangle\,\pi(\ud\theta)\,\ud t \\&\quad\le \|x_2-x_1\|_{\V} \sup_\Theta |\psi|\,\esssup_{t\in(0,1)}\int_\Theta \|D_x g(\mathbf s_{x_2,x_1}(t),\theta)\|_{\V'}\,\pi(\ud\theta)\\
&\quad\le \|x_2-x_1\|_{\V} \sup_\Theta |\psi|\,\lesssup_{x\in\X} \int_\Theta \|D_x g(x,\theta)\|_{\V'}\,\pi(\ud\theta).
\end{aligned}
\]
In the last inequality we have used the fact that any function $F:\X\to\R$ satisfies 
\begin{equation}\label{esss}
\esssup_{t\in (0,1)} F(\mathbf s_{x_1,x_2}(t))\le \lesssup_{x\in\X} F(x),
\end{equation} 
since $\mathcal H^1([x_1,x_2]\cap\Z)=0$ as $(x_1,x_2)\in\mathbb B(\Z)$. Therefore, we get
 \[\begin{aligned}
2\ud_{TV}(\pi(\cdot|x_2),\pi(\cdot|x_1)) &= \sup_{{\psi\in C_c(\Theta)}\atop{|\psi|\le 1}}(G_\psi(x_2)-G_\psi(x_1))\\& \le  \sup_{{\psi\in C_c(\Theta)}\atop{|\psi|\le 1}}  \|\nabla G_\psi\|_{L^\infty(\X)}\,\|x_1-x_2\|_{\V} \le K \|x_1-x_2\|_{\V}
\end{aligned}\]
for any $(x_1,x_2)\in\mathbb B(\Z)$. By invoking the extension result from Lemma \ref{lm:lipschitzextension}, the mapping $x\mapsto\pi(\cdot|x)\in\mathcal{P}(\overline\Theta)$ admits a Lipschitz continuous extension 
 $\X\ni x\mapsto\pi^*(\cdot|x)\in\mathcal{P}(\overline\Theta)$ with respect to the total variation distance. Since $\pi(\Theta)=1$, the continuity in total variation also shows that $\pi^*(\Theta|x)=1$ for any $x\in\X$. This ends the proof of {\bf (i)}.

The result in point {\bf (ii)} is obtained by introducing the Poincar\'e inequality \eqref{wirtinger} in the computations of the proof of point {\bf (i)}, as done in the proof of Theorem \ref{th:0}-{\bf (ii)}.

Let us conclude by proving {\bf (iii)}. The proof is similar to the one of Theorem \ref{th:0}-{\bf (iii)}.
Let $\mathcal D$ denote a countable dense subset of $C^1_c(\overline\Theta)$ (in the $C^1(\overline\Theta)$ norm), let $G_{\psi}(x):=\int_\Theta\psi(\theta)g(x,\theta)\,\pi(\ud\theta)$.
Let us consider a couple $(x_1,x_2)\in\mathbb B(\Z)$. Thanks to the assumptions, we have the absolute continuity of the map $t\mapsto g(\mathbf s_{x_1,x_2}(t),\theta)$ for $\pi$-a.e. $\theta\in\Theta$ and 
we may apply Lemma \ref{pilemma} from the Appendix, so that the map 
$t\mapsto \int_\Theta \psi(\theta) g(\mathbf s_{x_1,x_2}(t),\theta)\,\pi(\ud\theta)$
 is absolutely continuous for any $\psi\in \mathcal D$ and we may differentiate under integral sign to get for a.e. $t\in(0,1)$
\[
\frac{\ud}{\ud r}\int_\Theta\psi(\theta)\,g(\mathbf s_{x_1,x_2}(r),\theta)\,\pi(\ud\theta){\bigg|}_{r=t} = \int_\Theta\psi(\theta)\left(\frac{\ud}{\ud t}\,g(\mathbf s_{x_1,x_2}(t),\theta)\right)\,\pi(\ud\theta).
\]
As usual, the $\mathcal L^1$-null set of non-differentiability points of the map $t\mapsto \int_\Theta \psi(\theta) g(\mathbf s_{x_1,x_2}(t),\theta)\,\ud\theta$ is independent of $\psi\in\mathcal D$, since $\mathcal D$ is countable. By considering the Gateaux-differentiability property of $g$, we get, for a.e. $t\in(0,1)$,
\[
\frac{\ud}{\ud r}\int_\Theta\psi(\theta)\,{g(\mathbf s_{x_1,x_2}(r),\theta)}\,\pi(\ud\theta){\bigg|}_{r=t}\le \|x_1-x_2\|_{\V} \int_\Theta|\psi(\theta)| \|D_x g(\mathbf s_{x_1,x_2}(t),\theta)\|_{\V'}\,\pi(\ud\theta)\ .
\]
Hence, combining the Cauchy-Schwarz and the Poincar\'e inequality \eqref{wirtinger}, we obtain
\[
\begin{aligned}
&\frac{\ud}{\ud r}\int_\Theta\psi(\theta)\,g(\mathbf s_{x_1,x_2}(r),\theta)\,\pi(\ud\theta){\bigg|}_{r=t}
=\inf_{a\in\R} \frac{\ud}{\ud r}\int_\Theta(\psi(\theta)-a)\,g(\mathbf s_{x_1,x_2}(r),\theta)\,\pi(\ud\theta){\bigg|}_{r=t}\\
&\;\le \|x_1-x_2\|_{\V} \left(\int_\Theta\frac{\|D_x g(\mathbf s_{x_1,x_2}(t),\theta)\|^2_{\V'}}{g(\mathbf s_{x_1,x_2}(t),\theta)}\,\pi(\ud\theta)\right)^{\frac12}\inf_{a\in\R}\left(\int_\Theta |\psi(\theta)-a|^2 
g(\mathbf s_{x_1,x_2}(t),\theta)\,\pi(\ud\theta)\right)^\frac12\\
&\;\le \|x_1-x_2\|_{\V} \,\mathcal C[g(\mathbf s_{x_1,x_2}(t),\cdot)\,\pi] \left(\int_\Theta |\nabla\psi(\theta)|^2 g(\mathbf s_{x_1,x_2}(t),\theta)\,\pi(\ud\theta)\right)^\frac12
\mathcal J_{\pi}[g(\mathbf s_{x_1,x_2}(t),\cdot)]\ .
\end{aligned}
\]
Whence,
\[\begin{aligned}
\Psi_{x_1,x_2}(t) &:=\sup_{\psi\in\mathrm{span}\mathcal D} \left\{\frac{\ud}{\ud r}\int_\Theta\psi(\theta)\,g(\mathbf s_{x_1,x_2}(r),\theta)\,\pi(\ud\theta){\bigg|}_{r=t}: \int_\Theta |\psi(\theta)|^2 
g(\mathbf s_{x_1,x_2}(t),\theta)\,\ud\theta\le 1\right\}\\
&\le \mathcal C[g(\mathbf s_{x_1,x_2}(t),\cdot)\,\pi]\,\mathcal J_{\pi}[g(\mathbf s_{x_1,x_2}(t),\cdot)]\, \|x_1-x_2\|_{\V}.
\end{aligned}\]
Combining the latter estimate with \eqref{esss}, we conclude that for any $(x_1,x_2)\in\mathbb B(\Z)$ there holds
\[
\int_0^1\Psi_{x_1,x_2}(t)\,\ud t\le \|x_1-x_2\|_{\V}\int_0^1 \mathcal C[g(\mathbf s_{x_1,x_2}(t),\cdot)\,\pi] \mathcal J_{\pi}[g(\mathbf s_{x_1,x_2}(t),\cdot) \pi]\,\ud t\le K\|x_1-x_2\|_{\V}.
\]
Hence, an application of Theorem \ref{abs} shows that  $x\mapsto\pi(\cdot|x)$ satisfies the desired estimate for any $(x_1,x_2)\in\mathbb B(\Z)$. The Lipschitz extension property from Lemma 
\ref{lm:lipschitzextension}, applied to the complete metric space $(\mathcal P_2(\overline\Theta),\Wdue),$ yields the result.
\end{proofadinfinity}

\KKK

\section{Appendix} 

\subsection{A proof of Reynolds transport formula}

We give here a proof of some useful calculus formulae that are often needed through the paper.
 The following is a proof of Reynolds transport theorem, see also, for instance,  \cite[Th\'eor\`em 5.2.2]{hp} or \cite[Section 10]{Gu}. The proof is given for domains that vary according to a regular motion as defined in section \ref{theorysection}.   In the following lemma, we make use of the notation  $C^1_a(\overline\Theta):=\{\psi+a: a\in\R,\, \psi\in C^1_c(\overline\Theta)\}$ where, as usual, $\Theta$ is an open connected subset of $\rd$.
\begin{lem}\label{reynolds}
Let $\tilde g\in W^{1,1}((a,b)\times\Theta)$. Let $[a,b]\times\Theta_*\ni (t,\theta)\mapsto \Phi_t(\theta)\in\rd$ be a regular motion in $\Theta$ according to {\rm Definition \ref{motion}},  with $\Theta_t:=\Phi_t(\Theta_*)$. 
Then (any ACL representative of) $\tilde g$ is such that $[a,b]\ni t\mapsto\int_{\Theta_t}\psi(\theta)\tilde g(t,\theta)\,\ud\theta$ is absolutely continuous for any  $\psi\in C^1_a(\overline\Theta)$. Moreover, given $\psi\in C^1_a(\overline\Theta)$ there holds for a.e. $t\in(a,b)$ 
\[
\begin{aligned}
\frac{\ud}{\ud t}\int_{\Theta_t}\psi(\theta)\tilde g(t,\theta)\,\ud\theta
&=\int_{\Theta_t}\psi(\theta)\partial_t \tilde g(t,\theta)\,\ud\theta+\int_{\Theta_t} \nabla\cdot \left(\psi(\theta)\,\tilde g(t,\theta)(\partial_t\Phi_t\circ\Psi_t)(\theta)\right)\,\ud\theta,
\end{aligned}
\]
where $\Psi_t:\Theta_t\to\Theta_*$ is the inverse of $\Phi_t$. If $\psi\in C^1_c(\overline\Theta)$ we also have for a.e. $t\in(a,b)$
\[\begin{aligned}
\frac{\ud}{\ud t}\int_{\Theta_t}\psi(\theta)\tilde g(t,\theta)\,\ud\theta
&=\int_{\Theta_t}\psi(\theta)\partial_t \tilde g(t,\theta)\,\ud\theta+\int_{\partial\Theta_t} \psi(\sigma)\,\tilde g(t,\sigma)\mathbf{n}_t(\sigma)\cdot(\partial_t\Phi_t\circ\Psi_t)(\sigma)\,\mathcal{H}^{d-1}(\ud\sigma),
\end{aligned}\]
where $\mathbf{n}_t$  denotes the exterior normal to $\Theta_t$, and  the $L^1_{loc}(\partial\Theta_t)$ boundary trace of $\tilde g(t,\cdot)$ on $\partial\Theta_t$ appears in the last term.
\end{lem}
\begin{proof}
As $\Phi_t$ is a global diffeomorphism of $\Theta_*$ onto $\Theta_t$ for any $t\in[a,b]$, then  $(\mathbf{i},\Phi_t)$ is a global diffeomorphism of $(a,b)\times\Theta_*$ onto $\{(t,\theta)\in(a,b)\times\mathring\Theta:\theta\in\Theta_t\}$, whose Jacobian determinant is bounded away from $0$ and $+\infty$ (tanks to the assumptions in Definition \ref{motion}). Thus, $\tilde g\circ(\mathbf{i},\Phi_t)\in W^{1,1}((a,b)\times\Theta_*)$.

Let $\psi\in C^1_a(\overline\Theta)$. By change of variables we have $\mbox{for a.e. $t\in(a,b)$}$
\begin{equation}\label{basicchange*}
\int_{\Theta_t}\psi(\theta)\tilde g(t,\theta)\,\ud\theta=\int_{\Theta_*}\tilde g(t,\Phi_t(\theta))\psi(\Phi_t(\theta))\det\nabla\Phi_t(\theta)\,\ud\theta.
\end{equation}
By distributional chain rule we have 
$
\partial_t\tilde g(t,\Phi_t(\theta))+\nabla\tilde g(t,\Phi_t(\theta))\cdot\partial_t\Phi_t(\theta)=\frac{\ud}{\ud t}\tilde g(t,\Phi_t(\theta))
$ 
and similarly  
\[
\begin{aligned}
&\frac{\ud}{\ud t}\left[\tilde g(t,\Phi_t(\theta))\psi(\Phi_t(\theta))\det\nabla\Phi_t(\theta)\right]=\psi(\Phi_t(\theta))\det\nabla\Phi_t(\theta)\,\frac{\ud}{\ud t}\tilde g(t,\Phi_t(\theta))\\&\quad+\tilde g(t,\Phi_t(\theta))\det\nabla\Phi_t(\theta)\,\nabla\psi(\Phi_t(\theta))\cdot\partial_t\Phi_t(\theta) \\&\quad+\tilde g(t,\Phi_t(\theta))\psi(\Phi_t(\theta))\,\mathrm{Tr}((\partial_t\nabla\Phi_t)(\nabla\Phi_t)^{-1}) \,\det\nabla\Phi_t,
\end{aligned}
\]
where we used the identity $\tfrac{\ud}{\ud t}\det\nabla\Phi_t= \mathrm{Tr}((\partial_t\nabla\Phi_t)(\nabla\Phi_t)^{-1}) \,\det\nabla\Phi_t$ (with Tr denoting matrix trace).
Therefore, the assumptions in Definition \ref{motion} show that the map
$(t,\theta)\mapsto h_\psi(t,\theta):= \tilde g(t,\Phi_t(\theta))\psi(\Phi_t(\theta))\det\nabla\Phi_t(\theta)
$
belongs to $L^1((a,b)\times \Theta_*)$ together with $\partial_t h_\psi$. 
We conclude that, for a.e. $\theta\in\Theta_*$, the map  $t\mapsto h_\psi(t,\theta)$  is in $W^{1,1}(a,b)$,  hence  absolutely continuous  up to defining it  through  a representative of $\tilde g\circ(\mathbf{i},\Phi_t)\in W^{1,1}((a,b)\times\Theta_*)$ with the same property:  then, by the fundamental theorem of calculus and by Fubini theorem we have for $a\le s<t\le b$
\begin{equation}\label{fundamental}
\int_{\Theta_*} (h_\psi(t,\theta)-h_\psi(s,\theta))\,\ud\theta=\int_{\Theta_*}\int_{s}^{t}\partial_th_\psi(r,\theta)\,\ud r\,\ud\theta=\int_s^t \int_{\Theta_*}\partial_th_\psi(r,\theta)\,\ud\theta\,\ud r
\end{equation}
and this shows that $h_\psi(t,\cdot)\in L^1(\Theta_*)$ for any $t\in[a,b]$ and  that $t\mapsto\int_{\Theta_*}h_\psi(t,\theta)\,\ud\theta$ is   absolutely continuous  on $[a,b]$. 

We claim that if $\tilde g$ is a (not relabeled) ACL representative of $\tilde g$, then $t\mapsto \int_{\Theta_t}\psi(\theta)\,\tilde g(t,\theta)\,d\theta$ is indeed absolutely continuous on $[a,b]$. It is enough to check that it is continuous, since we have just shown that the right hand side of \eqref{basicchange*} has an absolutely continuous representative on $[a,b]$.  Assuming wlog that $a\le s<t\le b$, we have $\tilde g(t,\cdot)\in L^1(\Theta)$ for any $t\in[a,b]$ as well as the absolute continuity of $t\mapsto\int_\Theta\tilde g(t,\theta)\,d\theta$, since and $\partial_t\tilde g\in L^1((a,b)\times\Theta)$ and Fubini theorem implies as above
\[
\int_\Theta\tilde g(t,\theta)\,\ud\theta-\int_\Theta\tilde g(s,\theta)\,\ud\theta=\int_s^t \int_\Theta\partial_t\tilde g(t,\theta)\,\ud\theta\,\ud t\ . 
\] 
Then, again by Fubini theorem we get
\[
\int_{\Theta_t}\psi(\theta)\tilde g(t,\theta)\,\ud\theta-\int_{\Theta_s}\psi(\theta)\tilde g(s,\theta)\,\ud\theta
=\int_\Theta\psi(\theta)(\mathds 1_{\Theta_t}-\mathds 1_{\Theta_s})\tilde g(t,\theta)\,\ud\theta+\int_s^t\int_{\Theta_s}\psi(\theta)\partial\tilde g(t,\theta)\,\ud t
\] 
where the last term vanishes as $s\to t$ since $\psi$ is bounded and $\partial_t\tilde g\in L^1((a,b)\times\Theta)$. The first term in the right hand side vanishes as well as $s\to t$ thanks to dominated convergence, since 
$\tilde g(t,\cdot)\in L^1(\Theta)$ and since the pointwise converges of $\mathds 1_{\Theta_s}$ to $\mathds 1_{\Theta_t}$ easily follows from the assumptions in Definition \ref{motion}. The claim is proved.

Dividing \eqref{fundamental} by $s$ and using the Lebesgue points theorem, we see that for $a.e.$ $t\in (a,b)$ there holds
\[
\frac{\ud}{\ud t}\int_{\Theta_*}h_\psi(t,\theta)\,\ud\theta= \int_{\Theta_*}\partial_t h_\psi(t,\theta)\,\ud\theta.
\]
Since we can take the time derivative inside the integral sign, we have by change of variables and by the identity $(\nabla \Phi_t)^{-1}\circ \Psi_t=\nabla\Psi_t$, and with the notation $J_t=\det\nabla\Phi_t$ (so that $\tfrac{\ud}{\ud t}J_t= \mathrm{Tr}((\partial_t\nabla\Phi_t)(\nabla\Phi_t)^{-1}) \,\det\nabla\Phi_t$),
\begin{equation*}\label{gh}
\begin{aligned}
\frac{\ud}{\ud t}\int_{\Theta_t}\psi(\theta)\tilde g(t,\theta)\,\ud\theta&=\frac{\ud}{\ud t}\int_{\Theta_t}h_\psi(t,\theta)\,\ud\theta
=\frac{\ud}{\ud t}\int_{\Theta_*}\tilde g(t,\Phi_t(\theta))\,\psi(\Phi_t(\theta))\,J_t(\theta)\,\ud\theta\\
&=\int_{\Theta_*}\left(J_t(\theta)\,\frac{\ud}{\ud t}(\tilde g(t,\Phi_t(\theta))\psi(\Phi_t(\theta)))+\tilde g(t,\Phi_t(\theta))\psi(\Phi_t(\theta))\,\frac{\ud}{\ud t}J_t(\theta)\right)\,\ud\theta\\
&= \int_{\Theta_t}\partial_t\tilde g(t,\theta)\psi(\theta)\,\ud\theta+\int_{\Theta_t}\nabla\cdot{\Big{(}}\tilde g(t,\theta)\psi(\theta)\partial_t\Phi_t(\Psi_t(\theta)){\Big{)}}\,\ud\theta,
\end{aligned}
\end{equation*}
for a.e. $t\in(a,b)$. By the divergence theorem, the proof is concluded.
\end{proof}

Of course, if $\Theta_t\equiv\Theta$ for all $t$, we have that $\Phi_t$ is the identity map for any $t$.  Lemma \ref{reynolds} holds and Reynolds transport formula reduces to differentiation under integral sign. However, in such case we may extend the result to general probability measures on $\Theta$, without requiring a density. We have the following standard result.

\begin{lem}\label{pilemma} Let $(\Theta, \mathscr{T}, \pi)$ be a measure space, with $\pi$ a $\sigma$-finite measure.
Let $g:[a,b]\times\Theta\to\R$. Suppose that
\begin{itemize}
\item[i)] $g(\cdot,\theta)\in AC([a,b])$ for $\pi$-a.e. $\theta\in \Theta$
 and $g(t,\cdot)\in L^1_\pi(\Theta)$ for all $t\in [a,b]$; 
 
 \medskip
 
\item[ii)] 
$\displaystyle\int_a^b\int_\Theta|\partial_t g(t,\theta)|\,\pi(\ud\theta)\,\ud t<+\infty$.
\end{itemize}
Then the map $t\mapsto \int_\Theta \psi(\theta)g(t,\theta)\,\pi(\ud\theta)$ is absolutely continuous on $[a,b]$ for any bounded continuous function $\psi$ on $\Theta$. In particular, if $\psi$ is a bounded continuous function on $\Theta$ there holds a.e. on $(a,b)$
\[
\frac{\ud}{\ud t}\int_\Theta\psi(\theta)g(t,\theta)\,\pi(\ud\theta)=\int_\Theta\psi(\theta)\partial_t g(t,\theta)\,\pi(\ud\theta).
\]
\end{lem}
\begin{proof} 
By assumption i), $g$ is a Carath\'eodory function, hence  $(\ud t\otimes\pi)$-measurable, moreover $g$ is (classically) partially differentiable with respect to $t$ at $(\ud t\otimes\pi)$-a.e. $(t,\theta)\in (a,b)\times\Theta$, and then, by assumption ii), we have  $\partial_tg\in L^1_{\ud t\otimes\pi}((a,b)\times\Theta)$.
For any $a\le t_1<t_2\le b$, assumption i) entails that 
$g(t_2,\theta)-g(t_1,\theta)=\int_{t_1}^{t_2}\partial_tg(t,\theta)\,\ud t$ holds for $\pi$-a.e. $\theta\in\Theta$,
the mapping $t\mapsto \partial_t g(t,\theta)$ being in $L^1(a,b)$ for $\pi$-a.e $\theta\in \Theta$.
Thanks to ii) and the boundedness of $\psi$, the map $t\mapsto \int_\Theta\psi(\theta)\partial_t g(t,\theta)\,\pi(\ud\theta)$ belongs to $L^1(a,b)$, and we may apply Fubini theorem to obtain for 
$a\le t_1<t_2\le b$
\[\begin{aligned}
\int_\Theta\psi(\theta)(g(t_2,\theta)-g(t_1,\theta))\,\pi(\theta)&=\int_\Theta\psi(\theta)\left(\int_{t_1}^{t_2}\partial_t g(t,\theta)\right)\pi(\ud\theta)
\\&=\int_{t_1}^{t_2}\int_\Theta\psi(\theta)\partial_t g(t,\theta)\,\pi(\ud\theta)\,\ud t \ . 
\end{aligned}\]
This shows that the map $t\mapsto\int_\Theta\psi(\theta)g(t,\theta)\,\pi(\ud\theta)$ is  an $AC([a,b])$ map.
Finally, since the map $t\mapsto \int_\Theta\psi(\theta)\partial_t g(t,\theta)\,\pi(\ud\theta) $ is in $L^1(a,b)$ and $g(\cdot,\theta)\in AC([a,b])$ for $\pi$-a.e. $\theta\in\Theta$, we apply Fubini once more to 
get
\[
 \int_\Theta\psi(\theta)\partial_t g(t,\theta)\,\pi(\ud\theta)
=\frac{\ud}{\ud t}\int_\Theta \psi(\theta)(g(t,\theta)-g(t_0,\theta))\,\pi(\ud\theta)=\frac{\ud}{\ud t}\int_\Theta\psi(\theta)g(t,\theta)\,\pi(\ud\theta),
\]
for $t_0\in (a,b)$ and for $a.e.$ t in $(a,b)$, which completes the proof.
\end{proof}


\subsection{Absolutely continuous curves and the continuity equation}

The characterization of absolutely continuous curves in Wasserstein spaces is investigated in \cite[Chapter 8]{AGS}. Following the argument of \cite[Theorem 8.3.1]{AGS}, we get the next result.

\begin{lem}\label{technical}
In the same assumptions of {\rm Theorem \ref{abs}}, there exists a Borel map $[0,1]\times\overline\Theta\ni (t,\theta)\mapsto\mathbf{v}(t,\theta)\in\R^d$ such that \eqref{distributionalcontinuityequation} is satisfied by the couple $(\mathbf{v}(t,\cdot),\mu_t)$. Moreover, $\|\mathbf{v}(t,\cdot)\|_{L^2_{\mu_t}(\overline\Theta;\R^d)}=\Psi(t)$ holds for a.e. $t\in(0,1)$.
\end{lem}
\begin{proof}
Let $L:=\|\Psi\|_{L^1(0,1)}$ and let $g:[0,1]\to[0,1]$ be defined as $g(t)=\frac{1}{L+1}\int_0^t(\Psi(s)+1)\,ds$, thus $g$ is a strictly increasing absolutely continuous function such that $g'> 0$ for a.e. $t\in(0,1)$. Therefore $h:=g^{-1}$ is a also a strictly increasing absolutely continuous map and 
$
1/h'(t)={g'(g^{-1}(t))}
$
holds for a.e. $t\in(0,1)$. We define a new curve in $\mathcal{P}(\overline\Theta)$ by $[0,1]\ni t\mapsto\bar\mu_t:=\mu_{h(t)}$. It is clear that $\bar\mu_t$ itself satisfies assumption i) of Theorem \ref{abs}, as the mapping $t\mapsto\int_\Theta\psi(\theta)\,\mu_{h(t)}(d\theta)$ is the composition of an absolutely continuous function with an absolutely continuous strictly increasing function. Moreover, we have
\begin{align}\label{firstline}
\bar\Psi(t):&=\sup_{ \psi\in \mathrm{span}\mathcal D}\left\{\frac{\ud}{\ud s}\int_{\overline\Theta}\psi(\theta)\,\bar\mu_{s}(\ud\theta){\Bigg{|}_{s=t}}:
 \int_{\overline\Theta} |\nabla\psi(\theta)|^2\bar\mu_t(\ud\theta)\le 1\right\}\\
&=h'(t)\sup_{ \psi\in \mathrm{span}\mathcal D}\left\{\frac{\ud}{\ud s}\int_{\overline\Theta}\psi(\theta)\,\mu_s(\ud\theta){\Bigg{|}_{s=h(t)}}:
\int_{\overline\Theta} |\nabla\psi(\theta)|^2\mu_{h(t)}(\ud\theta)\le 1\right\}=h'(t)\Psi(h(t)),\nonumber
\end{align}
therefore $\bar \mu_t$ satisfies also assumption ii) of Theorem \ref{abs}. Notice that $\|\bar\Psi\|_{L^1(0,1)}=\|\Psi\|_{L^1(0,1)}$. 
Also, we have for a.e. $t\in(0,1)$
\begin{equation}\label{Linf}
\bar\Psi(t)=h'(t)\Psi(h(t))=\frac{\Psi(g^{-1}(t))}{g'(g^{-1}(t))}=\frac{\Psi(g^{-1}(t))}{\Psi(g^{-1}(t))+1}\,(L+1)\le L+1.
\end{equation}

Next, we consider the set
\[
{\mathscr Y}:=\left\{\varphi\in C^1([0,1]\times\overline\Theta): \varphi(t,\theta)=\eta(t)\psi(\theta),\; \eta\in C^\infty_c(0,1),\; \psi\in \mathcal D\right\}
\]
and we define on $\{\nabla\varphi:\varphi\in\mathrm{span}{\mathscr{Y}}\}$, where $\nabla$ is the gradient in $\theta$, the linear operator 
\[
\mathcal{L}[\nabla \varphi]:=-\int_0^1\int_{\overline\Theta}\partial_t\varphi(t,\theta)\bar\mu_t(\ud\theta)\,dt
\] 
Since $\bar\mu_t$ satisfies assumption i) of Theorem \ref{abs}, we may integrate by parts and get 
\begin{equation}
\label{LT}\begin{aligned}
\mathcal{L}[\nabla\varphi]&=-\int_0^1\eta'(t)\int_{\overline\Theta}\psi(\theta)\,\bar\mu_t(\ud\theta)\,dt=\int_0^1\eta(t)\frac{\ud}{\ud s}\int_{\overline\Theta}\psi(\theta)\,\bar\mu_s(\ud\theta){\Bigg{|}_{s=t}}\,\ud t
\\&=
\int_0^1\frac{\ud}{\ud s}\int_{\overline\Theta}\varphi(t,\theta)\,\bar\mu_s(\ud\theta){\Bigg{|}_{s=t}}\,\ud t.
\end{aligned} \end{equation}
for any $\varphi\in\mathscr{Y}$.
The same equality readily holds for any $\varphi\in\mathrm{span}\mathscr{Y}$. Notice that for any $\varphi\in \mathrm{span}\mathscr{Y}$ and any $t\in[0,1]$, the mapping $\theta\mapsto\varphi(t,\theta)$ belongs to $\mathrm{span} \mathcal D$. 
Therefore, for any $\varphi\in\mathrm{span}\mathscr{Y}$, from \eqref{LT} we may apply \eqref{firstline} along with Cauchy-Schwarz inequality and \eqref{Linf} and get
\[
\begin{aligned}
\left|\mathcal{L}[\nabla\varphi]\right|\le \int_0^1\bar\Psi(t)\left(\int_{\overline\Theta}|\nabla\varphi(t,\theta)|^2\,\bar\mu_t(\ud\theta)\right)^{\frac12}\,\ud t\le
(L+1)\left(\int_0^1\int_{\overline\Theta} |\nabla\varphi(t,\theta)|^2\,\bar\mu_t(\ud\theta)\,\ud t\right)^{\frac12}\!\!.
\end{aligned}
\]
This means that the linear operator $\mathcal{L}$ can be extended as a bounded linear operator on the $L^2_{\bar\mu}((0,1)\times\overline\Theta;\R^d)$ closure of 
$\{\nabla\varphi:\varphi\in\mathrm{span}\mathscr{Y}\}$, where $\bar\mu:=\bar\mu_t\otimes \ud t$ is the Borel probability measure on $(0,1)\times\overline\Theta$ whose disintegration with respect to $\ud t$ is $\bar\mu_t$. 
By Riesz representation theorem, there exists a unique vector field 
$$\bar{\mathbf{v}}\in \overline{\{\nabla\varphi:\varphi\in\mathrm{span} \mathscr{Y}\}}^{L^2_{\bar\mu}((0,1)\times\overline\Theta;\R^d)}$$
such that
\begin{equation}\label{2variables}
\mathcal{L}[\nabla\varphi]=\int_0^1\int_{\overline\Theta} \langle \bar{\mathbf{v}}(t,\theta),\nabla_\theta\varphi(t,\theta)\rangle\,\bar\mu(\ud t\,\ud\theta)\qquad\forall \varphi\in \mathrm{span} \mathscr{Y}.
\end{equation}
By the definition of $\mathcal{L}$, \eqref{2variables} shows that the equality in \eqref{distributionalcontinuityequation} is satisfied for any $\varphi\in\mathscr{Y}$.
A standard approximation argument shows that it is enough to check the continuity equation on factorized functions, and
since $\mathcal D$ is dense in $C^1_c(\overline\Theta)$ (in the $C^1(\overline\Theta)$ norm), we conclude that the couple $(\bar{\mathbf{v}}(t,\cdot),\bar\mu_t)$ satisfies \eqref{distributionalcontinuityequation}.
Moreover, from
 \eqref{LT}-\eqref{2variables}, by de DuBois-Reymond lemma we see that, for a.e. $t\in(0,1)$, $\bar{\mathbf{v}}(t,\cdot)$ satisfies
 \[
 \frac{\ud}{\ud s}\int_{\overline\Theta}\psi(\theta)\,\bar\mu_s(\ud\theta){\Bigg{|}_{s=t}}=\int_{\overline\Theta}\langle\bar{\mathbf{v}}(t,\theta),\nabla\psi(\theta)\rangle\,\bar\mu_t(\ud\theta)\qquad\forall{\psi\in \mathrm{span}\mathcal D}, 
 \]
therefore from \eqref{firstline}
\begin{equation}\label{id1}
\|\bar{\mathbf{v}}(t,\cdot)\|_{L^2_{\bar\mu_t(\overline\Theta;\R^d)}}
=\bar\Psi(t)=h'(t)\Psi(h(t))\quad\mbox{ for a.e. $t\in(0,1)$.}
\end{equation}

Let us define another Borel map $[0,1]\times\overline\Theta\ni(t,\theta)\mapsto \mathbf{v}(t,\theta)\in\R^d$ as $\mathbf{v}(t,\theta)=\bar{\mathbf{v}}(g(t),\theta)g'(t)$. The couple $(\mu_t,\mathbf{v}(t,\cdot))$ satisfies the continuity equation in the sense of \eqref{distributionalcontinuityequation}. This follows from a change of variable in time in \eqref{distributionalcontinuityequation}, see also \cite[Lemma 8.13]{AGS}. Eventually, from the definition of $\mathbf{v}$ and since $h$ is the inverse of $g$, for a.e. $t\in(0,1)$ there holds 
\begin{equation*}
\|\bar{\mathbf{v}}(t,\cdot)\|_{L^2_{\bar\mu_t(\overline\Theta;\R^d)}}=\left(\int_{\overline\Theta}\left|\frac{\mathbf{v}(h(t),\theta)}{g'(g^{-1}(t))}\right|^2\,\mu_{h(t)}(\ud\theta)\right)^{\frac12}=h'(t)\,\left(\int_{\overline\Theta} 
|\mathbf{v}(h(t),\theta)|^2\,\mu_{h(t)}(\ud\theta)\right)^{\frac12},
\end{equation*}
which entails, thanks to \eqref{id1},
\[
\left(\int_{\overline\Theta}|\mathbf{v}(h(t),\theta)|^2\,\mu_{h(t)}(\ud\theta)\right)^{\frac12}=\Psi(h(t))\quad\mbox{ for a.e. $t\in(0,1)$.}
\]
Since $h$ is a strictly increasing absolutely continuous bijection of $[0,1]$ onto itself, we obtain $\|\mathbf{v}(t,\cdot)\|_{L^2_{\mu_t}(\overline\Theta;\R^d)}=\Psi(t)$ for a.e. $t\in(0,1)$.
\end{proof}


\subsection{Lipschitz continuous extensions}
We provide the proof of a Lipschitz extension result, namely Lemma \ref{lm:lipschitzextension}. It is stated in a general framework, where $\V,\X$ and $\lambda$ are as in Subsection \ref{infinitesection}.
\begin{lm} \label{lm:dense}
Let $\Z$ be a subset of $\X$ with $\lambda(\mathbb{Z}) = 0$. Then, for any $x\in\X\setminus\Z$, one has $\lambda(\mathbb K_x)=0$, where 
$
\mathbb K_x := \{y \in \samplespace\setminus\mathbb{Z}\ :\ \mathcal H^1([x,y] \cap \mathbb{Z}) > 0\}.
$
\end{lm}
\begin{proof}
Suppose by contradiction that there exists $x\in \X\setminus\Z$ such that $\lambda(\mathbb K_x)>0$. Then there exists a bounded set $\mathbb U$ in $\X$ such that $\lambda(\mathbb K_x\cap \mathbb U)>0$. 
Moreover, by definition of $\mathbb K_x$, for every $y\in\mathbb K_x$ there holds $\mathcal H^1([y,x]\cap\Z)>0$, which together with $\lambda(\mathbb K_x\cap \mathbb U)>0$ implies
\[\begin{aligned}
0<\int_{\mathbb K_x\cap \mathbb U} \mathcal H^1([y,x] \cap \mathbb{Z})\, \lambda(\ud y) &= 
\int_{\mathbb K_x \cap \mathbb U} \|x-y\|\int_0^1\mathds 1_{\mathbb Z}((1-t)y+tx)\,\ud t\,\lambda(\ud y)
\\&\le\sup_{y\in\mathbb U}
\|x-y\|\int_0^1\int_{\mathbb K_x\cap \mathbb U} \ind_{\frac{1}{1-t}(\mathbb{Z}-tx)}(y)\,\lambda(\ud y)\, \ud t 
\end{aligned}\]
where the last inequality is due to Fubini's theorem. But this is a contradiction, since the right-hand side is equal to zero, being the set $\frac{1}{1-t}(\mathbb{Z}-tx)$ of zero $\lambda$-measure for all $t \in (0,1)$.
\end{proof}
\begin{lm}[\bf Lipschitz Extension] \label{lm:lipschitzextension}
Let   $\Z$ be a $\lambda$-null subset of $\X$ and let $\mathbb B(\Z)$ be defined by \eqref{Bset}.  Let $(S, \ud_S)$ be a complete metric space. Let $f : \X\setminus\Z \rightarrow S$. If 
\begin{equation} \label{eq:LipLemmaExt}
d_S(f(x), f(y)) \leq L \|x-y\|\ \qquad \forall\ (x,y) \in \mathbb B(\mathbb Z)
\end{equation}
holds for some $L \geq 0$, then $f$ admits a Lipschitz extension to the whole of $\X$, with the same Lipschitz constant $L$.   
\end{lm} 
\begin{proof}
For every $x\in\X\setminus\Z$, let $\mathbb K_x$ be defined as in Lemma \ref{lm:dense}.
First, we prove that $f$ is Lipschitz-continuous on $\X\setminus\Z$. In fact, fix two points $x,y$ in $ \X\setminus \Z$, and choose a sequence
$\{\xi_n\}_{n\geq 1} \subset \mathbb K_x^c \cap \mathbb K_y^c\cap \mathbb Z^c$ converging to $x$. This choice is possible since Lemma \ref{lm:dense} shows that
$\lambda(\mathbb K_x \cup \mathbb K_y\cup\mathbb Z) = 0$, implying that $\mathbb K_x^c \cap \mathbb K_y^c\cap\mathbb Z^c$ is dense in $\X$. 
Moreover, notice that $(\xi_n,x)\in\mathbb B(\Z)$ and $(\xi_n,y)\in\mathbb B(\Z)$ for every $n$, since $\xi_n$ belongs to both $\mathbb K_x^c\setminus\Z$ and $\mathbb K_y^c\setminus\Z$.
Then, invoke \eqref{eq:LipLemmaExt} to obtain, for every $n \in \naturals$, that 
$\ud_S(f(x), f(y)) \leq \ud_S(f(\xi_n), f(x)) + \ud_S(f(\xi_n), f(y)) \leq L( \|x-\xi_n\| + \|\xi_n -y\|)$.
By taking the limit as $n\to+\infty$, we get the desired Lipschitz property on $\X\setminus\Z$. In conclusion, the existence of a Lipschitz extension with same constant $L$ follows from the standard extension result with a dense domain, being $(S,\ud_S)$ complete.
\end{proof}

\subsection{Scaling estimates of the Poincar\'e constant} \label{sectionpoincare}

In view of the applications of our theorems, the log-concavity condition and its variants are the more natural tools, mostly when considering exponential statistical models (Section \ref{sect:exponential}), exchangeability (Section \ref{sect:nobservations}) and  Bayesian consistency (Section \ref{sect:consistency}). Accordingly,
 we summarize some estimates of the Poincar\'e constant in the following statement, providing some extension of the results in \cite{BBCG}. In particular, in view  of our results about Bayesian consistency in Section \ref{sect:consistency}, in this statement we highlight some scaling properties of the Poincar\'e constant that arise by multiplying $V$ by some large $n\in\N$.

\begin{pro}\label{francesi}
Let $V, U\in C^2(\rd)$ be bounded from below and such that 
$\int_{\rd} e^{-V(\theta)-U(\theta)}\,\ud\theta<+\infty$.   
Let $\mu_n(\ud\theta):=e^{-nV-U}\ud\theta$, for any $n\in\N$. The following statements about the squared Poincar\'e constant of $\mu_n$ hold.
\begin{itemize}
\item[\it (1)]  Suppose that  $\alpha>0$ and $h\in\R$ exist such that $\mathrm{Hess}(V)\ge \alpha I$ on $\rd$ and $\mathrm{Hess}(U)\ge hI$ on $\rd$ in the sense of quadratic forms. Then 
$\mathcal C^2[\mu_n]\le (n\alpha+h)^{-1}$ for every $n>-h/\alpha$.
\item[(2)]  Suppose that there exist $\alpha>0$, $c>0$, $R>0$, $h\in\R$, $\ell\in\R$ such that the following conditions hold: $\mathrm{Hess}(V(\theta))\ge \alpha I$ and $\mathrm{Hess}(U(\theta))\ge h I$ in the sense of quadratic forms whenever $|\theta|\le R$, moreover $\theta\cdot\nabla V(\theta)\ge c|\theta|$ and $\theta\cdot \nabla U(\theta)\ge \ell|\theta|$ whenever $|\theta|\ge R$. Then, for every 
$n>(-h/\alpha)\vee((d_R+1-\ell)/c)$,
 \[
\mathcal C^2[\mu_n]\le \frac{\alpha n+h+(cn+\ell-d_R+ nV_R+U_R)\,C_R}{(\alpha n+h)\,(cn+\ell-1-d_R)}
\]  
where $d_R:=(d-1)/R$, $V_R:=\sup_{B_R}|\nabla V|$, $U_R:=\sup_{B_R}|\nabla U|$ and $C_R$ is an explicit universal constant only depending on $R$. 
\item[(3)]  Suppose that there exist $\alpha>0$, $c_1>0$, $c_2>0$, $R>0$, $h\in\R$ such that the following conditions hold: $\mathrm{Hess}(V(\theta))\ge \alpha I$ and $\mathrm{Hess}(U(\theta))\ge h I$ in the sense of quadratic forms whenever $|\theta|\le R$, and 
\begin{equation}\label{positiveassumption}
|\nabla V(\theta)|^2\ge 2c_1+c_2\, [\Delta V(\theta)+\nabla V(\theta)\cdot\nabla U(\theta)]_+
\end{equation} 
whenever $|\theta|\ge R$. Then, for every $n>(1+1/c_2)\vee(-h/\alpha)$,
\[
\mathcal C^2[\mu_n]\le \frac{\alpha n+h+e^{\omega_R}(c_1n+V_R^*+W_R)} {(\alpha n+h)c_1n}
\]  
where $V_R^*:=\sup_{B_R}|\Delta V|$, $W_R:=\sup_{B_R}|\nabla U||\nabla V|$ and $\omega_R:=\sup_{B_R}V-\inf_{\rd}V$. 
\end{itemize}
\end{pro}

By choosing $n=1$ and $U\equiv 0$ in  Proposition \ref{francesi}, we obtain a direct estimate for the Poincar\'e constant of a given finite measure of the form $\mu(\ud\theta)=e^{-V}\ud\theta$, where $V$ is bounded from below (see also \cite{BBCG}). Besides the Bakry-Emery criterion, which requires the Hessian of $V$ to be bounded away from zero, perturbations of convex functions are included. For instance $V(\theta)=\tfrac12\,|\theta|^2-2\cos|\theta|$ satisfies the assumptions of points {\it (2)} and {\it (3)} with $R=1$.

\begin{proofadfrancesi}
Point {\it (1)} is the  Bakry-Emery criterion, see for instance \cite[Theorem 3.1]{MS}. 

Let us consider point {\it (2)}. We shall apply the arguments from \cite{BBCG}. Let $V_n:=nV+U$. First of all, if $|\theta|\ge R$ we get from the assumptions
\begin{equation}\label{cn}
\theta\cdot V_n(\theta)\ge (cn+\ell)|\theta|.
\end{equation}
Let $W(\theta)$ be a $C^2(\rd)$ function such that $W\ge1$ on $\rd$ and such that $W(\theta)=e^{|\theta|}$ if $|\theta|\ge R$. Let $C_R:=\sup_{B_r}|W|+\sup_{B_R}|\nabla W|+\sup_{B_R}|\Delta W|$. Let us introduce the diffusion operator $\mathcal L_{V_n}[\phi]:=\Delta\phi - \langle\nabla\phi,\nabla V_n\rangle$. A computation shows that if $|\theta|\ge R$ there holds
\[
\mathcal L_{V_n} W=\left(\frac{d-1}{|\theta|}+1-\frac{\theta}{|\theta|}\cdot\nabla V_n(\theta)\right)\,W(\theta).
\]
Let $\tau_n:=cn+\ell-1-d_R$, so that $\tau_n>0$ as soon as $n>(d_R+1-\ell)/c$. 
If $|\theta|\ge R$,  from \eqref{cn} we deduce $\mathcal L_{V_n}W(\theta)\le -\tau_n W(\theta)$. If $|\theta|\le R$, we  estimate 
as
\[\begin{aligned}
\mathcal L_{V_n}W(\theta)&=-\tau_n W(\theta)+\tau_n W(\theta)+\mathcal L_{V_n}W(\theta)\\&\le -\tau_nW(\theta)+\tau_n W(\theta)+|\Delta W(\theta)|+|\nabla W(\theta)||\nabla V_n(\theta)|\\
&\le -\tau_n W(\theta)+C_R(1+\tau_n+nV_R+U_R).
\end{aligned}\]
All in all we have $W(\theta)\ge 1$ and $\mathcal L_{V_n}W(\theta)\le -\tau_n W(\theta)+b_n\chi_{B_R}(\theta)$ for every $\theta\in\rd$, where $b_n:=C_R(1+\tau_n+nV_R+U_R)$. By \cite[Theorem 1.4]{BBCG} we conclude that $\mathcal C^2(\mu_n)\le \tfrac1{\tau_n}(1+b_n\,k_R)$, where $k_R$ is the squared Poincar\'e constant of the measure $e^{-V_n}\mathcal L^d\mr B_R$. Since we have by assumption $\mathrm{Hess} V_n\ge (n\alpha+h) I$ in the sense of quadratic forms on $B_R$, the Bakry-Emery criterion yields $k_R\le (\alpha n+h)^{-1}$ as soon as $n>-h/\alpha$. The conclusion follows.

Let us prove point {\it (3)}. The argument is similar. Let $W(\theta)=\exp\{V(\theta)-\inf_{\mathbb R^d}V\}$, $\theta\in\rd$. Let once more $V_n:=nV+U$ and $\mathcal L_{V_n}[\phi]:=\Delta\phi - \langle\nabla\phi,\nabla V_n\rangle$.
 We have after a direct computation $\Delta W=W|\nabla V|^2+W\Delta V$ and then
\begin{equation}\label{computa}
\mathcal L_{V_n}W=\left((1-n)|\nabla V(\theta)|^2+\Delta V(\theta)-\nabla U(\theta)\cdot\nabla V(\theta)\right)W(\theta), \qquad \theta\in\rd.
\end{equation}
Thanks to  assumption \eqref{positiveassumption} we have $(n-1)|\nabla V(\theta)|^2\ge\Delta V(\theta)-\nabla V(\theta)\cdot\nabla U(\theta)+n c_1$ whenever $|\theta|\ge R$ and $n>1+1/c_2$. Therefore, if $n>1+1/c_2$, from \eqref{computa} we obtain $\mathcal L_{V_n}W(\theta)\le- c_1n\,W(\theta)$ whenever $|\theta|\ge R$. On the other hand, if $|\theta|\le R$ we easily estimate from \eqref{computa} as
\[
\mathcal L_{V_n}W(\theta)=-c_1n\,W(\theta)+c_1n\,W(\theta)+\mathcal L_{V_n}W(\theta)\le -c_1n\,W(\theta)+ e^{\omega_R}(c_1n +V_R^*+W_R).
\]
Then, we have $W(\theta)\ge 1$ and $\mathcal L_{V_n}W(\theta)\le -c_1n\,W(\theta)+\tilde b_n\chi_{B_R}(\theta)$ for every $\theta\in\rd$, where $\tilde b_n:=+ e^{\omega_R}(c_1n+V_R^*+W_R)$. By invoking  \cite[Theorem 1.4]{BBCG} and the Bakry-Emery criterion, the conclusion follows by repeating the same argument in the end of the proof of point {\it (2)}. 
\end{proofadfrancesi}

\subsection*{Acknowledgements}
 The authors acknowledge support from the MIUR-PRIN  project  No 2017TEXA3H. E.M. acknowledges support from the INdAM-GNAMPA 2019 project   {\it ``Trasporto ottimo per dinamiche con interazione''}.
E. D.  received funding from the European Research Council (ERC) under the European Union's Horizon 2020 research and innovation programme under grant agreement No 817257. E. D.  gratefully acknowledges the financial support from the Italian Ministry of Education, University and Research (MIUR), ``Dipartimenti di Eccellenza'' grant 2018-2022.

\begin{flushleft}

EMANUELE DOLERA\\  Dipartimento di  Matematica\\
Universit\`a di Pavia\\ Via Adolfo Ferrata 5, 27100 Pavia, Italy\\
emanuele.dolera@unipv.it

\bigskip

%
  EDOARDO MAININI\\
Dipartimento di Ingegneria meccanica, energetica, gestionale e dei trasporti \\
  Universit\`a  degli studi di Genova\\ Via all'Opera Pia, 15 - 16145 Genova, Italy\\ edoardo.mainini@unige.it

\end{flushleft}

\end{document}